\documentclass[11pt,a4paper,oneside]{article}
\usepackage{amsmath,amsthm}
\usepackage{amsfonts}
\usepackage{graphicx, color}
\usepackage{amssymb,dsfont,hyperref,color}
\usepackage[normalem]{ulem}

\usepackage[left=2.5cm,right=2.5cm,top=3cm,bottom=3cm]{geometry}
\allowdisplaybreaks

\newtheorem{theorem}{Theorem}
\newtheorem{proposition}{Proposition}
\newtheorem{corollary}{Corollary}

\newtheorem{lemma}{Lemma}
\newtheorem{definition}{Definition}
\newtheorem{remark}{Remark}
\newtheorem{assumption}{Assumption}
\numberwithin{equation}{section}

\newcommand{\vertiii}[1]{{\left\vert\kern-0.25ex\left\vert\kern-0.25ex\left\vert #1
    \right\vert\kern-0.25ex\right\vert\kern-0.25ex\right\vert}}
\newcommand{\kv}{\mathbb{E}}
\newcommand{\id}{\mathds{1}}
\newcommand{\xs}{\mathbb{P}}
\renewcommand{\L}{\mathbb{L}}
\newcommand{\ps}{\mathbb{V}\mbox{ar}}

\newcommand{\rhoa}{\mathbb{R}}
\newcommand{\nhoa}{\mathbb{N}^*}

\newcommand{\ic}{\mathbf{i}}

\newcommand{\usup}[1]{\,\underset{#1}{\sup}\,}

\newcommand{\pt}{\partial}

\newcommand{\ch}[1]{\left\|#1\right\|_2}
\newcommand{\vc}{\infty}

\newcommand{\bw}{\ell}

\newcommand{\kep}[2]{\langle #1,#2 \rangle}
\newcommand{\kepa}[1]{\langle #1,1 \rangle}

\newcommand{\intn}{\int_{\mathbb{R}}}

\DeclareMathOperator{\argmin}{argmin\,}

\DeclareMathOperator{\Beta}{Beta}

\newcommand{\mxi}{M^*(\xi)}
\newcommand{\mhxi}{\widehat{M_n^*(\xi)}}

\newcommand{\parinc}[2]{\parbox[c]{#1}{\includegraphics[width=#1]{#2}}}

\newcommand{\Co}{\mathcal{C}}

\def\L{\mathbb{L}}
\def\D{\mathbb{D}}

\def\N{\mathbb{N}}

\def\R{\mathbb{R}}
\def\E{\mathbb{E}}

\def\ind{{\mathchoice {\rm 1\mskip-4mu l} {\rm 1\mskip-4mu l}
{\rm 1\mskip-4.5mu l} {\rm 1\mskip-5mu l}}}
\def\eg{\textit{e.g.} }
\def\ie{\textit{i.e.} }

\def\etal{\textit{et al.} }

\def\tvc#1{\color{black}#1 \color{black}}

\makeatletter
\newenvironment{proofof}[1]{\par
  \pushQED{\qed}%
  \normalfont \topsep6\p@\@plus6\p@\relax
  \trivlist
  \item[\hskip\labelsep
        \bfseries
    Proof of #1\@addpunct{.}]\ignorespaces
}{%
  \popQED\endtrivlist\@endpefalse
}
\makeatother

\author{Van H\`a Hoang\thanks{Faculty of Mathematics and Computer Science, University of Science, Ho Chi Minh City, Vietnam; \texttt{E-mail: hvha@hcmus.edu.vn}} \thanks{Vietnam National University, Ho Chi Minh City, Vietnam},
Thanh Mai Pham Ngoc\thanks{Laboratoire  de Math\'ematiques, UMR 8628, Universit\'e Paris Sud,
91405 Orsay Cedex France; \texttt{E-mail: thanh.pham$\_$ngoc@math.u-psud.fr}},
Vincent Rivoirard\thanks{Ceremade, CNRS, UMR 7534, Universit\'e Paris-Dauphine, PSL Research University, 75016 Paris,
France; E-mail: \texttt{Vincent.Rivoirard@dauphine.fr}},
Viet Chi Tran\thanks{LAMA, Univ Gustave Eiffel, UPEM, Univ Paris Est Creteil, CNRS, F-77447, Marne-la-Vallée, France; E-mail: \texttt{chi.tran@u-pem.fr}}}
\title{Nonparametric estimation of the fragmentation kernel based on a PDE stationary distribution approximation}
\date{\today}

\begin{document}
\maketitle

\begin{abstract}
We consider a stochastic individual-based model in continuous time to describe a size-structured population for cell divisions. This model is motivated by the detection of cellular aging in biology. We address here the problem of nonparametric estimation of the kernel ruling the divisions based on the eigenvalue problem related to the asymptotic behavior in large population. This inverse problem involves a multiplicative deconvolution operator. Using Fourier techniques we derive a nonparametric estimator whose consistency is studied. The main difficulty comes from the non-standard equations connecting the Fourier transforms of the kernel and the parameters of the model. A numerical study is carried out and we pay special attention to the derivation of bandwidths by using resampling.
 \end{abstract}



\section{Introduction}\label{section:intro}

We consider a population model with size structure in continuous time, where individuals are cells which grow continuously and undergo binary divisions after random exponential times at rate $R>0$. When a cell of size $x$ divides, it dies and is replaced by two daughter cells of sizes $\gamma x$ and $(1-\gamma)x$, where $\gamma$ is assumed here to be a random variable drawn according to a distribution with a density with respect to the Lebesgue measure on $[0,1]$: $\Gamma(d\gamma)=h(\gamma)d\gamma$. Between divisions, the sizes of the cells grow with speed $\alpha>0$. Because the two daughter cells are exchangeable, we assume that $h$ is a symmetric density with respect to $\gamma=1/2$. When $h$ is piked at $1/2$, then both daughters tend to have similar sizes, \textit{i.e.} the half of their mother's size. The more $h$ puts weight in the neighbourhood of $0$ and $1$, the more asymmetric the divisions are. They give birth to one small daughter and one big daughter with size close to its mother's.  In this article, we are interested in the estimation of this function $h$ in the case of large populations where the division tree is not observed. We stick to constant rate $R$ and speed $\alpha$ for the sake of simplicity.

{Our biological motivation for studying this model comes from the understanding of aging phenomena associated with cell division. When a cell that contains toxic content divides asymmetrically, the daughter that contains less toxicity can be viewed as younger in the sense that it has a higher fitness.  \tvc{This toxic content could be detrimental cellular components, such as proteins, extrachromosomal rDNA circles or possibly damaged mitochondria, etc.} The concentration of toxic content, that is an increasing function of time during the cell's life, can be seen as a `size'. Asymmetry during the divisions impacts the distributions of toxicity among cells and the shapes of trees describing the successive generations of cells in continuous time. Statistical evidence of asymmetrical divisions and biological consequences are described in Stewart \etal \cite{stewart}. See also Ackermann \etal \cite{ackermann}, Aguilaniu \etal \cite{aguilaniu}, Banks \etal \cite{Banks}, Doumic, Robert and co-authors \cite{robert14,DoumicOlivierRobert} or Moseley \cite{moseley} for discussions on these topics.
}

The population can be described by a stochastic individual-based (particle) model, where the population at time $t$ is represented by a random measure that is the sum of Dirac masses on $\R_+$ weighting the cells' sizes. Stochastic continuous time individual-based models of dividing cell populations with size-structure have made the subject of an abundant literature starting from Athreya and Ney \cite{Athreya70}, Harris \cite{Harris63}, Jagers \cite{Jagers69} etc. until recent years (\eg Bansaye \etal \cite{Bansaye11,Bansaye11a}, Cloez \cite{Cloez11}). Similar models in discrete time should also be mentioned (\eg \cite{Bansaye08,Bansaye13,Bercu09,BitsekiPenda15,Delmas10,Guyon07}). {For the individual-based model considered in this work, exact numerical simulations are possible. This model offers a convenient framework for statistics (see \eg Hoffmann and Olivier \cite{Hoffmann14}, Hoang \cite{Hoang2015,HoangPhD}).} It also connects to the partial differential equations (PDEs) that are usually used in population dynamics (see \cite{BansayeMeleard}).\\
We start from an initial population where the individuals are labelled in an exchangeable way by integers. The population of cells descending from these initial individuals can be seen as the forest of trees rooted in these initial individuals. We use the Ulam-Harris-Neveu notation to label the cells appearing in the population: if the mother has a label $i\in \mathcal{I}=\cup_{\ell\geq 1}\N \times \{0,1\}^{\ell-1}$, then the two daughters have labels $i$0 and $i$1 obtained by concatening the mother's label with integers 0 or 1. \\
The population at time $t$ is described by the point measure:
\begin{equation}\label{zk}Z^K_t=\frac{1}{K}\sum_{i\in V^K_t}\delta_{x_i(t)},
\end{equation}
where $\delta_x$ is the Dirac measure at $x$,  $V^K_t$ is the set of labels of living individuals at time $t$ and $K$ is a renormalizing parameter corresponding to the order of the initial population size. {In what follows}, the parameter $K$ will tend to $+\infty$. The individual with label $i\in V^K_t$ is represented by a Dirac mass weighting the size $x_i(t)$ of this individual at time $t$. \tvc{Notice that if we follow a lineage starting from a cell at time $0$ and choosing a daughter at random at each division, we recover an ergodic process with multiplicative jumps (see \cite[Section 2.2.2]{HoangPhD}) implying that the cell sizes are controlled over time, whatever the values of $\alpha$ and $R$.}\\

When the complete division forest is observed, we can associate to each division an independent random variable with distribution $h$: if $T_i$ is the division time of the cell $i$, then, we define $\Gamma_i=x_{i0}(T_i)/x_i(T_{i-})$, {where $x_i(T_{i-})=\lim_{t\to T_i, t<T_i} x_i(t)$}. Estimating the function $h$ from such a sample has been considered in \cite{HoangPhD, Hoang2015}. Here, we focus on the situation when the division tree is not completely observed. Following ideas from  Doumic \etal \cite{DoumicHoffmann12, Doumic12,Doumic17} or Bourgeron \cite{Bourgeron14} whose aim was to recover the division rate $R$ when the latter depends on the size, our strategy is to consider the PDE approximating the evolution of the measure-valued process $(Z^K_t)_{ t\geq 0}$ when $K$ is large. The long-time behavior of the solution of this PDE can be studied thanks to an eigenvalue problem. This yields a stationary distribution $N(x)dx$ from which we can assume that we have drawn a sample of $n$ i.i.d. random variables $X_1,\dots ,X_n$. \textcolor{black}{Along this paper, we do not take into account the approximation errors related to the asymptotic setting $K\rightarrow +\infty$ (the fluctuations associated to the convergence of $Z^K$ could be established following \cite{Bansaye11a,Tran14}) nor the approximation by the stationary solution. The latter assumption is discussed in the next section.} The function $h$ is then solution to an intricate inverse problem involving a multiplicative convolution operator. We use deconvolution techniques inspired by those used by Comte and Lacour \cite{Comte2011, ComteLacour}, Comte \etal \cite{comte2014deconvolution}, Neumann \cite{Neumann97} to construct and study a kernel estimator of $h$. Changing variables and taking Fourier transforms lead us to an equation where the regularities of the different terms are strongly related to the regularity of the unknown function $h$ to be estimated. \tvc{In the setting of a large population close to its stationary state, we define an original estimator of $h$.} The consistency of the estimator is studied, and simulations are performed. In particular, we discuss and illustrate numerically the bandwidth selection rules for the kernel estimator.

The paper is organized as follows. Section \ref{section:microscopique} describes the miscroscopic model. Section  \ref{section:estimation-h} tackles the problem of estimating the division kernel $h$. Section \ref{sec:simus} presents the numerical performances of our estimation procedure. Eventually, all the proofs are gathered in the Appendix.  \\

\paragraph{Notation:}
We denote by $\mathcal{M}_F(\R_+)$ the space of finite measures on $\R_+$ endowed with the weak convergence topology. For $\mu \in \mathcal{M}_F(\R_+)$ and for $f\in \Co_b(\R_+,\R)$ a bounded continuous real function on $\R_+$, $\langle \mu,f\rangle= \int_{\R_+} f d\mu$ is the integral of $f$ with respect to $\mu$. We denote by $\D(\R_+,\mathcal{M}_F(\R_+))$ the space of c\`adl\`ag functions from $\R_+$ to $\mathcal{M}_F(\R_+)$ embedded with the Skorokhod topology (e.g. \cite{Billingsley68}).\\
The set of integrable (resp. bounded) non-negative functions with respect to the Lebesgue measure on $\R_+$ is denoted by $\L^1(\R_+,\R_+)$ (resp. $\L^\infty(\R_+,\R_+)$).\\
The Fourier transform of any integrable function $f$ is defined by
$$f^*(\xi)=\int_{-\infty}^{+\infty} f(x) e^{\ic x\xi}dx,\quad \xi\in\R.$$

\section{Microscopic model\label{section:microscopique}}

Let $(\Omega, \mathcal{F},\mathbb{P})$ be a probability space, let $(Z^K_0)_{K\in \N^*}$ be a sequence of random point measures on $\R_+$ of the form \eqref{zk} that converges to $\xi_0\in \mathcal{M}_F(\R_+)$ in distribution and for the weak convergence topology on $\mathcal{M}_F(\R_+)$. We also assume that
\begin{equation}\label{eq:non-explosion-condition}
 \usup{K\in\N^*}\kv(\kepa{Z^K_0}^2)<+\vc.
\end{equation}
For each $K\in \N^*$ and initial condition $Z^K_0$ as above, we can represent the measure-valued process $(Z^K_t)_{t\geq 0}$ as the unique solution of a stochastic differential equation (SDE) driven by a Poisson point measure that satisfies the following martingale problem.

\begin{proposition}\label{prop:renorm1}
For a given $K\in \N^*$ and a test function $f: (x,s)\mapsto f(x,s) = f_s(x)$ $\in \mathcal{C}^{1,1}_b(\rhoa_+\times \rhoa_+,\rhoa)$, the process $(Z^K_t)_{t\geq 0}$ satisfies:
\begin{align}
\langle Z^K_t,f_t\rangle =  \langle Z_0^K, f_0\rangle & + \int_0^t\int_{\R_+} \Big(\pt_s f_s(x) + \alpha \pt_x f_s(x) \notag \\
& + R \int_0^1 \big(f_s(\gamma x)+f_s((1-\gamma)x)-f_s(x)\big)h(\gamma)d\gamma \Big)
 Z_s^K(dx)ds + M^{K,f}_t,\label{eq:martingale}
\end{align}where $(M^{K,f}_t)_{t\geq 0}$ is a square integrable martingale started at 0 with bracket:
\begin{equation}
\langle M^{K,f}\rangle_t= \frac{1}{K} \int_0^t\int_{\R_+}\int_0^1  R\big(f_s(\gamma x)+f_s((1-\gamma)x)-f_s(x)\big)^2 h(\gamma)d\gamma
 Z_s^K(dx)ds.
\end{equation}
\end{proposition}

{
The above equations in Proposition \ref{prop:renorm1} show the evolution of a microscopic random system of particles. The drift coefficient (r.h.s. in the first line of \eqref{eq:martingale}) indicates that each particle grows with speed $\alpha$. When a particle of size $x$ divides, it is replaced by two daughters of sizes $\gamma x$ and $(1-\gamma)x$, where $\gamma$ is drawn in the probability distribution with density $h$: this corresponds to the second line of \eqref{eq:martingale}. When the function $h$ is piked at $1/2$, the daughter cells have almost equal sizes at division, whereas when $h$ has large variance, it is likely to have an asymmetrical division.}\\
The detailed construction of the SDE satisfied by $(Z^K_t)_{t\geq 0}$ is given in Appendix \ref{app:renorm}, as well as a sketch of proofs for the results of this section. The martingale property and quadratic variation are direct consequences of stochastic calculus with the SDE. {The variance of the martingale part $M^{K,f}$ is of order $1/K$ and we heuristically expect a deterministic limit when $K\to +\vc$.} The following theorem states the limit of $(Z^K)_{K\in\nhoa}$ when $K\to +\vc$.
\begin{theorem}\label{th:large-population-limit}
If $(Z^K_0)_{K\in \N^*}$ converges in distribution to the deterministic measure
$\xi_0 \in\mathcal{M}_F(\rhoa_+)$ as $K\to +\vc$ then $(Z^K)_{K\in \nhoa}$ converges in distribution in $\mathbb{D}\left(\R_+,\mathcal{M}_F(\rhoa_+) \right)$
 as $K\to +\vc$ to the unique solution $\xi \in \mathcal{C}\left(\R_+,\mathcal{M}_F(\rhoa_+) \right)$ of
\begin{align}
 \kep{\xi_t}{f_t} = \kep{\xi_0}{f_0} & + \int_0^t\int_{\rhoa_+} \Big(\pt_s f_s(x) + \alpha \pt_x f_s(x) \notag\\
 & + R \int_0^1 \big(f_s(\gamma x)+f_s((1-\gamma)x)-f_s(x)\big)\, h(\gamma) d\gamma\Big)\xi_s(dx)ds, \label{eq:limit-PDE}
\end{align}
where $f_t(x)\in\mathcal{C}^{1,1}_b(\rhoa_+\times\rhoa_+,\rhoa)$ is a test function.
\end{theorem}

When the limiting initial condition $\xi_0$ admits a smooth density with respect to the Lebesgue measure, the following proposition allows us to connect the measure-valued processes with the growth-fragmentation integro-differential equations usually introduced for cell divisions, \eg \cite{Perthame07, Doumic09}.
\begin{proposition}\label{prop:PDE}
If $\xi_0$ has a density $n_0\in \Co^{1}_b(\R_+,\R_+)$ with respect to the Lebesgue measure on $\R_+$, then $\forall t\in \rhoa_+$, $\xi_t(dx)$ admits a density $n(t,x)$ that is the unique solution of the PDE:
\begin{equation}\label{eq:growth-fragmentation}
\pt_t n(t,x) + \alpha\pt_x n(t,x) + Rn(t,x) = 2R\int_0^\vc n(t,y)h\left(\frac xy\right) \frac{dy}{y},
\end{equation}
where $h(x/y) = 0$ if $y<x$ (since $h$ is supported on $[0,1]$).
\end{proposition}
\tvc{See Appendix~\ref{app:renorm} and \cite[Proposition 3.2.10]{HoangPhD} for the proof of this result.}
Besides the drift $\alpha$ associated with the continuous growth of individuals in time, the PDE \eqref{eq:growth-fragmentation} involves  the death term $Rn(t,x)$ and the birth term $2R\int_0^\vc n(t,y)h( x/y) dy/y$. These terms highlight  that a particle disappearing at $x$ is replaced by two particles whose sizes are fractions of $x$. The division is ruled by the density function $h$ and as explained in the introduction, we are interested in the estimation of this density function.\\
The long time behaviour of the solution of PDE \eqref{eq:growth-fragmentation} is well-known and presented in the following proposition. In this work, we shall base our statistical estimation of $h$ on the long time limit of the PDE. Notice that by change of variable in the integral, the right hand side of Equation \eqref{eq:growth-fragmentation} can also be rewritten as: $2R\int_0^1 n(t,x/u)h(u)\, du/u.$ We observe that a convenient assumption on the density $h$ is the following:
\begin{equation}
\int_0^1 h(u)\frac{du}{u}<+\infty.\label{hyp:temporaire}
\end{equation}
In this paper, a stronger assumption will be needed to obtain the consistency of our estimators.


\begin{proposition}\label{prop:renorm2}
Assume \eqref{hyp:temporaire}. Then, there exists a unique probability density $N\in \L^1(\R_+,\R_+)$ solving the following system:
\begin{equation}\label{eq:eigenvalue}
\begin{cases}
 &\alpha \pt_x N(x) + 2R \, N(x) = 2R\int_0^\vc N(y)h\left(\dfrac xy\right) \dfrac{dy}{y} ,\quad x\ge 0, \\
 &N(0) = 0,\quad \int N(x)dx =1,\quad N(x) \ge 0.
\end{cases}
\end{equation}
With $\rho = \| n_0 \|_{1}=\int_0^\vc n_0(u) du$ (where $n_0$ has been introduced in Prop. \ref{prop:PDE}), we have:
\begin{equation}\label{eq:tpslong}
 \int_0^\vc |n(t,x)e^{- R t} - \rho N(x)| dx \leq e^{-Rt} \big(\|g_0\|_1+\frac{6R}{\alpha}\|G_0\|_1\big),
\end{equation}where
$g_0(x)=n_0(x) - \rho N(x)$, and $G_0(x)=\int_0^x g_0(y)dy.$
\end{proposition}
Proposition \ref{prop:renorm2} shows that the renormalized population density $\rho^{-1} n(t,x)e^{- R t}$ converges exponentially fast, when the time $t$ tends to infinity, to a stationary density $N(x)$ that is obtained by solving an eigenvalue problem. The proof of Proposition \ref{prop:renorm2} is given in Appendix A. Notice that we do not have such a strong result if the division rate is not a constant. Another remark is that the right hand side of \eqref{eq:eigenvalue} is a multiplicative convolution between $y\mapsto N(y)/y$ and $h$. Multiplicative convolutions appear naturally in problems where independent random variables are multiplied (here the size of the cell undergoing division and the random variable of density $h$ ruling how the cell breaks into two daughters). Estimating $h$ can thus be seen as performing a multiplicative deconvolution. We explain in the next section the building of our statistical estimation procedure based on the results of this proposition.

\section{Estimation of the division kernel}\label{section:estimation-h}
\subsection{Estimation procedure and assumptions}\label{section:construction-estimator}
\subsubsection{Principle}
We consider the problem of estimating the density $h$ in the case of incomplete data of divisions. As explained previously, we shall construct an estimator of $h$ based on the stationary size distribution which results from the study of the large population limit $n(t,x)$. The long time behavior provides us an observation scheme for the estimation of the density $h$ in the statistical approach: since $e^{-R t}n(t,x)$ converges exponentially fast to $N(x)$ (up to a constant) as $t$ increases by Proposition~\ref{prop:renorm2}, when we pick $n$ cells randomly in the population at a large time $t$, we can assume that we have $n$ i.i.d observations $X_1, X_2, \ldots, X_n$ with distribution $N(x)dx$. We estimate $h$ from the data $X_1, \ldots, X_n$ and Equation \eqref{eq:eigenvalue}. {This experimental scheme has also been used in \cite{Doumic12} and \cite{Bourgeron14}.}

\par  {Starting with Equation \eqref{eq:eigenvalue}, the multiplicative convolution $\int_0^\vc N(y)h\left(\frac xy\right) \frac{dy}{y}$ leads to more intricate technical problems than for the classical additive convolution. }
 So, we apply a logarithmic change of variables to transform the multiplicative convolution in the right hand side of \eqref{eq:eigenvalue} into an additive one. Then, we classically apply the Fourier transform and work with products of functions in the Fourier domain. {We end up with a deconvolution problem which is more involved and quite different when compared with classical deconvolution problems (see Remark \ref{classic})}.  


\noindent  Let us now describe our estimation procedure in details.
By using the change of variable $x = e^u$ for $x>0$ and $u\in\rhoa$, we introduce the functions $$g(u) = e^uh(e^u),$$ and
\[
M(u) = e^uN(e^u), \quad D(u) = \partial_u\big(u\mapsto N(e^u)\big) = e^u N'(e^u).
\]
Equation \eqref{eq:eigenvalue} becomes
\begin{equation}\label{eq:eigenvalue-change-variable}
\alpha D(u) + 2R\, M(u) = 2R\big(M\star g)(u),
\end{equation}
where $\star$ denotes the standard convolution product, so
$$\big(M\star g)(u)=\int M(u-x)g(x)dx,\quad u\in\R.$$
We have $h(\gamma) = \gamma^{-1}g\big(\log(\gamma)\big)$ for $\gamma \in (0,1)$. Then, the estimator of $h$ will be obtained from the estimator of $g$ once we have obtained estimators for unknown functions $M$ and $D$.\\

\subsubsection{Assumptions on h}
 First, assumptions on the density $h$ are needed. Of course, since $h$ is the density of a symmetric probability distribution on $[0,1]$, it satisfies $\int h(x)dx =1$ and $\int xh(x)dx = 1/2$. For the proofs, we will also need the following condition.
\begin{assumption}\label{assump:moment-h}
The function $h$ is of class $\Co^\beta$ on $[0,1]$, for some $\beta> 3$: the function $h$ is $[\beta]$ times differentiable (where $[\beta]$ is the largest integer smaller than $\beta$) and the derivative of order $[\beta]$ is $\beta-[\beta]$ H\"older continuous.\\
Moreover, we assume that there exists a positive integer $\nu_0\geq 2$ such that for all $k\in \{0,\dots,\nu_0\}$, $h^{(k)}(0)=0$.
\end{assumption}
Under Assumption \ref{assump:moment-h}, $h$ can take positive values only on $(0,1)$, and the function $g$ introduced previously is supported on $\R_-$.
\begin{remark}
Assumption \ref{assump:moment-h} implies \eqref{hyp:temporaire}. For $t\in (0,1)$, by Taylor's formula, there exists indeed  $\theta\in (0,1)$ such that:
$$0\leq \frac{h(t)}{t}= \sum_{k=\nu_0+1}^{[\beta]-1}\frac{1}{k !} h^{(k)}(0) t^{k-1} + \frac{h^{[\beta]}(\theta t)}{[\beta]!} t^{[\beta]-1},$$which is integrable in the neighborhood of 0 (the sum in the right hand side being 0 if $\nu_0+1>[\beta]-1$.
\end{remark}
This remark shows that, under Assumption \ref{assump:moment-h}, the results of Proposition~\ref{prop:renorm2} are hence available to justify our approximation to start with a sample of i.i.d. random variables with density $N(x)$. We also have the following proposition that will be essential to show consistency and derive rates of convergence (the proof is in Appendix \ref{app:lemmas}):
\begin{proposition}\label{lem:regM}
Under Assumption \ref{assump:moment-h}:\\
(i) the first eigenvector $N$ of the eigenproblem \eqref{eq:eigenvalue} satisfies
\begin{equation}\label{eq:moment-condition}
\int_0^{+\vc} x^{-\nu}N(x)dx < +\vc\quad\text{ for } \nu\in \{1,\ldots,(\nu_0+2)\wedge ([\beta]+1)\}.
\end{equation}
(ii) $M$ is of class $\Co^{[\beta]}$ and its Fourier transform $M^*$ satisfies:
$$\limsup_{|\xi|\rightarrow +\infty}\left\{|\xi|^{[\beta]\wedge (\nu_0+3)}\times|M^*(\xi)|\right\}<+\infty.$$
(iii) The extension of $M^*(\xi)$ to the complex half-plane $\{\xi\in \mathbb{C} : \Im(\xi)<1\}$, $\xi \mapsto M^*(\xi)=\int_{\R} e^{\ic x\xi }M(x)dx$, is holomorphic and thus, $M^*$ admits only isolated zeros on this half-plane. Moreover, $M^*$ does not admit zeros on the real line.
\end{proposition}

The point (i) is crucial for proving the consistency. This proof relies on the use of the Rosenthal inequality (see Eq. \eqref{eq:bound-Zj}). This explains why we need $\nu\geq 4$ and hence $\nu_0\geq 2$ and $\beta>3$ in Assumption \ref{assump:moment-h}. The point (ii) establishes strong connections between the regularities of functions involved in \eqref{eq:eigenvalue}. Paradoxically, the more regular $h$ is, the faster $M^*$ converges to 0 at infinity, which may lead to some difficulties in view of the subsequent \eqref{eq:FT-g}. Fortunately, point (iii) shows that $M^*(\xi)$ does not vanish on the real line.\\

\subsubsection{Fourier transformation}
Notice that $g$ is square integrable since we have
\[
\intn g^2(u) du= \intn e^{2u}h^2(e^u) du = \int_0^{\infty} x h^2(x) dx= \int_0^1 x h^2(x) dx < +\infty.
\]
We can thus take the Fourier transform of both sides of equation \eqref{eq:eigenvalue-change-variable}. We obtain
$$
\alpha D^*(\xi) + 2 R\, M^*(\xi) = 2R M^*(\xi)\times g^*(\xi).
$$
Therefore, under Assumption \ref{assump:moment-h}, the Fourier transform of $g$ is obtained via the formula
\begin{equation}\label{eq:FT-g}
g^*(\xi) = \frac{\alpha D^*(\xi)}{2RM^*(\xi)} + 1,\quad \xi\in\rhoa.
\end{equation}
\begin{remark}\label{classic}
Note that Equation (\ref{eq:FT-g}) is not standard in classical inverse problems.
{Actually, classical deconvolution problems with independent noise (see \eqref{deconv-classique} below) can be transformed so that, in the Fourier domain, they can be written similarly to \eqref{eq:FT-g} and assumptions are made on the asymptotic behavior of the Fourier transform of the noise density, which is the analog of $M^*$ (see \cite{ComteLacour} for instance). Such assumptions are not possible here since the smoothness of $g$ is related to the smoothness of $M$ via \eqref{eq:eigenvalue-change-variable}. Assumptions on $M^*(\xi)$ when $|\xi|\to+\infty$ would break these strong relationships between $g$ and $M$. But these connections between $g$ and $M$ allow us to deduce the asymptotic behavior of $M^*(\xi)$ in Proposition~\ref{lem:regM} and the issues are circumvented.}
\end{remark}

%
%

\subsubsection{Estimators of $g$ and $h$}\label{section:estimators-g-h}
Given the sample of i.i.d random variables $X_1, \ldots, X_n$ with density function $x\mapsto N(x)$, we can consider the random variables $U_1, \ldots, U_n$ defined as $U_i = \log(X_i)$. These random variables are i.i.d of density function $u\mapsto M(u)=e^uN(e^u)$. In view of \eqref{eq:FT-g}, the purpose is  first to propose an estimator for $g^*$ and then to apply the inverse Fourier transform to obtain an estimator of $g$. Our procedure will be naturally based on $\widehat{M_n^*(\xi)}$ and $\widehat{D_n^*(\xi)}$, estimators of $M^*(\xi)$ and $D^*(\xi)$ respectively, and defined by
\begin{align}
\widehat{M_n^*(\xi)} &= \frac{1}{n}\sum_{j=1}^n e^{\ic\xi U_j},\label{eq:estimator-FT-M} \\
\widehat{D_n^*(\xi)} &= (-\ic\xi)\frac{1}{n}\sum_{j=1}^n e^{(\ic\xi-1)U_j}. \label{eq:estimator-FT-D}
\end{align}
Obviously, we have that $\widehat{M_n^*(\xi)}$ and $\widehat{D_n^*(\xi)}$ are unbiased estimators of $M^*(\xi)= \kv\big[e^{\ic\xi U_1} \big]$ and $D^*(\xi)= (-\ic\xi)\kv\left[e^{(\ic\xi - 1)U_1} \right]$ respectively. \\


As usual in the nonparametric setting, the estimate of $g$ will be obtained by regularization technics. For density estimation, convoluting by an appropriate rescaled kernel is a natural methodology. Convolution is expressed by products in the Fourier domain. \tvc{Along the paper, we use the sinus cardinal kernel} defined by  $K(x)= \frac{\sin( x)}{ \pi x }$ for which $K^*({t})=\mathds{1}_{[-1,1]}(t)$. For $\bw >0$, define \tvc{the rescaled kernel} \[K_\bw(\cdot): = \frac{1}{\bw}K\left(\frac{\cdot}{\bw}\right).\]
%
\begin{definition}
Given $\ell>0$, the estimate $\hat{g}_{n,\bw}$ of $g$ is defined through its Fourier transform:
\begin{equation}\label{eq:estigell}
\hat{g}_{n,\bw}^*(\xi) = K_\ell^*(\xi) \times  \left(  \frac {\alpha \widehat{D_n^*(\xi)}}{2R }\frac{\id_{\Omega_n(\xi)}}{\widehat{M_n^*(\xi)}} +  1   \right),
\end{equation}
where $\Omega_n(\xi) = \big\{|\widehat{M_n^*(\xi)}|\ge n^{-1/2} \big\}$ and  $\frac{\id_{\Omega_n(\xi)}}{\widehat{M_n^*(\xi)}}$is the truncated estimator of  $\frac{1}{\widehat{M_n^*(\xi)}}$:
\begin{equation}\label{eq:truncated-estimator-1/M*}
\dfrac{\id_{\Omega_n(\xi)}}{\widehat{M_n^*(\xi)}} = \begin{cases}
\dfrac{1}{\widehat{M_n^*(\xi)}}, &\text{ if } |\widehat{M_n^*(\xi)}|\ge n^{-1/2}, \\
0, &\text{otherwise}.
\end{cases}
\end{equation}
\end{definition}
The technique used to obtain (\ref{eq:estigell}) is similar to inverse truncation filtering (see \cite{Bertero} or \cite{Byrne}). Truncation is necessary to avoid explosion when $|\widehat{M_n^*(\xi)} |$ is close to $0$. Finally, taking the inverse Fourier transform of $\hat{g}_{n,\ell}^*$, we obtain the estimator of $g$.
\begin{definition}The estimator of $g$ is
\begin{equation}\label{eq:estimator-g}
\hat{g}_{n,\bw}(u) = \frac{1}{2\pi }\int_\rhoa \hat{g}_{n,\bw}^*(\xi)e^{-\ic u\xi}d\xi,\quad u\in\R_{-}.
\end{equation}
The estimator of the division kernel $h$ is deduced from $\hat{g}_{n,\bw}$:
\begin{equation}\label{eq:estimator-h}
\hat{h}_{n,\bw}(\gamma) = \gamma^{-1}\hat{g}_{n,\bw}\big(\log(\gamma)\big), \quad \gamma\in (0,1).
\end{equation}
\end{definition}

The main difficulty lies in the choice of $\ell$. This problem is dealt with subsequently. Deconvolution estimators have been studied in Comte and Lacour \cite{Comte2011, ComteLacour}, Comte \etal \cite{comte2014deconvolution}, Neumann \cite{Neumann97}. However, the difference and the difficulty in our problem come from the fact that the regularities of $g$ and $h$ are closely related to the functions $M$ and $D$ that solve the eigenvalue problem \eqref{eq:eigenvalue}, in particular through Equation \eqref{eq:FT-g}. This complicates the study of the rates of convergence. The next section studies the  quadratic risk of $\hat{g}_{n,\bw}$ and $\hat{h}_{n,\bw}$.
\subsection{Study of the quadratic risk}
\subsubsection{Relations between the risks of the estimators of $h$ and $g$}\label{sec:Relations}
The first goal is to connect the $\mathbb{L}^2$-risk of $\hat{h}_{n,\bw}$ and the $\mathbb{L}^2$-risk of $\hat{g}_{n,\ell}$. Using a randomized estimator, we can show the following result.
\begin{proposition}\label{prop:htilde}
For a Bernoulli random variable $\tau$ with parameter $1/2$ independent of $X_1,\dots,X_n$, let us define the randomized estimator
$$\check{g}_{n,\bw}(u)=\tau \hat{g}_{n,\bw}(u)+(1 -\tau)\tilde{g}_{n,\bw}(u), \mbox{ where }\tilde{g}_{n,\bw}(u)=e^u \hat{h}_{n,\bw}(1-e^u). $$
We have
\begin{equation}
\E\big[\| \hat{h}_{n,\ell}-h\|^2_2\big]=2 \E\big[\| \check{g}_{n,\bw} - g\|^2_2\big]= \E\Big[\int_{\R_-} e^{-u} \big(\hat{g}_{n,\bw}(u) -g(u) \big)^2 du \Big].\label{etape6}
\end{equation}
\end{proposition}
The last equality in \eqref{etape6} shows that if we want to control the quadratic risk of $\hat{h}_{n,\bw}$ with respect to the Lebesgue measure, tight controls on the loss of $\hat{g}_{n,\bw}$ at $-\infty$ are needed. But, since $h$, as defined in our biological problem, is a symmetric function (as the daughter cells obtained after a division are exchangeable), it is natural to consider
\begin{equation}\label{eq:symm}
\hat h_{n,\bw}^{sym}(x)=\frac{1}{2}\big(\hat h_{n,\bw}(x)+\hat h_{n,\bw}(1-x)\big),
\end{equation}
whose quadratic risk is controlled by the quadratic risk of  $\hat{g}_{n,\bw}$ except at boundaries of the interval $[0,1]$, as proved by the next proposition.
\begin{proposition}\label{prop:h-hat}
Setting $m(x)=x(1-x)$, we have that
\begin{equation}
\int_0^1 \big(\hat h_{n,\bw}^{sym}(x) - h(x)\big)^2m(x)dx \leq \ch{\hat{g}_{n,\bw} - g}^2.
\end{equation}
\end{proposition}
Propositions \ref{prop:htilde}  and \ref{prop:h-hat} are proved in Appendix \ref{app:risk-h}.
The previous result does not provide any control on boundaries of the interval $[0,1]$ but the consistency of $\hat{g}_{n,\bw}$ will establish the consistency of $\hat h_{n,\bw}^{sym}$ on every compact set of $(0,1)$. The study of the consistency of $\hat{g}_{n,\bw}$ is the goal of the next section.
\subsubsection{Consistency of the estimator of $g$ for the quadratic-risk}\label{section:consistency-estimator-g}
This section is devoted to the theoretical study of the estimate $\hat g_{n,\ell}$. More precisely, we establish the $\L^2$-consistency of $\hat g_{n,\ell}$ under a suitable choice of the bandwidth $\ell$.

We first study the bias-variance decomposition of the $\mathbb{L}^2$-risk of $\hat{g}_{n,\bw}$. Recall that from Proposition~\ref{lem:regM}(iii), we have that under Assumption \ref{assump:moment-h}, $|M^*(\xi)|$ is strictly positive on every compact set of the real line $\xi\in [-A,A]$, $A>0$, and thus lower bounded by a positive constant on each of these intervals (that depends on $A$).

\begin{theorem}\label{prop:rate-of-convergence}
Under Assumption \ref{assump:moment-h}, there exists a positive constant $C < +\vc$ such that
\begin{equation}\label{eq:rate-of-convergence}
\kv\left[\ch{\hat{g}_{n,\bw} - g}^2 \right] \leq \ch{K_\bw\star g - g}^2 + \frac{C}{n}S(\bw),
\end{equation}
where
\[
S(\bw) = \Big\| \frac{K^*_\bw(\xi)\xi}{M^*(\xi)}\Big\|_2^2 + \Big\| \frac{K^*_\bw(\xi)}{M^*(\xi)}\Big\|_2^2.
\]
\end{theorem}

Then the following corollary gives the $\L^2$-consistency of the estimator $\hat g_{n,\ell}$.

\begin{corollary}\label{TheoremConsistance}
We suppose that Assumption \ref{assump:moment-h} is satisfied and the kernel bandwidth $\ell=\ell(n)$  satisfies $\lim\limits_{n \to +\infty} \ell =0 $. Provided that
\begin{equation}\label{eq:limit-variance-term-M*}
\lim \limits_{n \to +\infty} \frac{1}{n} \left( \Big\| \frac{K^*_\bw(\xi)\xi}{M^*(\xi)}\Big\|_2^2 + \Big\| \frac{K^*_\bw(\xi)}{M^*(\xi)}\Big\|_2^2\right) =  0,
\end{equation}
we have
\begin{equation}\label{eq:consistance}
 \lim \limits_{n \to +\infty} \mathbb{E} \left [\| \hat g_{n,\ell} - g \|^2_2 \right] = 0.
\end{equation}
\end{corollary}
The proof of Corollary \ref{TheoremConsistance} is straightforward. Indeed, due to the well-known results on kernel density, we have $\lim \limits_{n \to +\infty}  \ch{K_\bw\star g - g}^2 = 0$ and under the assumptions of the corollary we have for the variance term $\lim_{n \to +\infty} n^{-1} S(\ell) = 0$. Thus we get the result \eqref{eq:consistance}.
The proof of Theorem \ref{prop:rate-of-convergence} is given in Appendix \ref{app:main-theorems}.  Note that under Assumption \ref{assump:moment-h}, we have by Proposition \ref{lem:regM} that $|M^*(\xi)|=O(|\xi|^{-([\beta]\wedge (\nu_0+3))})$ when $|\xi|\to+\infty$. If we have $|M^*(\xi)|\sim C |\xi|^{-([\beta]\wedge (\nu_0+3))}$, for a constant $C>0$, a bandwidth $\ell$ can be easily derived. Indeed,
$$K_\bw^*(\xi)=K^*(\bw\xi)=\mathds{1}_{[-\bw^{-1},\bw^{-1}]}(\xi)$$ and
$$
\Big\| \frac{K^*_\bw(\xi)\xi}{M^*(\xi)}\Big\|_2^2=\int_{-\bw^{-1}}^{\bw^{-1}}\frac{\xi^2}{|M^*(\xi)|^2}d\xi=O(\bw^{-(3+2([\beta]\wedge (\nu_0+3)))}),$$
and then, Assumption \eqref{eq:limit-variance-term-M*} is satisfied if
\[
\ell^{-1}=o\left(n^{\frac{1}{3+2([\beta]\wedge (\nu_0+3))}}\right).
\]
We obtain convergences rates for the quadratic risk of $\hat g_{n,\ell} $ under additional smoothness properties for the density $g$. For this purpose, we introduce Sobolev spaces defined as follows.
\begin{definition}\label{assump:sobolev-g}
We consider Sobolev spaces $S(\beta, L)$ defined as the class of integrable functions $ f : \R \rightarrow \R $ satisfying
\[
\int |f^*(t)|^2(1+t^2)^{\beta}dt \leq L^2.
\]
\end{definition}
We then obtain the following result.
\begin{proposition}\label{vitesse:g} If $g \in S(\beta,L)$ and $|M^*(\xi)|\sim C |\xi|^{-([\beta]\wedge (\nu_0+3))}$, for a constant $C>0$,  then we have
\[
\kv\left[\ch{\hat{g}_{n,\bw} - g}^2 \right] = O\left(n^{-\frac{2\beta}{2\beta +2([\beta]\wedge (\nu_0+3)) + 3}}\right).
\]
\end{proposition}
The rate of convergence of Proposition \ref{vitesse:g} is the usual rate of convergence for ill-posed inverse problems involving a derivative and an ordinary smooth noise with a polynomial decay of  order $[\beta]\wedge (\nu_0+3)$. This result shows good theoretical performances of our procedure.
\section{Numerical simulations}\label{sec:simus}
\subsection{\tvc{Influence of the preliminary estimators $\widehat{M_n^*}$ and $\widehat{D_n^*}$ on the performances of $\hat h_{n,\ell}$ and $\hat g_{n,\ell}$}}

In this section, we study the numerical performances of our estimation procedure.

In the literature (\eg \cite{stewart,wang}), it is possible to obtain real datasets of sample size $n=30,000$ or even a larger: in \cite{stewart}, the authors followed divisions of E. coli and obtained a complete record of measurements of 35,049 cells, in \cite{wang}, the authors introduced their experimental procedures and techniques that allow to obtain a real dataset of $10^7$ cells. Therefore, the simulations presented here are performed on simulated samples of sizes $n$ varying from 1,000 to 30,000.

We consider the density of the $\Beta(2,2)$-distribution and the density of the truncated normal distribution on $[0,1]$ with mean $1/2$ and variance $0.25^2$, respectively denoted $h_1$ and $h_2$. The density $h_1$ is proportional to $x(1-x)\id_{[0,1]}(x)$ and $h_2$ has the following form:
\[
h_2(x) = \frac{\phi\left(\frac{x - \mu}{\sigma} \right)}{\sigma \left(\Phi\Big(\frac{1-\mu}{\sigma} \Big) -  \Phi\Big(\frac{-\mu}{\sigma} \Big)\right)},\quad x\in [0,1],
\]
where $\mu = 0.5$, $\sigma = 0.25$ and $\phi(\cdot)$ and $\Phi(\cdot)$ are respectively the density and the cdf  of the standard normal distribution. Furthermore, for all simulations we take $\alpha = 0.7$ and  $R = 1$. Figures \ref{fig:h1N1} and \ref{fig:h2N2} show $h_1$, $h_2$ and their corresponding stationary densities $N_1$, $N_2$. The stationary densities are obtained by solving numerically the PDE \eqref{eq:growth-fragmentation} using the method presented in Doumic {\it et al.} \cite{Doumic09}.

\begin{figure}[ht!]
\centering
\includegraphics[scale=0.4]{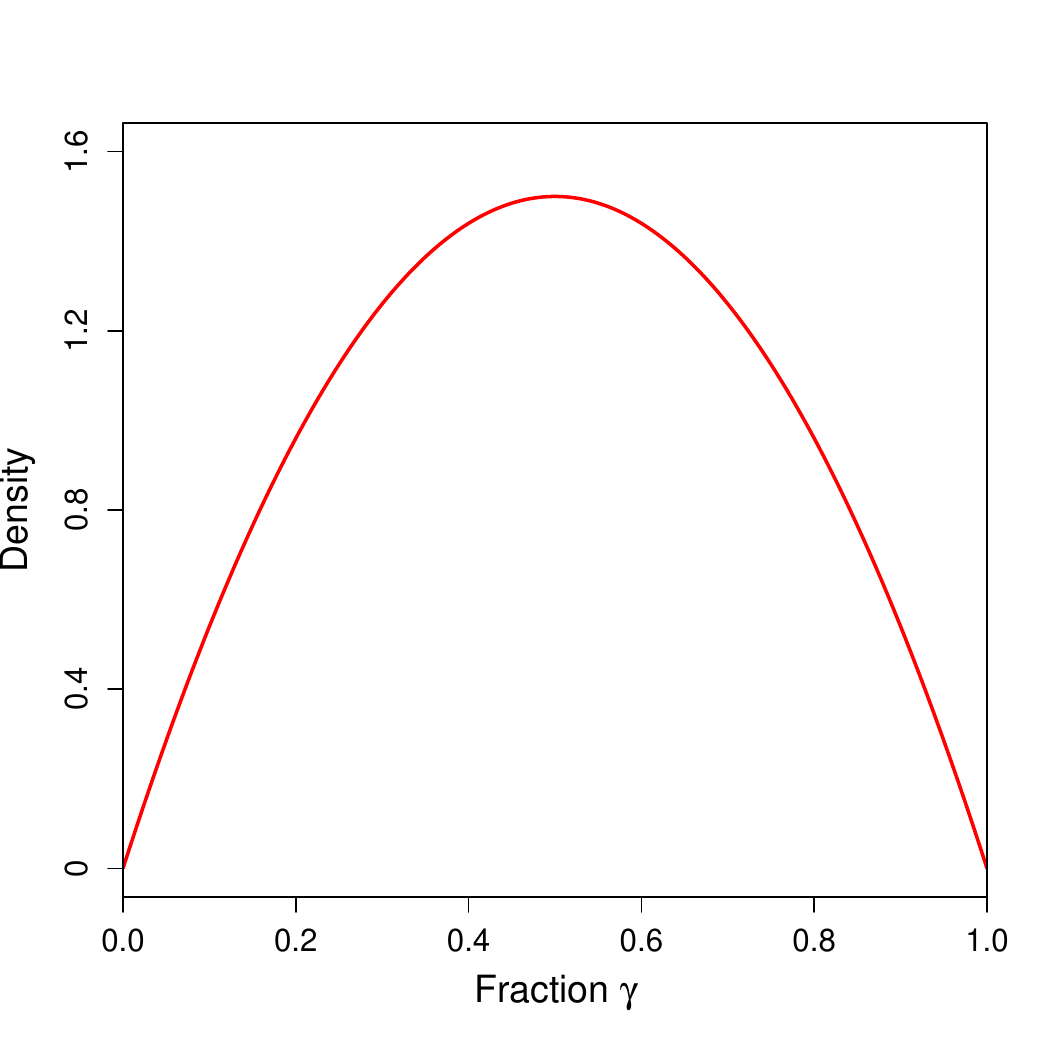}
\includegraphics[scale=0.4]{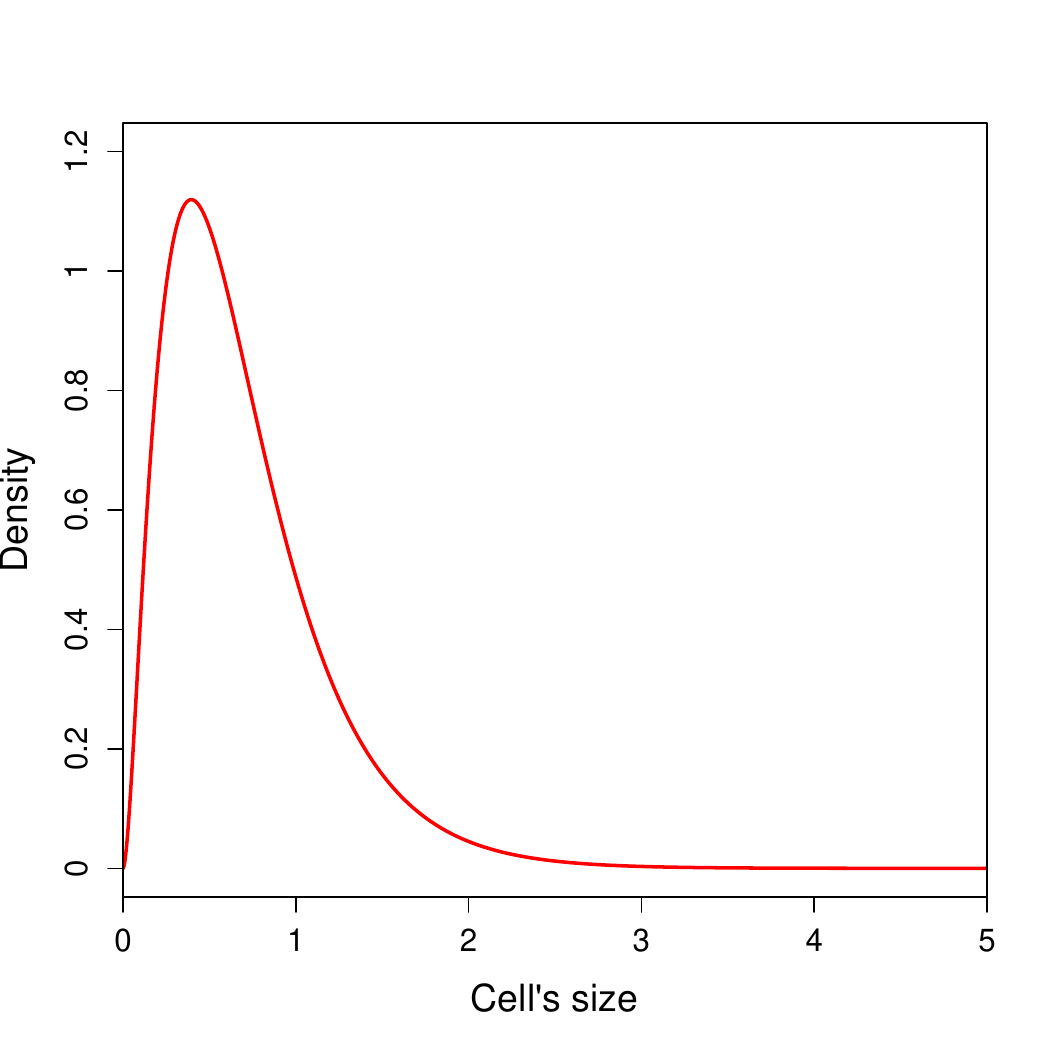}\vspace*{-0.25in}
\caption{\it\footnotesize The $\Beta(2,2)$ density $h_1$ (left) and its corresponding stationary density $N_1$ (right). \label{fig:h1N1}}
\end{figure}

\begin{figure}[htbp]
\centering
\includegraphics[scale=0.4]{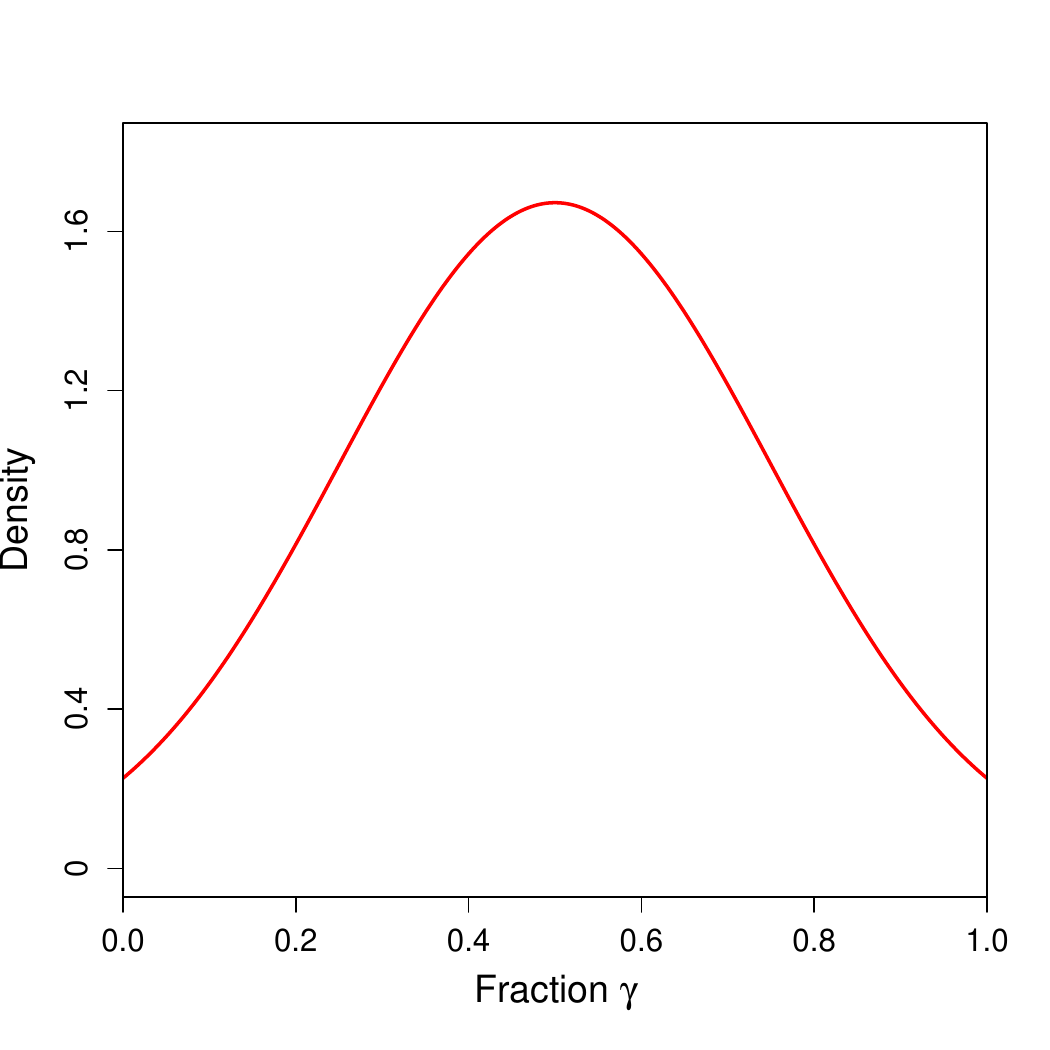}
\includegraphics[scale=0.4]{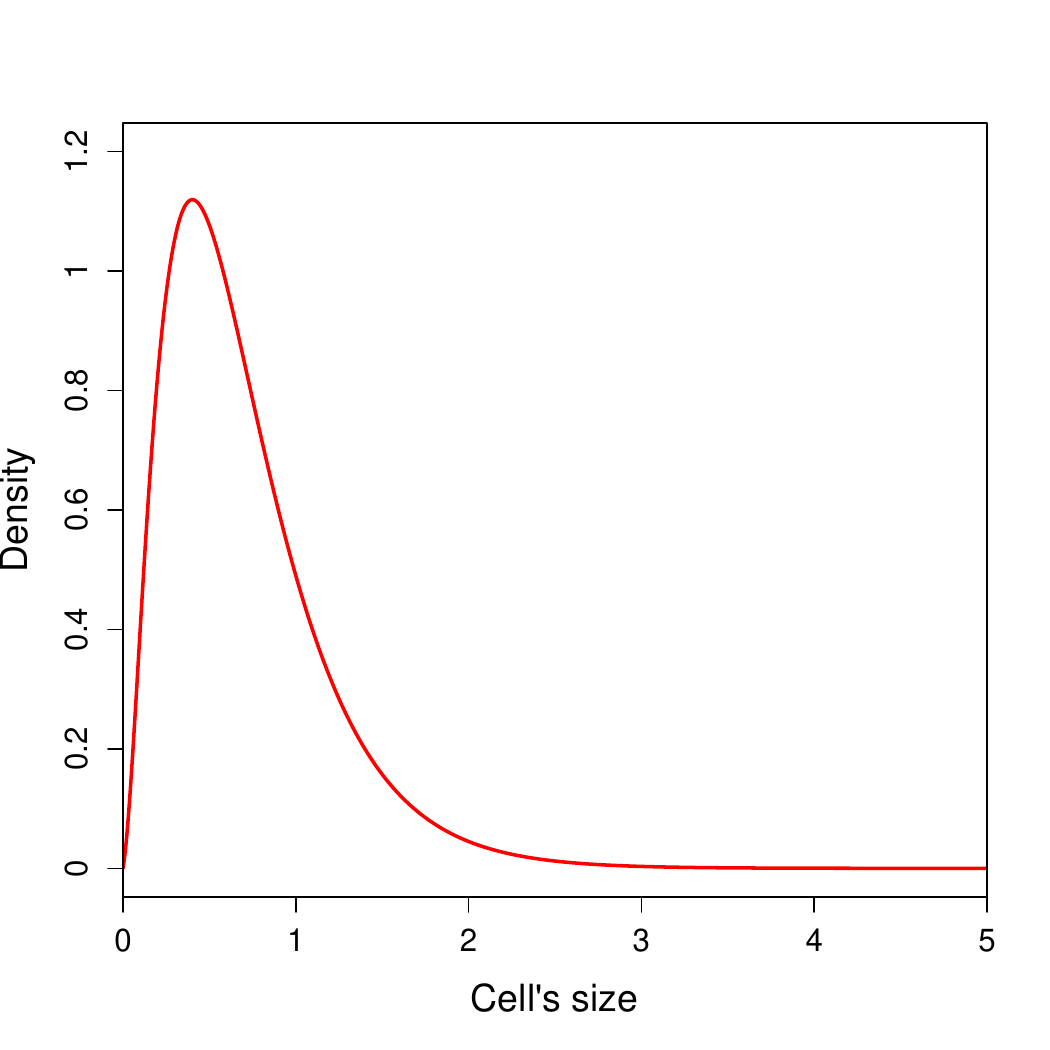}\vspace*{-0.25in}
\caption{\it\footnotesize The truncated normal density $h_2$ (left) and its corresponding stationary density $N_2$ (right). \label{fig:h2N2}}
\end{figure}

For the estimation of $h_1$ and $h_2$,  even if theoretical boundary conditions stated in Assumption~\ref{assump:moment-h} are not satisfied,
we shall observe that the procedure does a good job. Before presenting the numerical results, let us point out some difficulties that affect the quality of the estimation.

First, one can observe in Figures \ref{fig:h1N1} and \ref{fig:h2N2} that shapes of functions $N_1$ and $N_2$ are very similar although functions $h_1$ and $h_2$ are very different. This illustrates a major difficulty of our inverse problem and leads to some difficulties for the estimation of the densities $g$ and $h$.

Secondly, in view of \eqref{eq:FT-g} and \eqref{eq:estigell}, the construction of the estimator $\hat g_{n,\ell}$ is based on the estimation of $M^*$ and $D^*$. Remember that $D^*(\xi) = (-\ic\xi)\kv\big[e^{(\ic\xi - 1)U_1}\big]$ and the leading term $-\ic\xi$ of the last expression, coming from the computation of the Fourier transform of the derivation function $D$, gives large fluctuations for the estimation of $D^*$  when $\xi$ takes large values. To justify this point, we introduce the modified formulas of $D^*$ and $\widehat{D_n^*}$, denoted respectively by $\mathfrak{D}^*$ and $\widehat{\mathfrak{D}_n^*}$, obtained by removing $-\ic\xi$ from the original formulas:
\[
\mathfrak{D}^*(\xi) = \kv\big[e^{(\ic\xi - 1)U_1}\big] \quad\text{ and }\quad \widehat{\mathfrak{D}_n^*}(\xi) = \frac{1}{n}\sum_{j=1}^n e^{(\ic\xi - 1)U_j}.
\]
\bigskip

Figures \ref{fig:h1Mt}, \ref{fig:h1Dt} and \ref{fig:h1Dtv2} provide a reconstruction of $\widehat{M_n^*}$, $\widehat{D_n^*}$ and $\widehat{\mathfrak{D}_n^*}$ based on a random sample $U_1,\ldots, U_n$ of size $n=30,000$ for $h_1$. For each figure,  we represent both the real part and the imaginary part of $\widehat{M_n^*}$ (resp. $\widehat{D_n^*}$, $\widehat{\mathfrak{D}_n^*}$) and we compare them with those of $M^*$ (resp. $D^*$, $\mathfrak{D}^*$). The Fourier transforms $M^*$, $D^*$ and $\mathfrak{D}^*$ are computed directly from the function $N_1$, indicating that one can consider $M^*$, $D^*$ and $\mathfrak{D}^*$ as the ``true'' functions.

Figures \ref{fig:h1Mt} and \ref{fig:h1Dtv2} show that the reconstructions of $M^*$ and $\mathfrak{D}^*$ are very satisfying, whereas large oscillations in the reconstruction of $D^*(\xi)$ appear when $\xi$ is large (see Figure \ref{fig:h1Dt}), due to a large variance term. This confirms what we mentioned: the estimation of the derivative $D^*$ has a strong influence for our statistical problem.

\begin{figure}[ht!]
\centering
\includegraphics[scale=0.4]{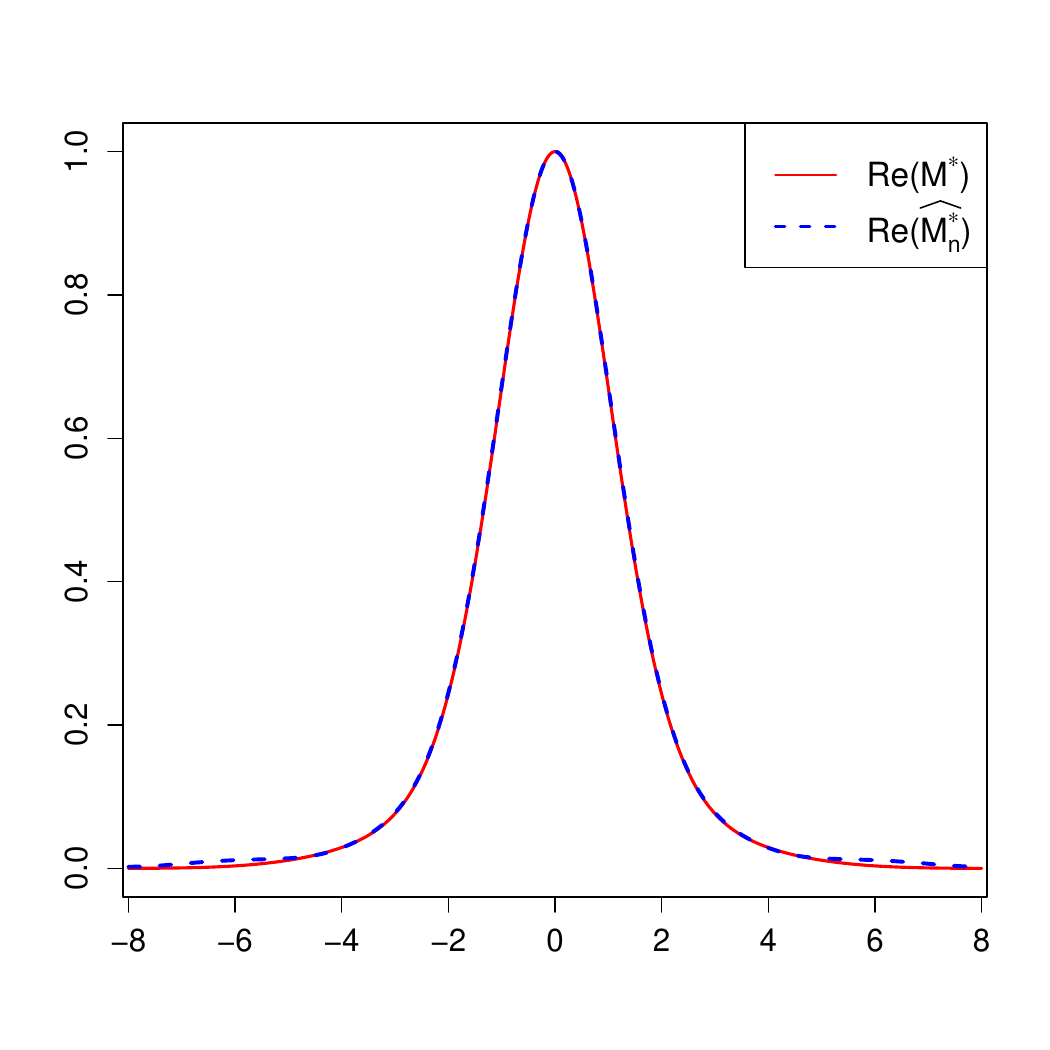}
\includegraphics[scale=0.4]{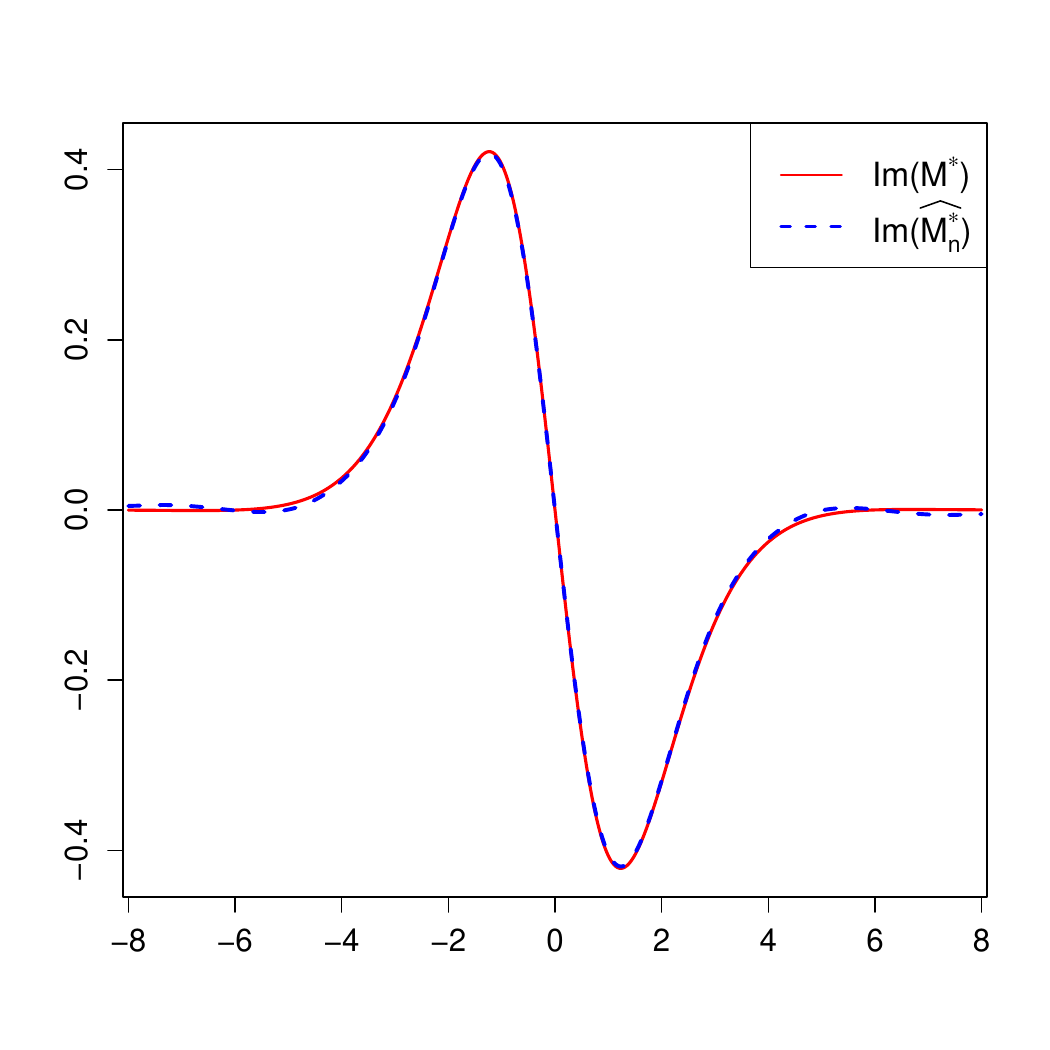}
\caption{\it\footnotesize For the $\Beta(2,2)$ density, the real part (left) and the imaginary part (right) of $\widehat{M_n^*}$ (blue line) compared with those of $M^*$ (red line). \label{fig:h1Mt}}
\end{figure}

\begin{figure}[ht!]
\centering
\includegraphics[scale=0.4]{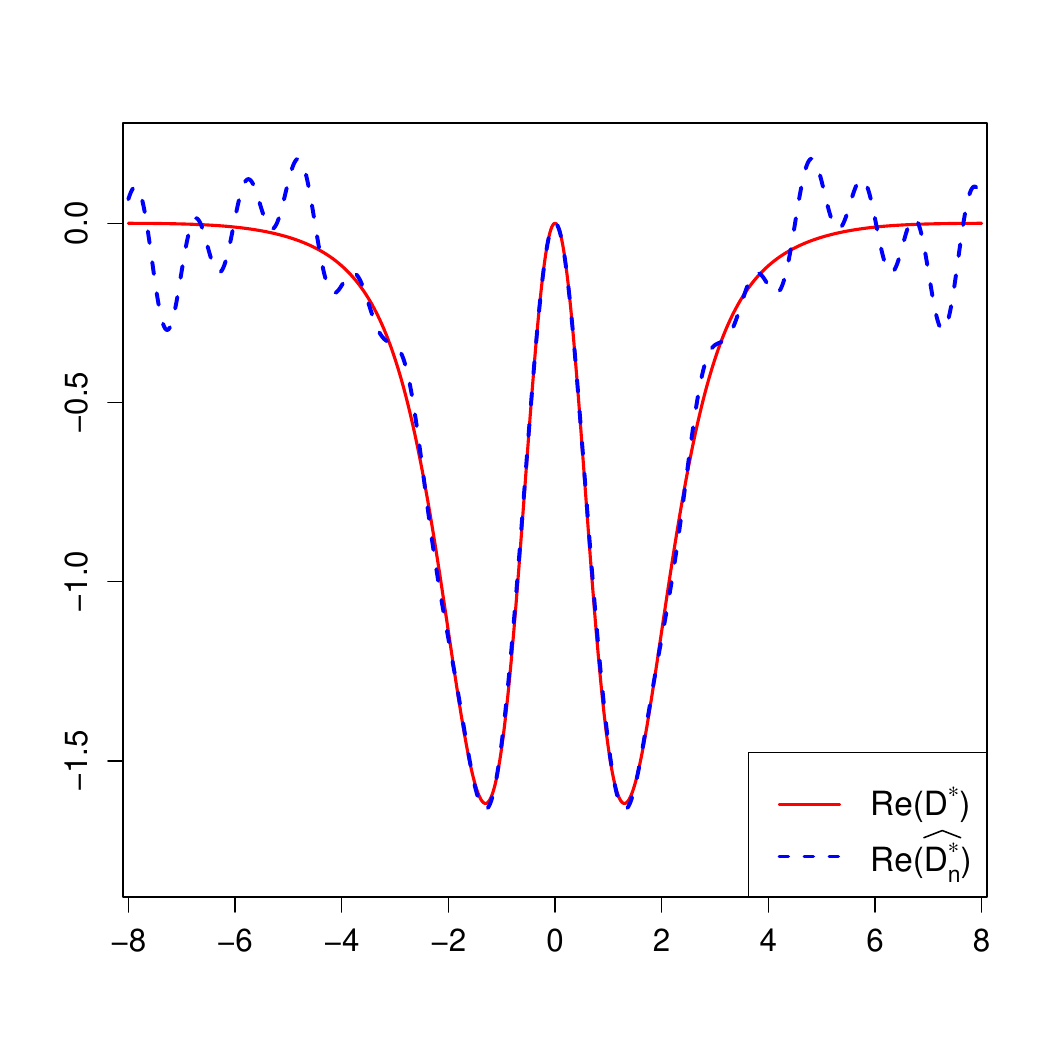}
\includegraphics[scale=0.4]{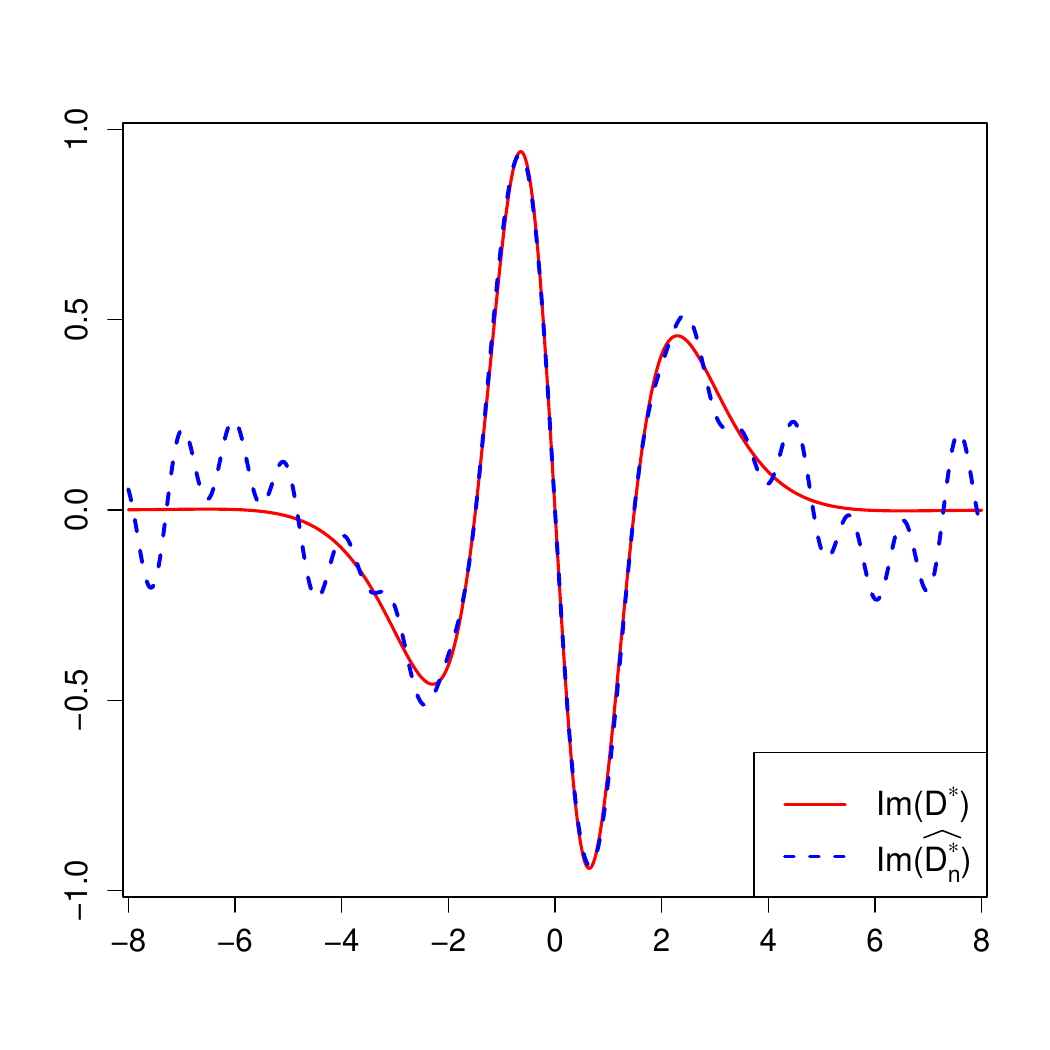}
\caption{\it\footnotesize For the $\Beta(2,2)$ density, the real part (left) and the imaginary part (right) of $\widehat{D_n^*}$ (blue line) compared with those of $D^*$ (red line). \label{fig:h1Dt}}
\end{figure}

\begin{figure}[ht!]
\centering
\includegraphics[scale=0.4]{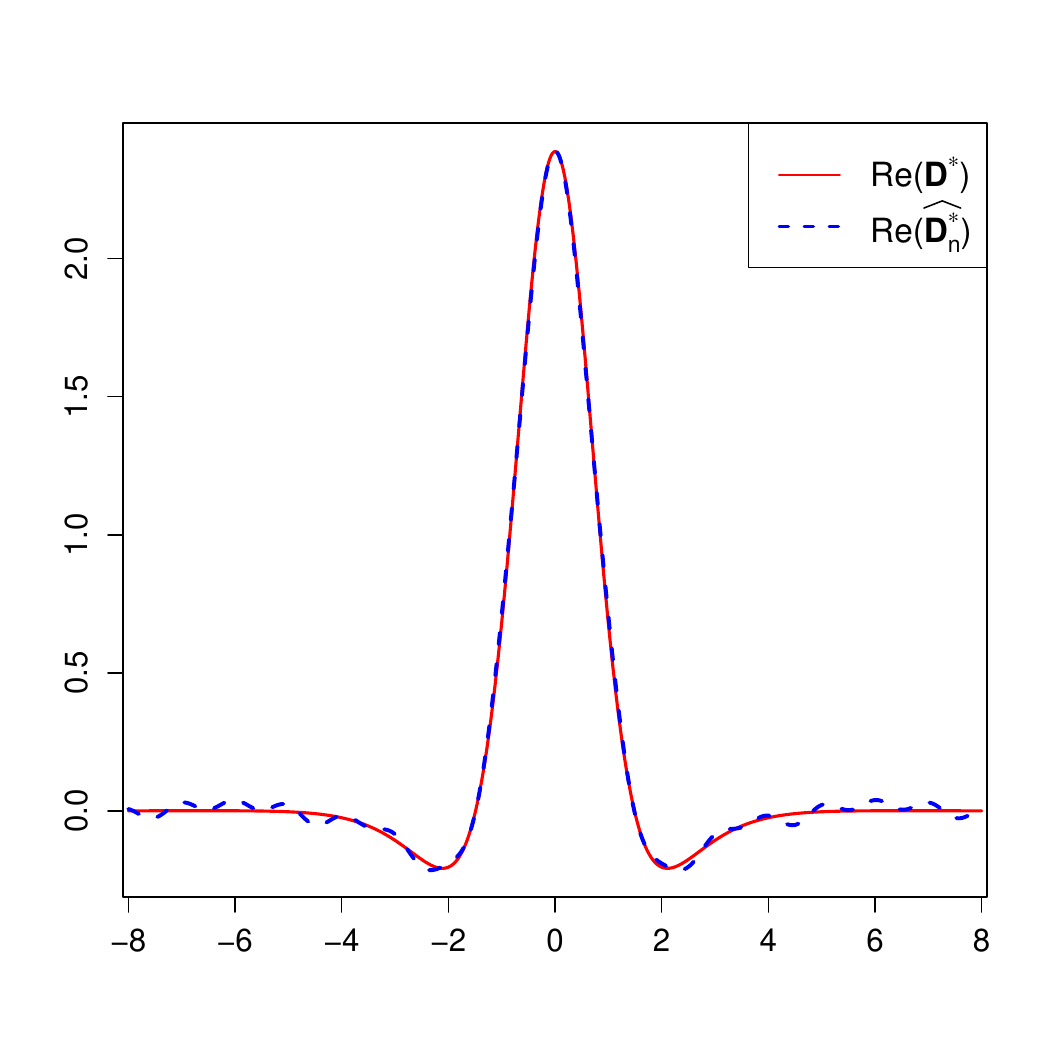}
\includegraphics[scale=0.4]{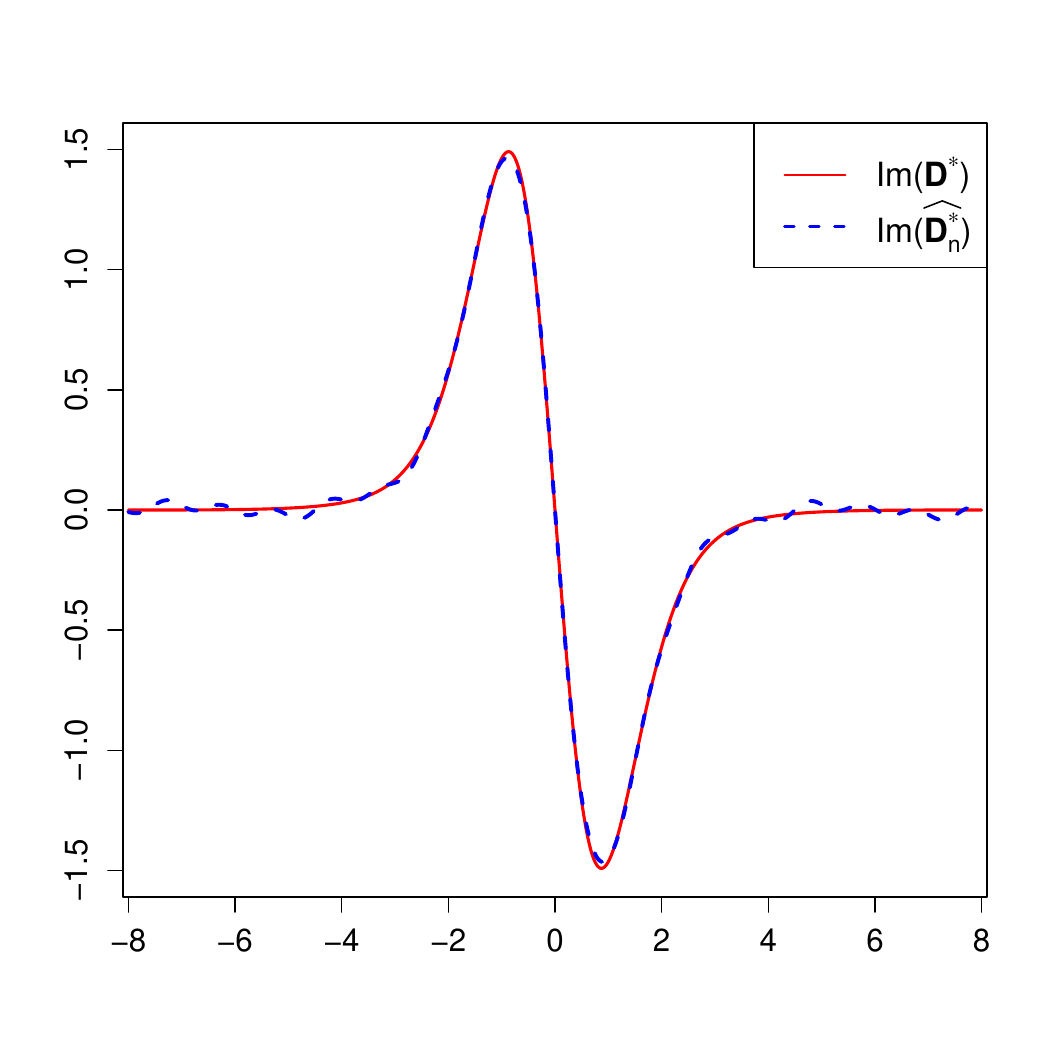}
\caption{\it\footnotesize For the $\Beta(2,2)$ density, the real part (left) and the imaginary part (right) of $\widehat{\mathfrak{D}_n^*}$ (blue line) compared with those of $\mathfrak{D}^*$ (red line). \label{fig:h1Dtv2}}
\end{figure}

In what follows, we introduce our bandwidth selection rules for the estimators $\hat{g}_{n,\bw}$ and $\hat{h}_{n,\bw}$, then we present some numerical results to illustrate the performances of our estimators.
\subsection{Bandwidth selection rules}
To establish a bandwidth selection rule for the estimator $\hat{g}_{n,\bw}$ and $\hat{h}_{n,\bw}$, we use resampling techniques inspired from the principle of cross-validation. We first study the $\mathbb{L}^2$-risk of the estimator $\hat{g}_{n,\bw}$ in the Fourier domain:
\begin{align*}
\ch{\hat{g}_{n,\bw} - g}^2 = \frac{1}{2\pi}\ch{\hat{g}^*_{n,\bw} - g^*}^2 = \frac{1}{2\pi}\left(\ch{\hat{g}^*_{n,\bw}}^2 - 2\langle \hat{g}^*_{n,\bw}, g^* \rangle\right) + \frac{1}{2\pi}\ch{g^*}^2.
\end{align*}
Define
\[
J(\bw) := \ch{\hat{g}^*_{n,\bw}}^2 - 2\langle \hat{g}^*_{n,\bw}, g^* \rangle
\]
where the scalar product of two complex functions $u$ and $v$ is defined as
\[
\langle u, v \rangle = \int_\rhoa u(\xi)\overline{v(\xi)}d\xi.
\]
Let $\mathcal{L}$ be a family of possible bandwidths, the optimal bandwidth is given by
\[
\bw_{CV} := \underset{\bw\in \mathcal{L}}{\argmin} J(\bw)=\underset{\bw\in \mathcal{L}}{\argmin} \ch{\hat{g}_{n,\bw} - g}^2.
\]
We aim at constructing an estimator of $J(\bw)$, which is equivalent to providing an estimate of the scalar product $\langle \hat{g}^*_{n,\bw}, g \rangle$ since $\ch{\hat{g}^*_{n,\bw}}^2$ is known. Instead of finding a closed formula for the estimator of the $\mathbb{L}^2$-risk which is intricate in our case, we use the following alternative approach: we start from a random sample and divide it into two disjoint sets, called the training set and the validation set. They are respectively used for computing the estimator and measuring its performance. For sake of simplicity, those sets have the same size.
Let $\hat{g}^{*(t)}_{n,\bw}$ (resp. $\hat{g}^{*(v)}_{n,\bw}$) be the estimator of $g^*$ constructed on the training set (resp. on the validation set). The heuristics is that if $\hat{g}^{*(v)}_{n,\bw'}$ is an estimator constructed on the validation set, then $\langle \hat{g}^{*(t)}_{n,\bw}, \hat{g}^{*(v)}_{n,\bw'} \rangle$ gives us an estimate of $\langle \hat{g}^{*(t)}_{n,\bw}, g^* \rangle$ and subsequently  an estimate of $J(\bw)$. The final bandwidth is the one which minimizes the average of all risk estimates computed over a number of couples of training-validation set selected from the same sample.

In detail, let $\{X_1,\ldots, X_n\}$ be a random sample. Let $E$ and $E^C$ be the subsets of $\{1,\ldots,n\}$ such that $|E|=n/2$ and $E^C = \{1,\ldots,n\}\setminus E$. We divide $\{X_1,\ldots, X_n\}$ into two sub-samples:
\[
\mathbf{X}^{E} := (X_i)_{i\in E}\quad\text{ and }\quad \mathbf{X}^{E^C} := (X_i)_{i\in E^C}.
\]
There are $V_{\max}$ possibilities to select the subsets $(E, E^c)$, where
$$V_{\max} := \binom{n}{n/2}.$$ If $n$ is large then $V_{\max}$ will be huge. Hence we choose in practice a number $V$ which is smaller than $V_{\max}$ to reduce computation time.
We propose two criteria  for the selection of bandwidths as follows.
\begin{definition}\label{def:selection-rule-1}
Let $(E_j, E_j^C)_{1\le j\le V}$, $V\le V_{\max}$ be the sequence of subsets selected from $\{1,\ldots,n\}$ and the corresponding sub-samples $(\mathbf{X}^{E_j}, \mathbf{X}^{E^C_j})_{1\le j\le V}$. Let
${\hat{g}_{n,\bw}}^{*(E_j)}$ and ${\hat{g}_{n,\bw}}^{*(E_j^C)}$ be the estimators of $g^*_\bw$ respectively constructed on the sub-samples $\mathbf{X}^{E_j}$ and $\mathbf{X}^{E^C_j}$. Define
\begin{equation}\label{CV-risk-crit01}
\widehat{J}_{Crit1}(\bw) := \frac{1}{V}\sum_{j=1}^V \left[\ch{{\hat{g}_{n,\bw}}^{*(E_j)}}^2 - 2\big\langle {\hat{g}_{n,\bw}}^{*(E_j)}, {\hat{g}_{n,\bw}}^{*(E^C_j)}\big\rangle \right].
\end{equation}
Then the selected bandwidth is given by
\begin{equation}\label{CV-crit01}
\hat{\bw}_{Crit1} := \underset{\bw \in\mathcal{L} }{\argmin} \widehat{J}_{Crit1}(\bw).
\end{equation}
\end{definition}

\begin{definition}\label{def:selection-rule-2}
Let ${\hat{g}_{n,\bw}}^{*(E_j)}$ and ${\hat{g}_{n,\bw'}}^{*(E_j^c)}$ be the estimators of $g^*_{n,\bw}$ as in Definition \ref{def:selection-rule-1}. Define,
\begin{equation}\label{CV-risk-crit02}
\widehat{J}_{Crit2}(\bw,\bw') := \frac{1}{V}\sum_{j=1}^V \left[\ch{{\hat{g}_{n,\bw}}^{*(E_j)}}^2 - 2\big\langle {\hat{g}_{n,\bw}}^{*(E_j)}, {\hat{g}_{n,\bw'}}^{*(E^C_j)}\big\rangle \right].
\end{equation}
Then an alternative bandwidth selection rule is given as follows:
\begin{equation}\label{CV-crit02}
\hat{\bw}_{Crit2} := \underset{\bw \in\mathcal{L} }{\argmin}\Big\{\underset{\bw'\in\mathcal{L}}{\min}\, \widehat{J}_{Crit2}(\bw,\bw')\Big\}.
\end{equation}
\end{definition}
Note that the second criterion is more computationally intensive.
\subsection{Numerical results}\label{sec:num-CV}
Remember that we aim at reconstructing the densities $h_1$ and $h_2$, \ie the $\Beta(2,2)$ density and the density of a truncated normal $\mathcal{N}(0.5,0.25^2)$ on $[0,1]$. We apply formulas \eqref{eq:estigell}, \eqref{eq:estimator-g} and \eqref{eq:estimator-h} to construct the estimators for these densities. The bandwidth $\ell$ is chosen in the family $\mathcal{L} \subset \big\{1/(0.5\Delta) , \ \Delta =1,\ldots, 50\big\}$ according to two bandwidth selection rules. We compare the estimated densities when using our selection rules with those estimated with the oracle bandwidth. The oracle bandwidth is the optimal bandwidth obtained by assuming that we know the true density and defined as follows:
\[
\ell_{\text{oracle}}:= \underset{\ell\in \mathcal{L}}{\argmin} \ch{\hat g_{n,\ell}- g}^2.
\]
Of course, $\ell_{\text{oracle}}$ and $\hat{g}_{n,\ell_{\text{oracle}}}$ cannot be used in practice (since they depend on the true function to estimate) but they can be viewed as benchmark quantities.
For $n=30, 000$ observations, we illustrate in Figures \ref{fig:estimate-g1h1} and \ref{fig:estimate-g2h2} the estimates of $( g_1, h_1)$ and $(g_2, h_2)$ using the first bandwidth selection rule (see Definition \ref{def:selection-rule-1}).

These graphs show  bad behaviors when reconstructing $h_1$ and $h_2$ if we do not take into account the symmetry of theses densities.
Considering symmetrization (see \eqref{eq:symm}) provides significant improvements (see  Figure \ref{fig:htilde12}). Reconstructions of densities are quite satisfying except at boundaries of $[0,1]$, which is expected in view of remarks of Section~\ref{sec:Relations}.

\begin{figure}[!ht]
\centering
\includegraphics[scale=0.38]{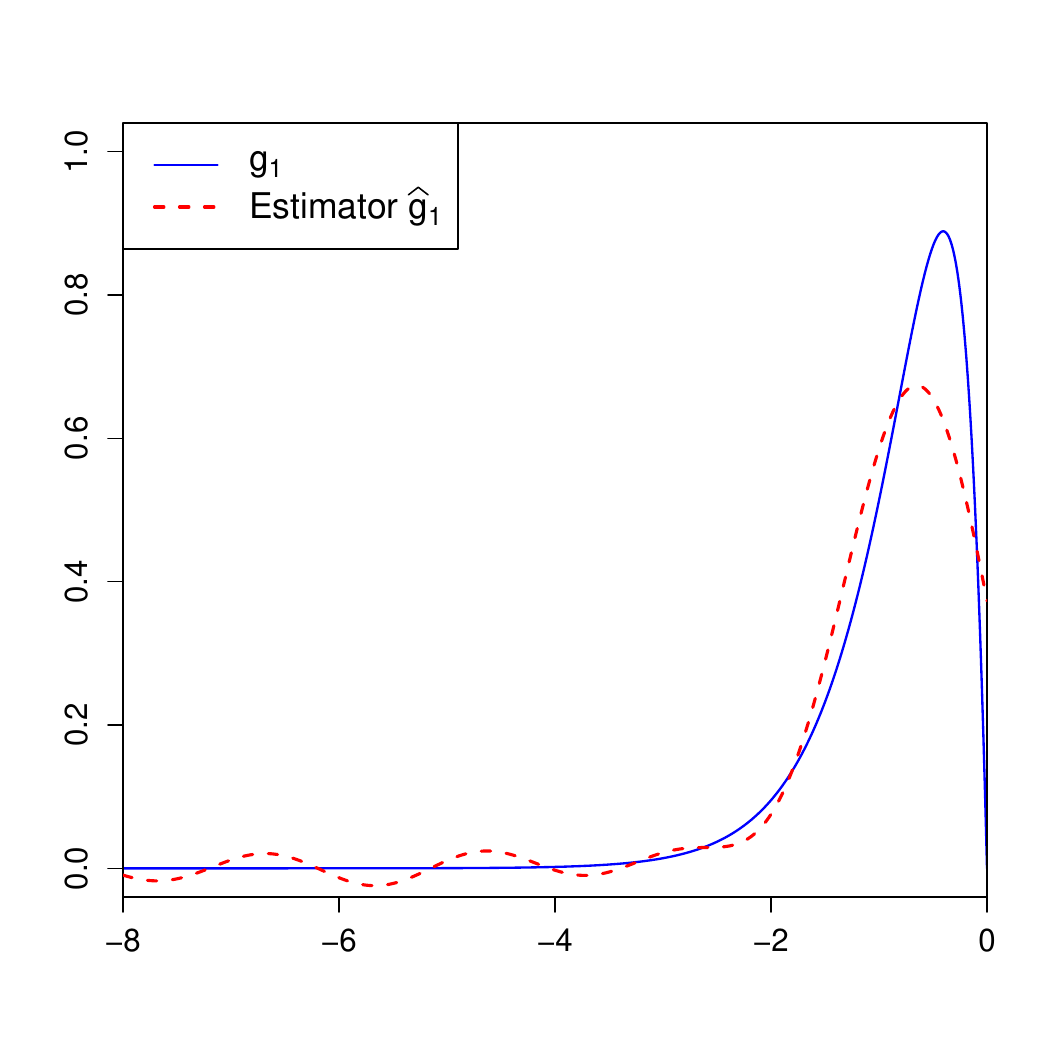}
\includegraphics[scale=0.38]{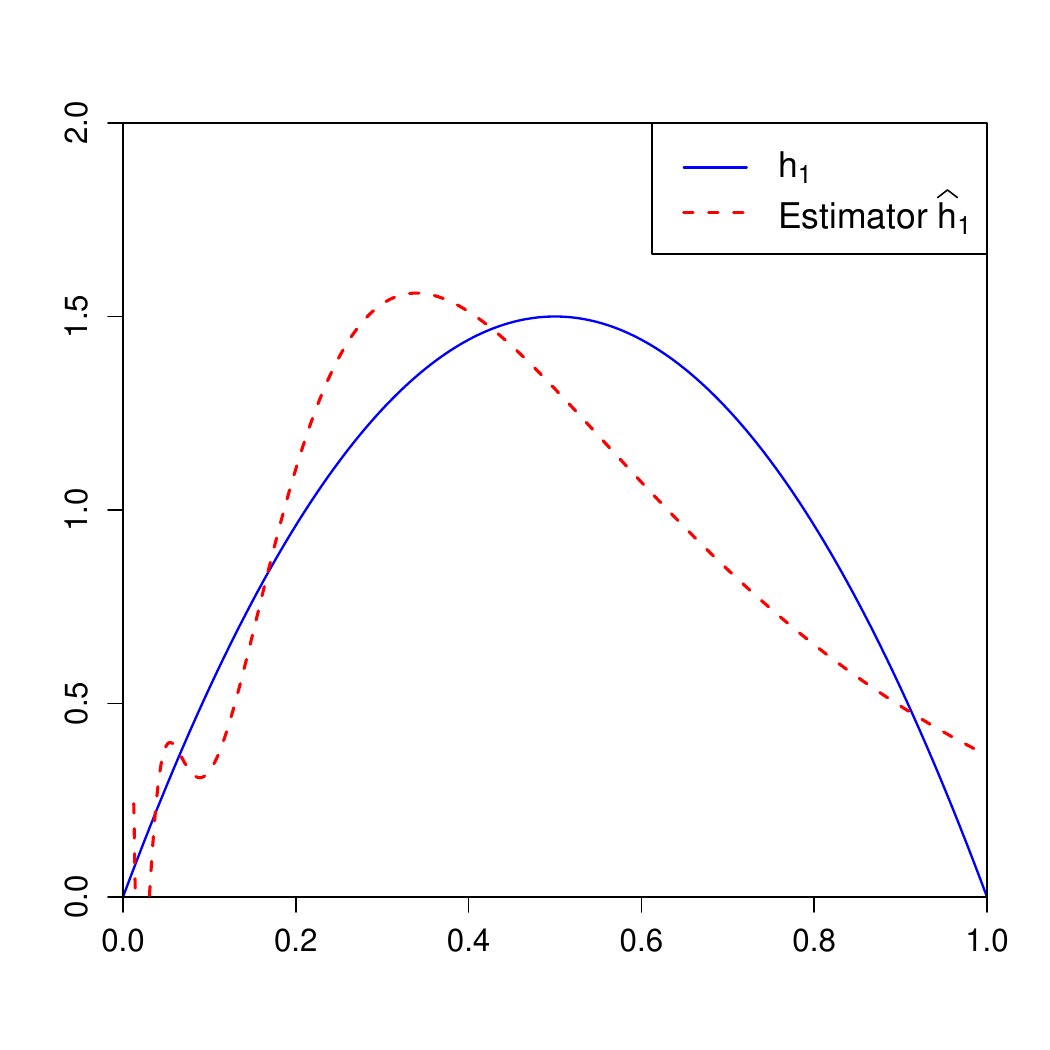}\vspace*{-0.25in}
\caption{\it\footnotesize Estimation of $g_1(x) = e^xh_1(e^x)$ (left) and $h_1$ (right). \label{fig:estimate-g1h1}}
\end{figure}

\begin{figure}[!ht]
\centering
\includegraphics[scale=0.38]{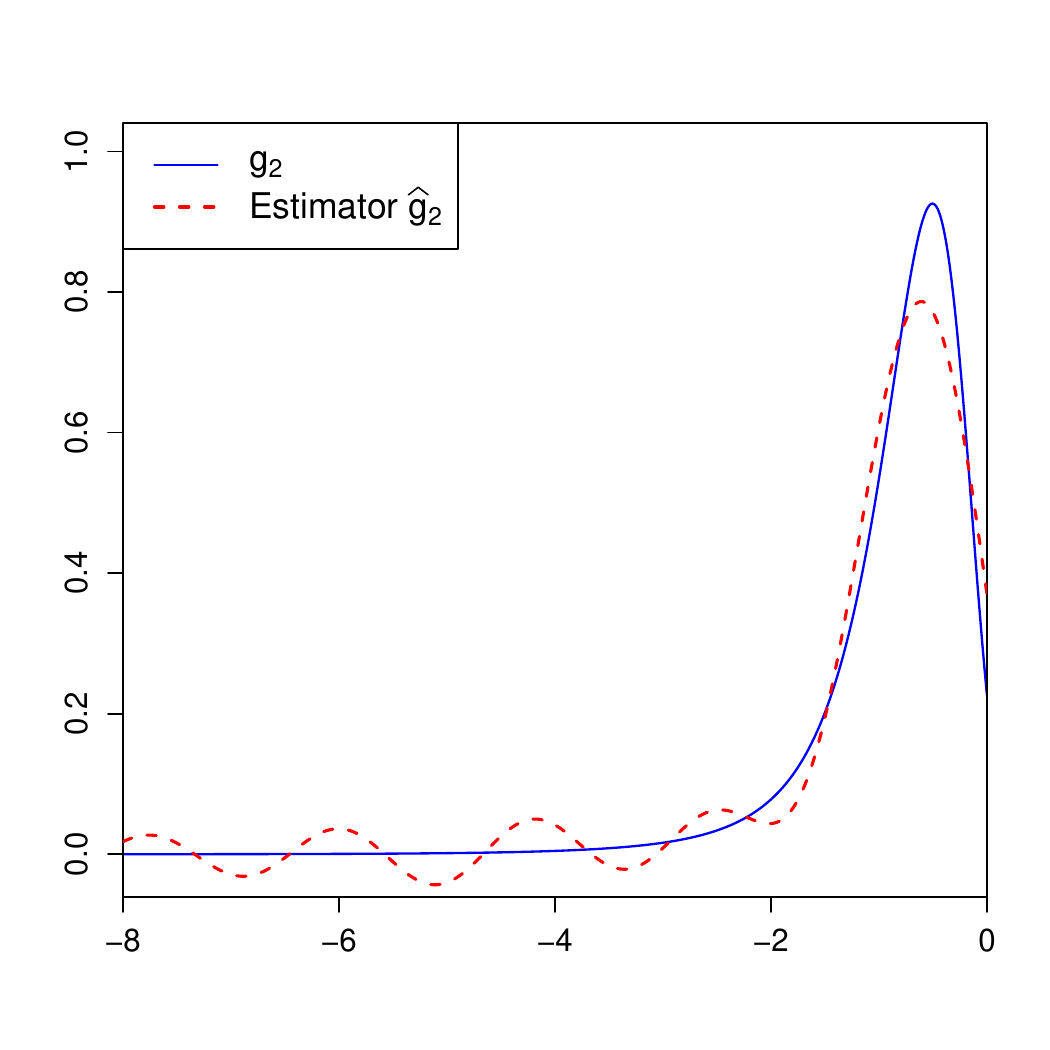}
\includegraphics[scale=0.38]{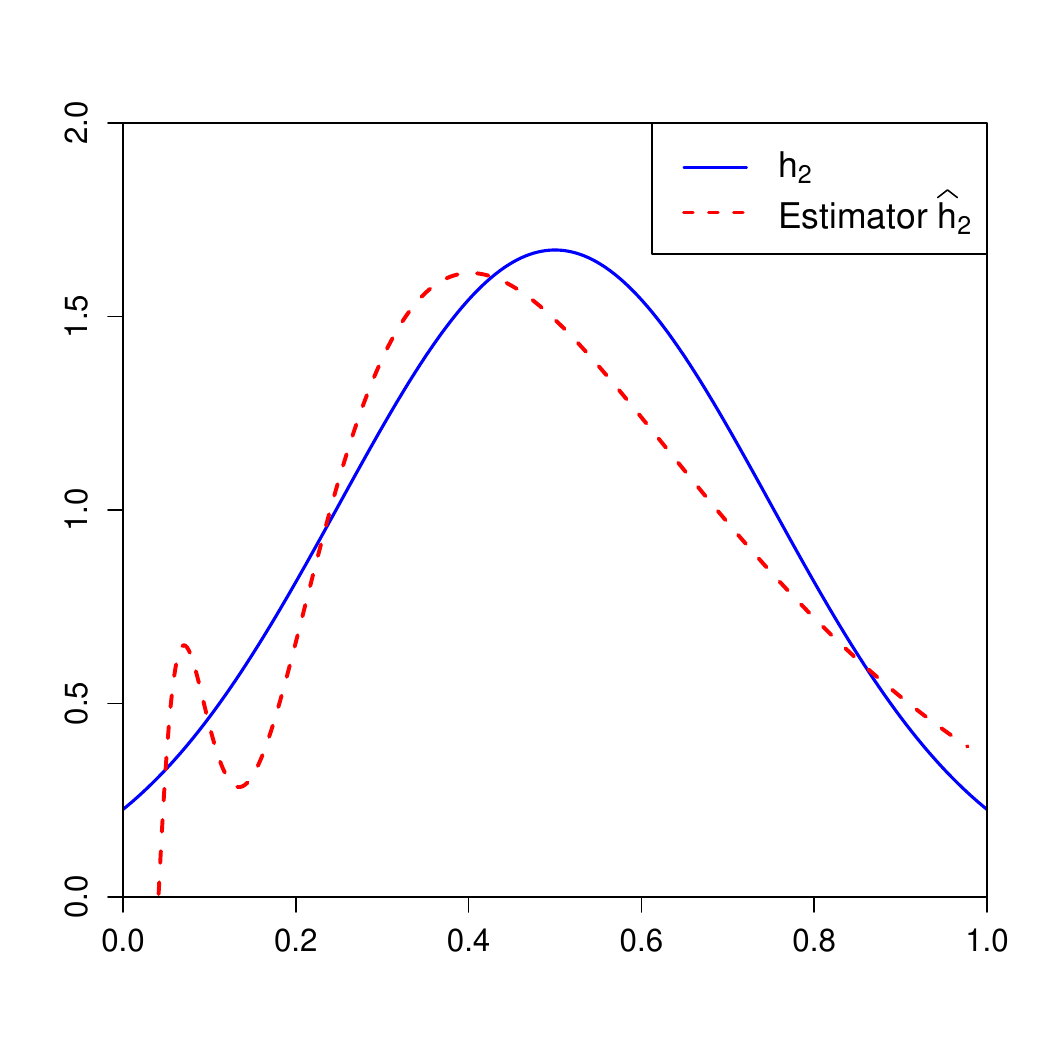}\vspace*{-0.25in}
\caption{\it\footnotesize Estimation of $g_2(x) = e^xh_2(e^x)$ (left) and $h_2$ (right). \label{fig:estimate-g2h2}}
\end{figure}

\begin{figure}[!ht]
\centering
\includegraphics[scale=0.38]{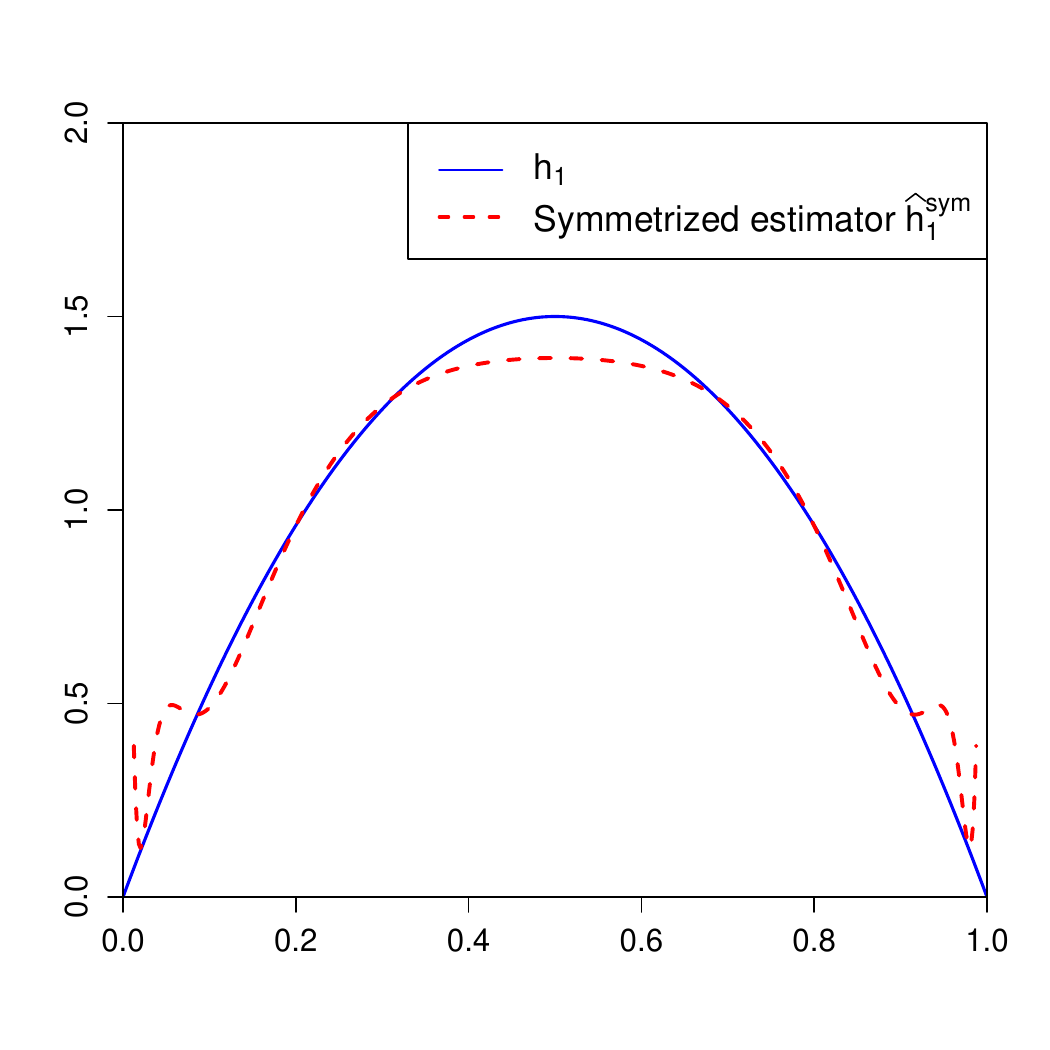}
\includegraphics[scale=0.38]{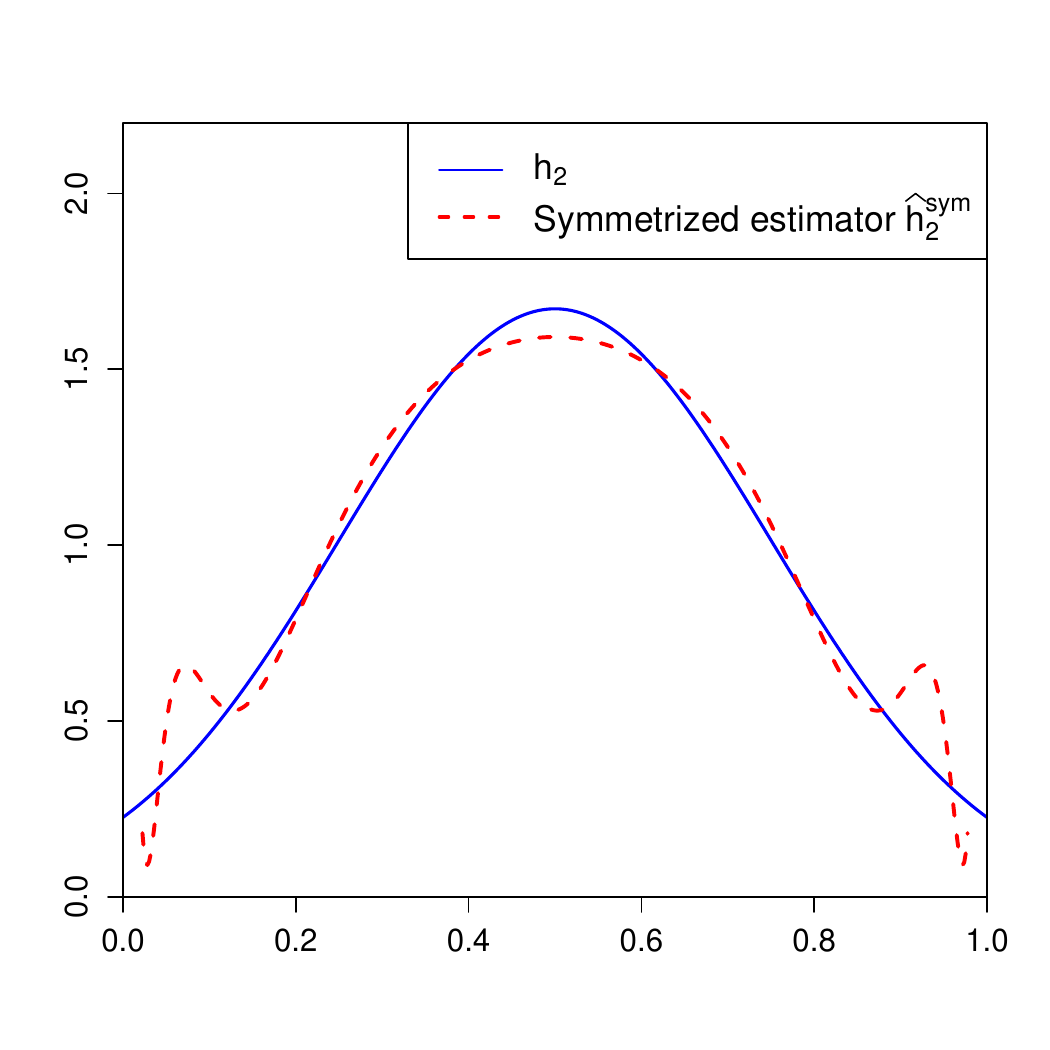}
\caption{Reconstructions of $h_1$  (left) and $h_2$ (right) after symmetrization.}\label{fig:htilde12}
\end{figure}
\newpage
Table \ref{tab1} shows the $\mathbb{L}^2$-risk of $\hat{g}_{n,\hat{\ell}_{Crit1}}$ and  $\hat{g}_{n,\hat{\ell}_{Crit2}}$ where $\hat{\ell}_{Crit1}$ and $\hat{\ell}_{Crit2}$ are the bandwidths selected by our selection rules (see Definitions \ref{def:selection-rule-1} and \ref{def:selection-rule-2}), over 100 Monte Carlo runs for estimating  $h_1$ and $h_2$ with respect to $V = 10, 25$ and $40$. The sample size for each repetition is $n = 30, 000$. We also provide associated Boxplots in Figure \ref{fig:box1} and \ref{fig:box2}.

\begin{table}[ht]
\centering
\begin{tabular}{|ll|lll|lll|}
\hline
			&					&  \multicolumn{3}{|c|}{$h_1$ - \text{$\Beta(2,2)$}} & \multicolumn{3}{|c|}{\text{$h_2$ - Truncated normal}}\\
\cline{3-8}										
			&					&	Crit1	&	Crit2	&	Oracle	&	Crit1	&	Crit2	&	Oracle		\\
\hline											
$V = 10$		&	$\bar{e}$			&0.04155 &0.04031 &0.03056	&0.03703 &0.03669 &0.02806	\\
			&	$\bar{\hat\ell}$		&0.29839 &0.29606 &0.27583	&0.30255 &0.30312 &0.27858	\\

\hline											
$V = 25$		&	$\bar{e}$			&0.04145 &0.03898 &0.03056		 	&0.03679 &0.03602 &0.02806				\\
			&	$\bar{\hat\ell}$		&0.29732 &0.29787 &0.27583			&0.30348 &0.30155 &0.27858				\\
\hline											
$V = 40$		&	$\bar{e}$			&0.04039 &0.03708 &0.03056			&0.03613 &0.03440 &0.02806				\\
			&	$\bar{\hat\ell}$		&0.29837 &0.29985 &0.27583			&0.30396 &0.30303 &0.27858				\\
\hline	
\end{tabular}
\caption{\footnotesize\itshape Average of the $\mathbb{L}^2$-risk of $\hat{g}_{n,\hat{\ell}_{Crit1}}$ and  $\hat{g}_{n,\hat{\ell}_{Crit2}}$ over $100$ Monte Carlo repetitions for estimating $h_1$ and $h_2$, compared with those of the oracle. \label{tab1}}
\end{table}

\begin{figure}[!ht]
\def\widthboxplot1{4.75cm}
\centering
\begin{tabular}{cccc}
& $V = 10$ 	& $V = 25$ 	&$V = 40$ \\
\rotatebox[origin = c]{90}{Bandwidths}
& \parinc{\widthboxplot1}{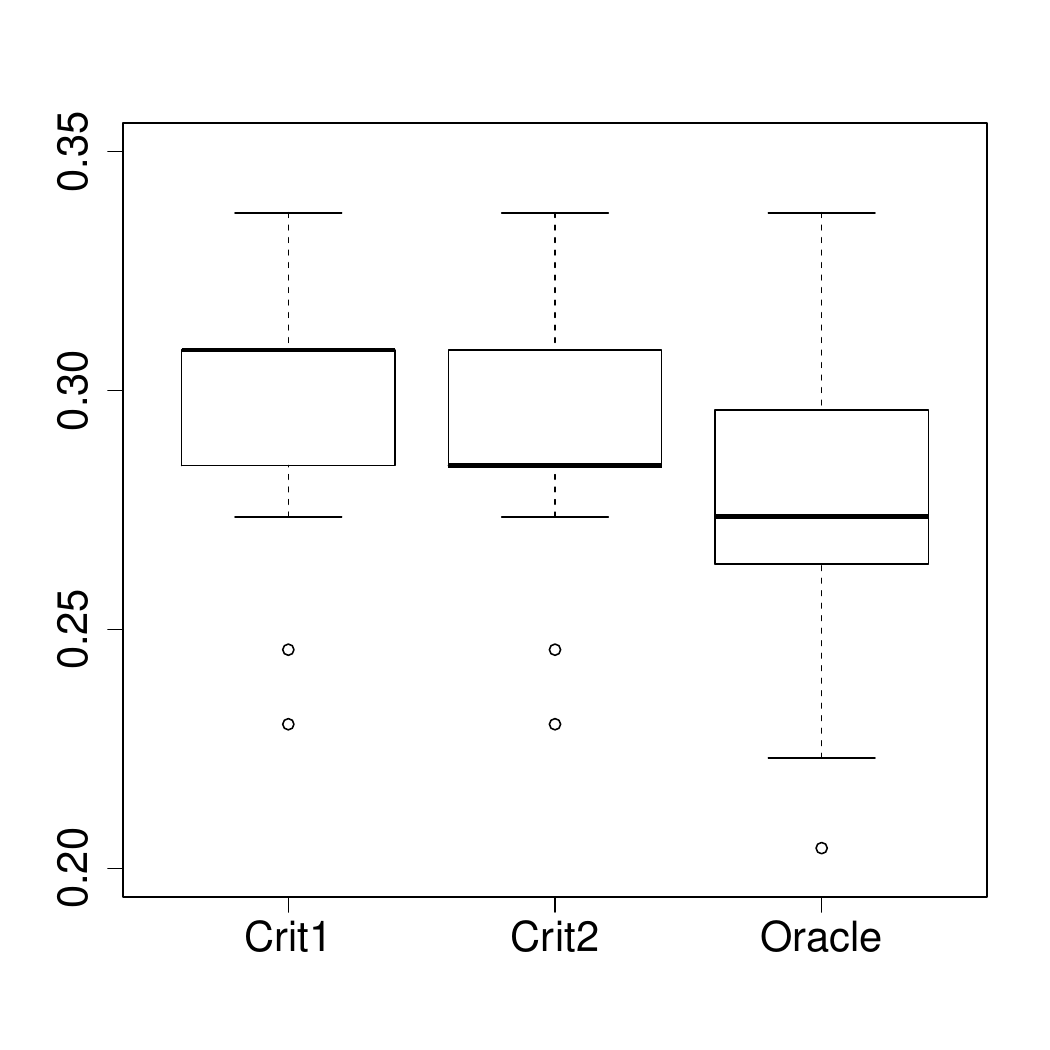}
& \parinc{\widthboxplot1}{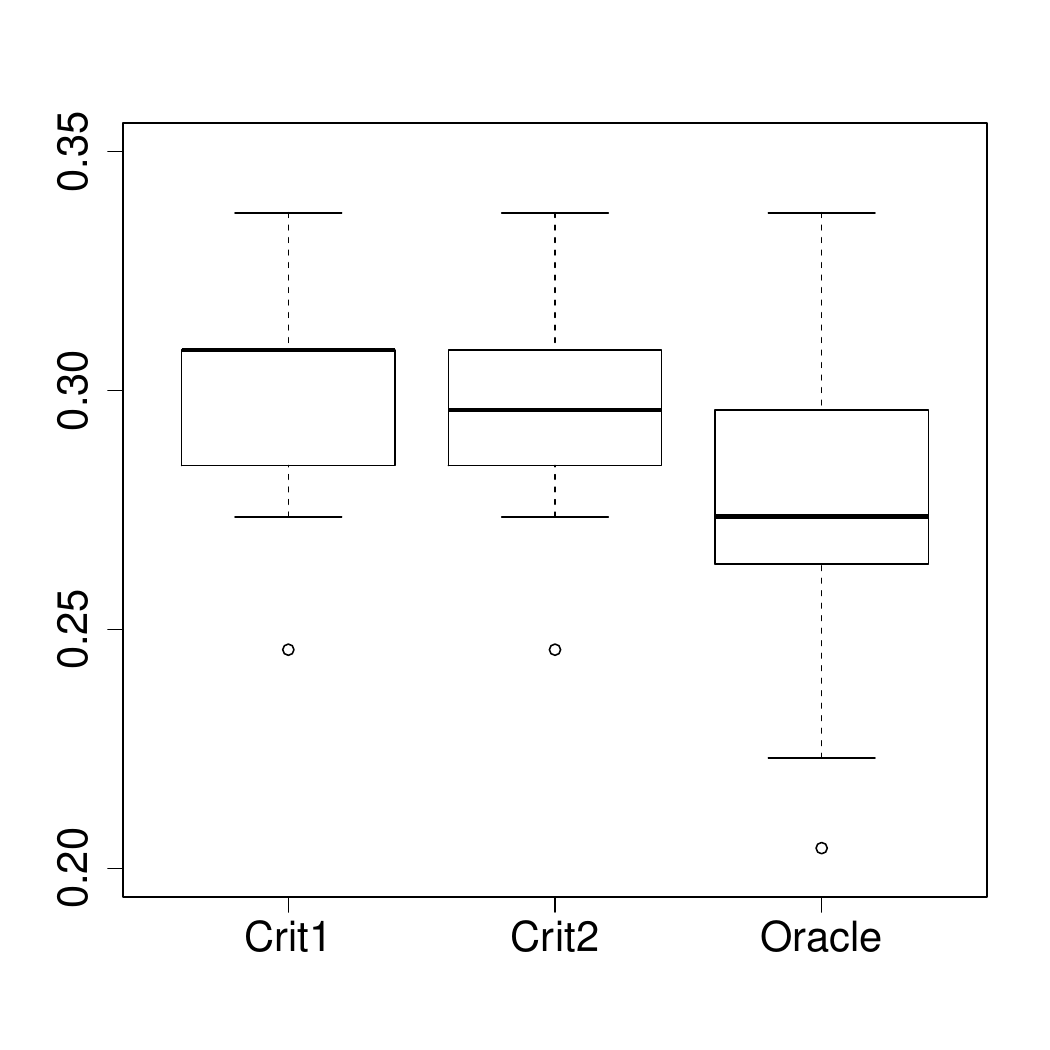}
& \parinc{\widthboxplot1}{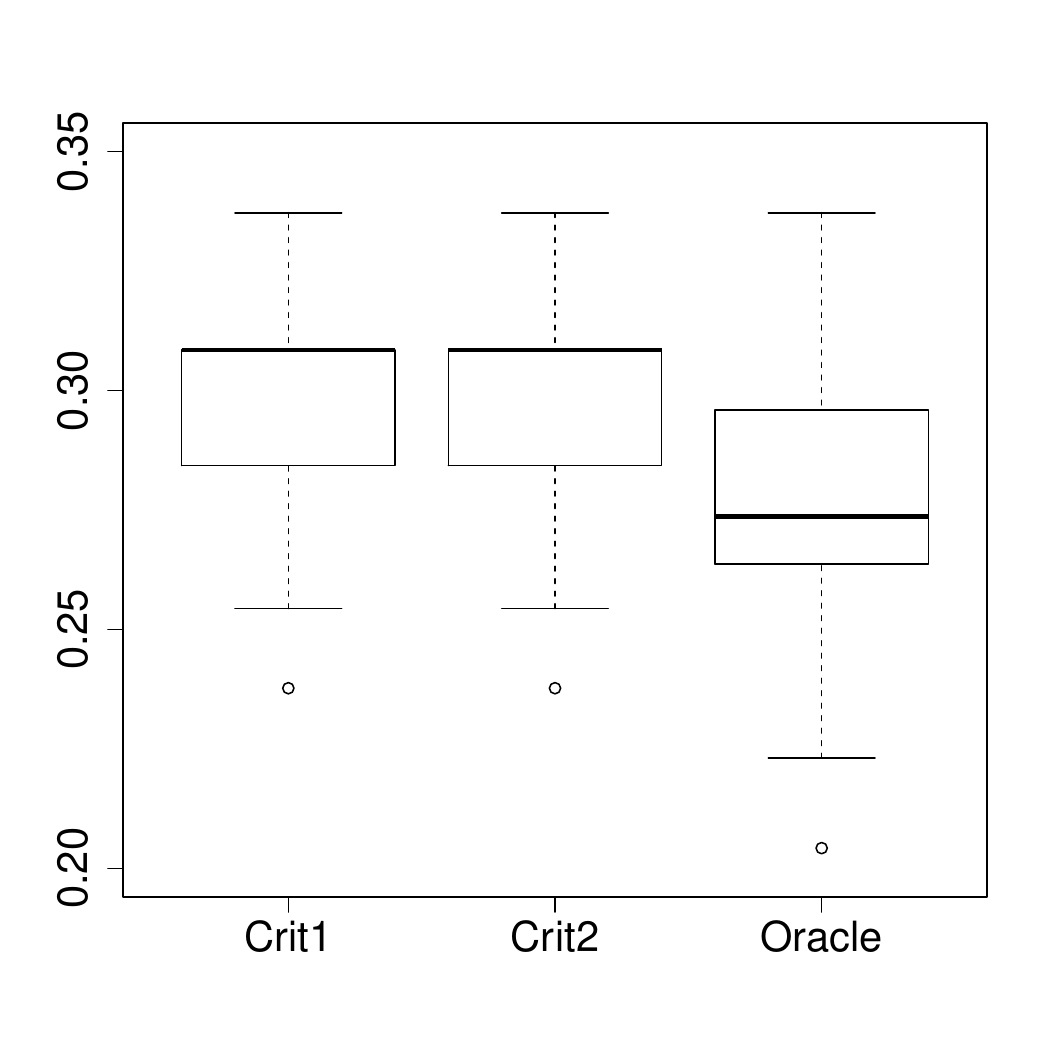} \\
\rotatebox[origin = c]{90}{Errors}
& \parinc{\widthboxplot1}{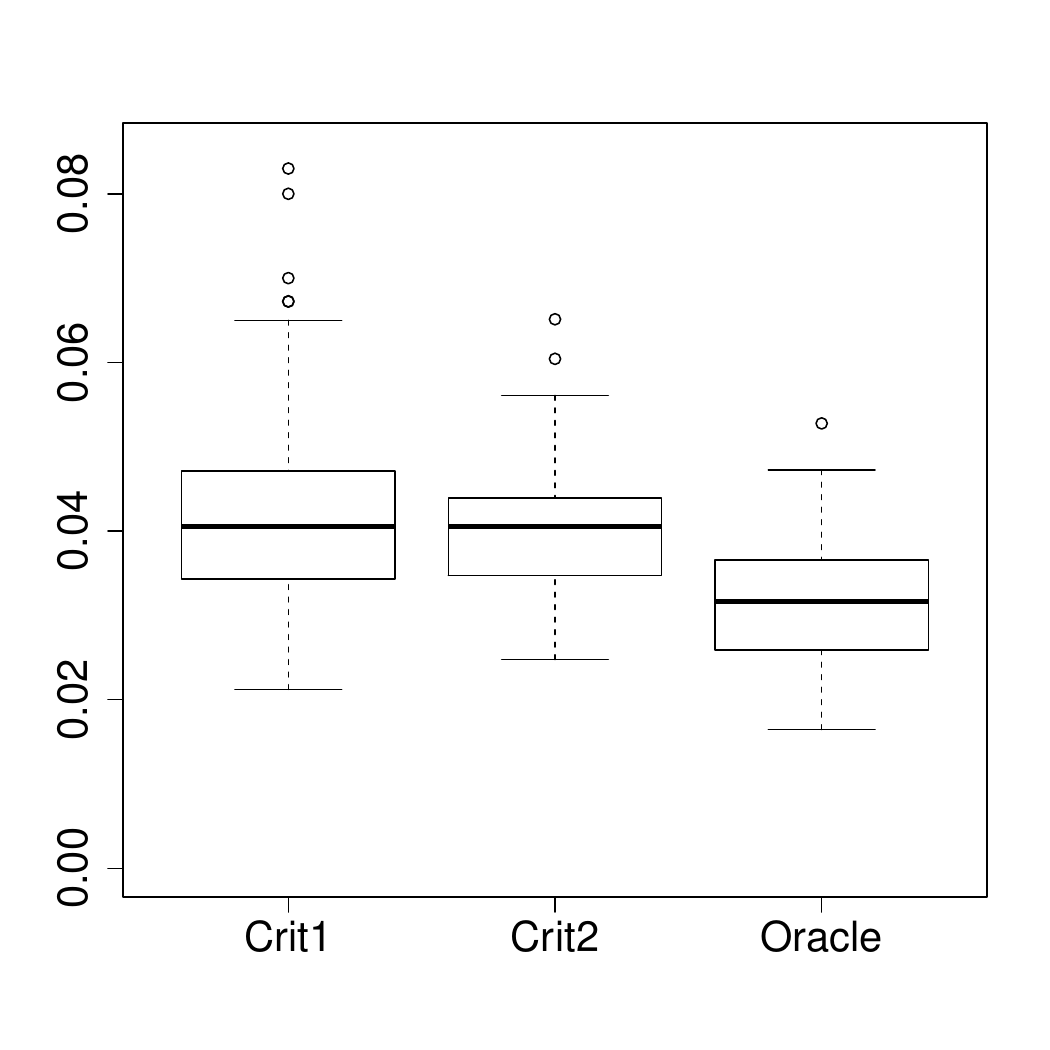}
& \parinc{\widthboxplot1}{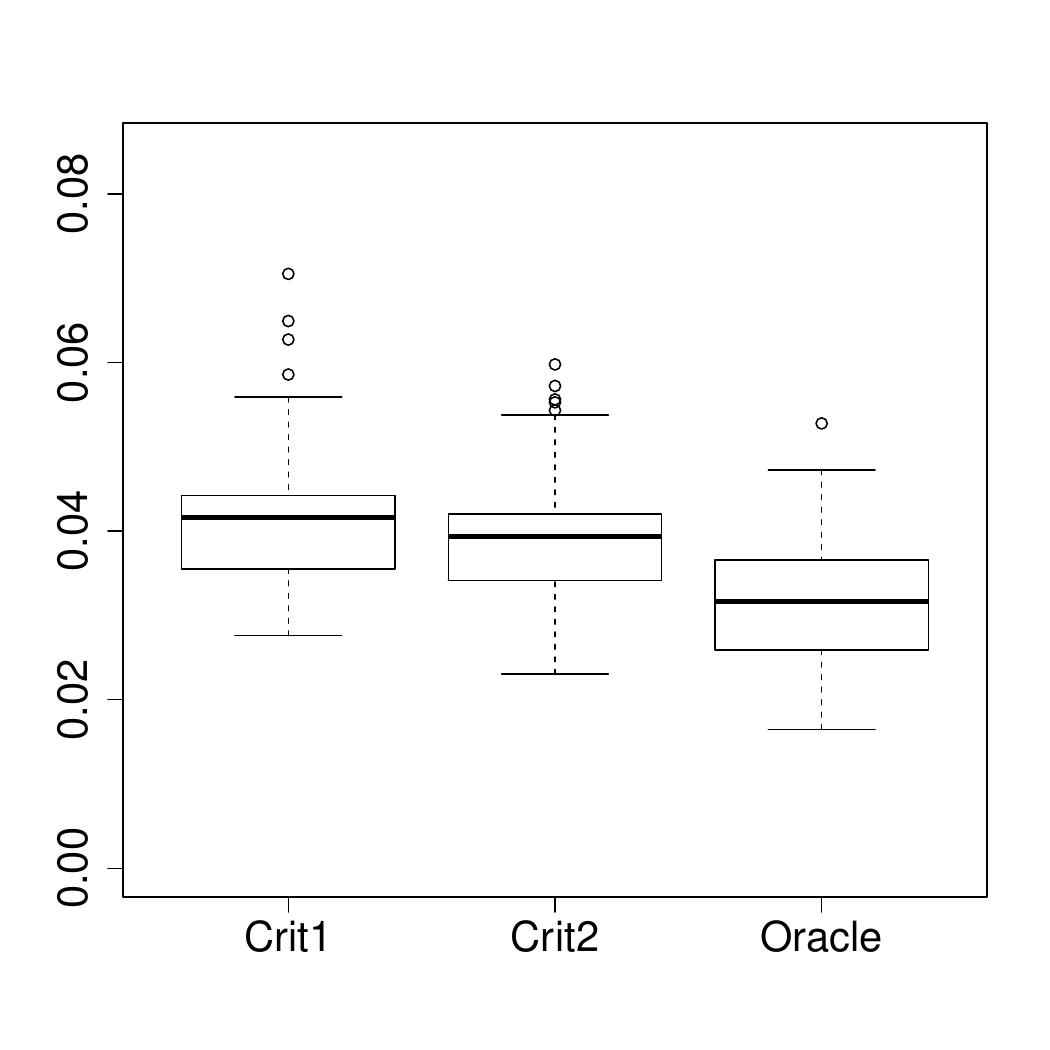}
& \parinc{\widthboxplot1}{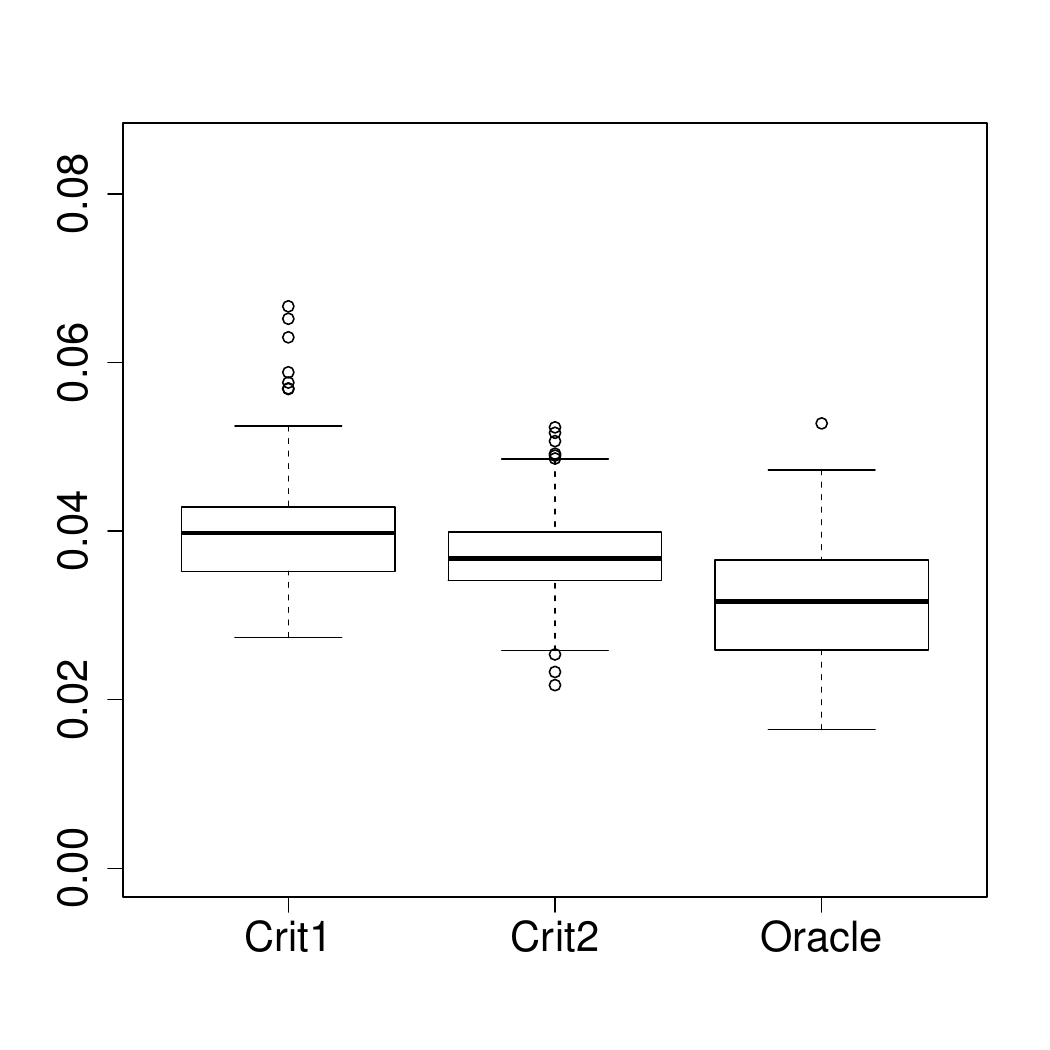} \\
\end{tabular}
\caption{\footnotesize\it Bandwidths and errors for the estimation of $h_1$ ($\Beta(2,2)$ distribution). \label{fig:box1}}
\end{figure}

Table \ref{tab1} and boxplots show that the performances of our estimators are close to those of the oracle. When comparing the first bandwidth selection rule Crit1  with the second one Crit2, one can observe that the performances of Crit2 are slightly better than those of Crit1 (see Table~\ref{tab1}). However, Crit2 is more time-consuming than Crit1. For both selection rules, we observe that the performances are slightly better when we increase the number of selected sub-samples $V$.  Remember that the larger the value of $V$, the larger the computation time whereas the performances are improved marginally. Hence, in practice it is reasonable to choose the first bandwidth selection rule Crit1 with $V=10$. {Finally, for both estimation of $g_1$ and $g_2$ according to the $Beta(2,2)$ distribution and the truncated normal distribution respectively, we illustrate in Figure \ref{fig:regline-error} the regression lines of the logarithm of the mean squared error of $\hat{g}_{n,\hat{\ell}_{Crit1}}$ versus the logarithm of the sample size, with $n\in \{1,000; \ 2,000; \ 5,000;
 \ 10,000; \ 20,000; \ 30,000\}$.} One can observe that the MSE's decrease as the sample size $n$ increases. This justifies the convergence of our estimators from the practical point of view.\par

\begin{figure}[!ht]
\def\widthboxplot1{4.75cm}
\centering
\begin{tabular}{cccc}
& $V = 10$ 	& $V = 25$ 	&$V = 40$ \\
\rotatebox[origin = c]{90}{Bandwidths}
& \parinc{\widthboxplot1}{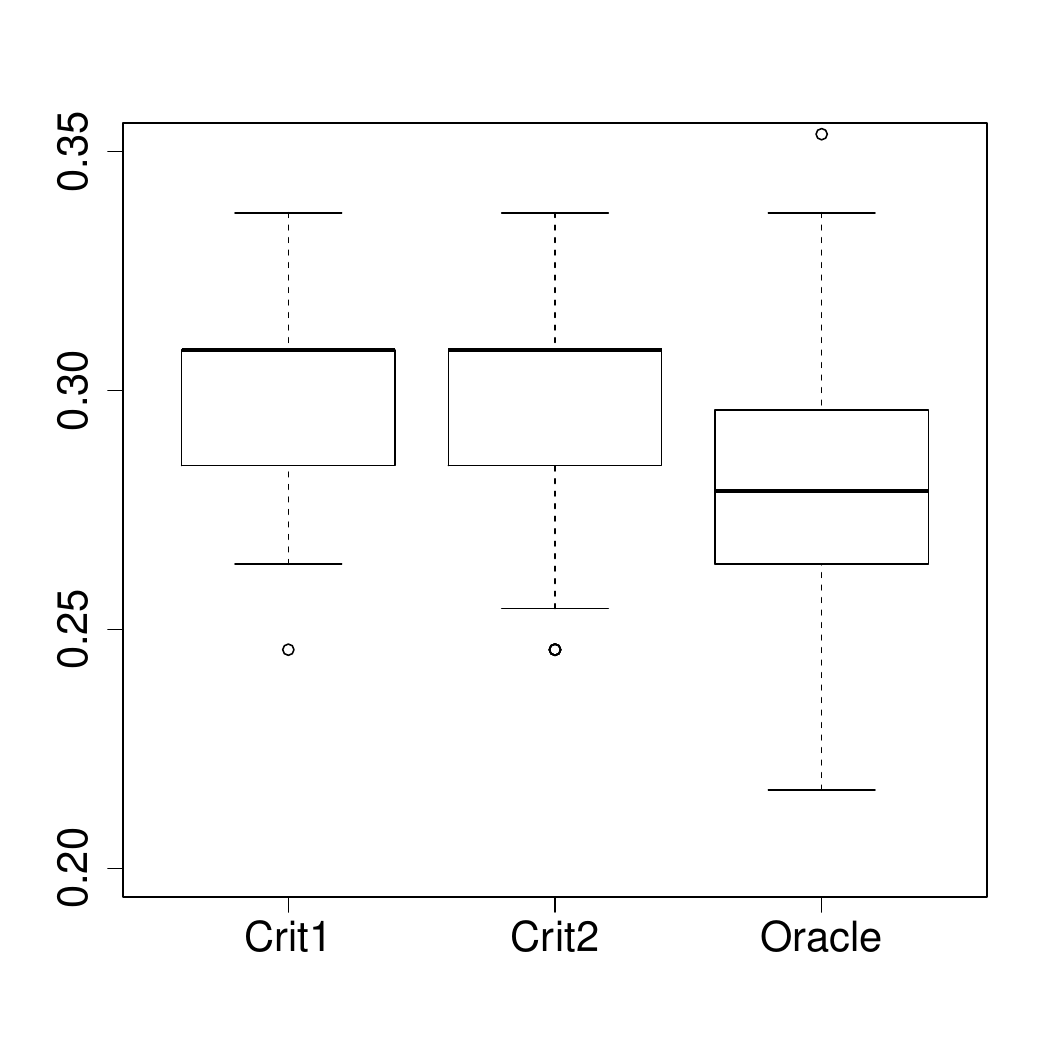}
& \parinc{\widthboxplot1}{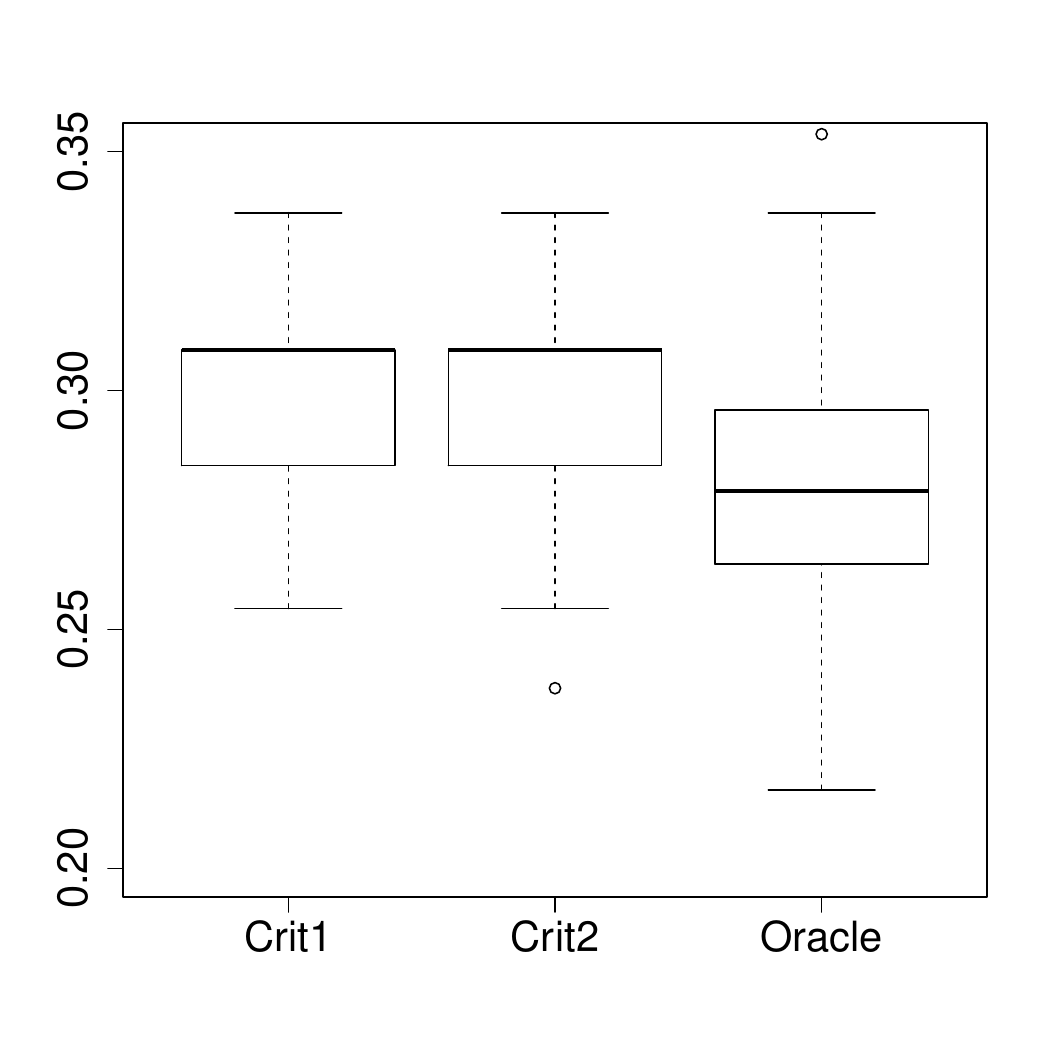}
& \parinc{\widthboxplot1}{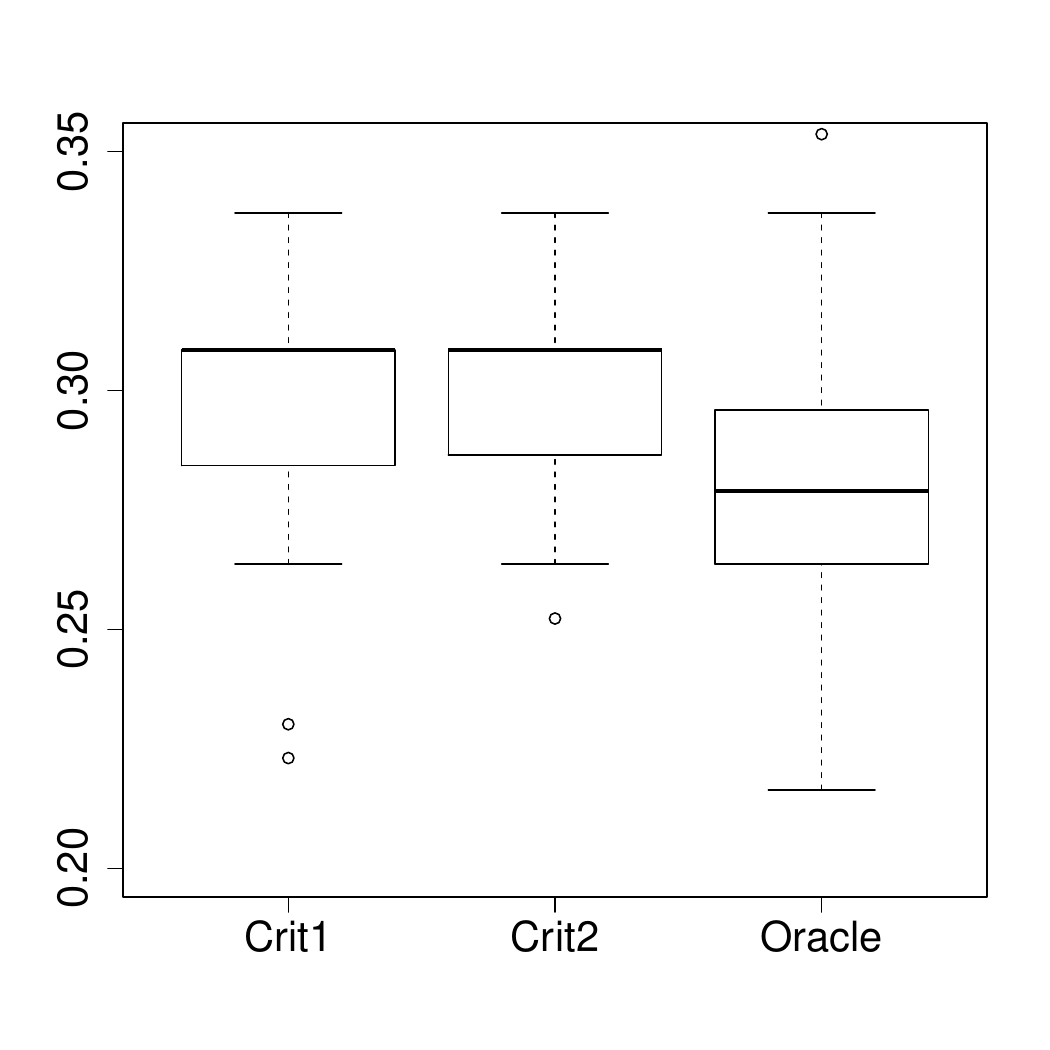} \\
\rotatebox[origin = c]{90}{Errors}
& \parinc{\widthboxplot1}{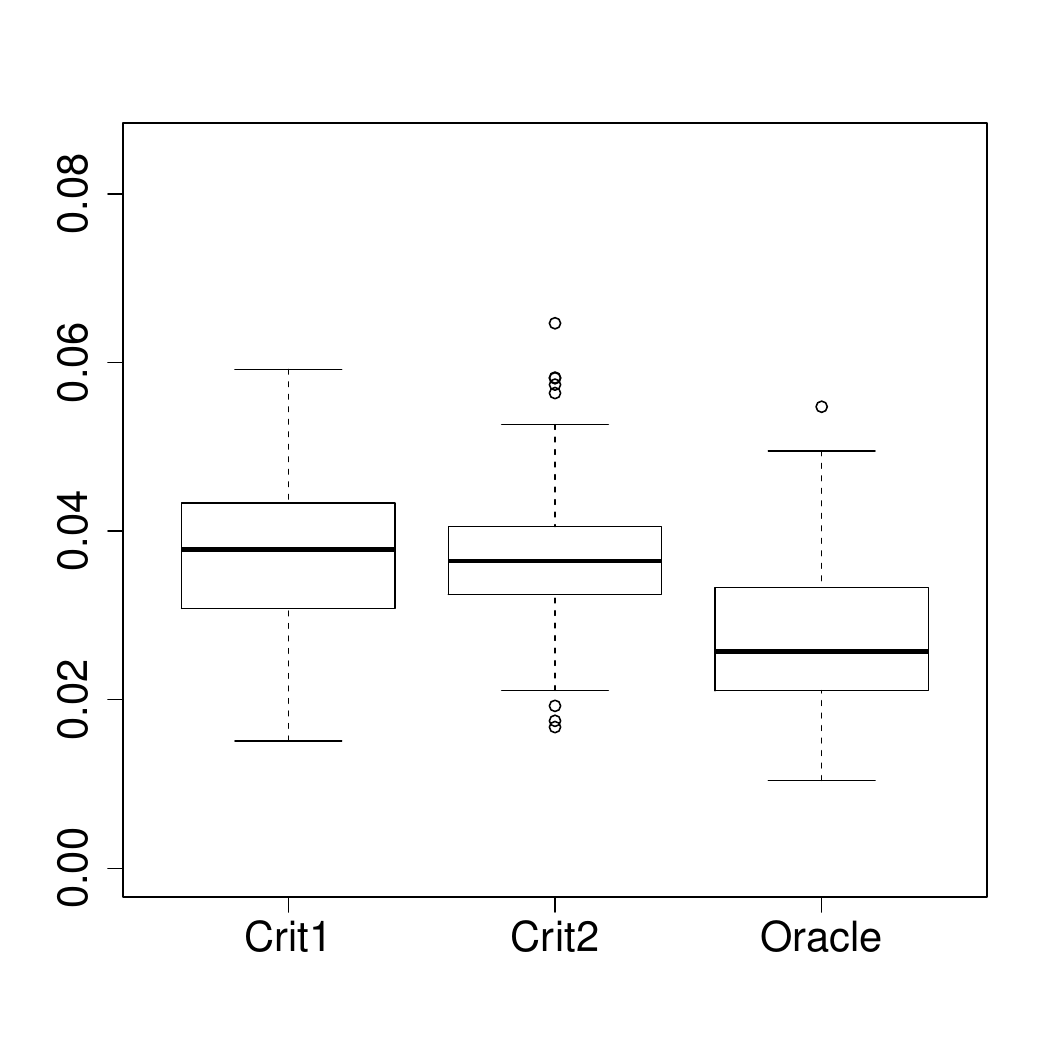}
& \parinc{\widthboxplot1}{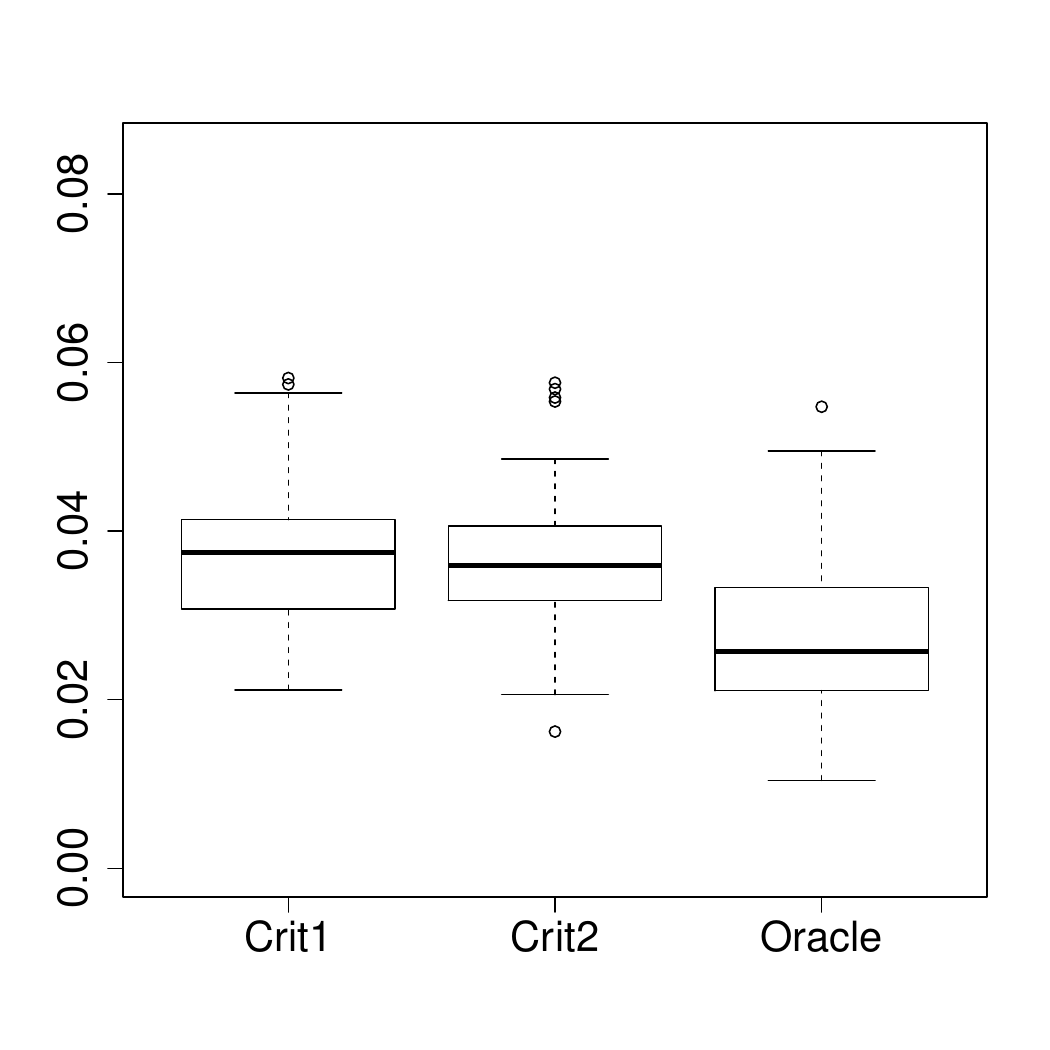}
& \parinc{\widthboxplot1}{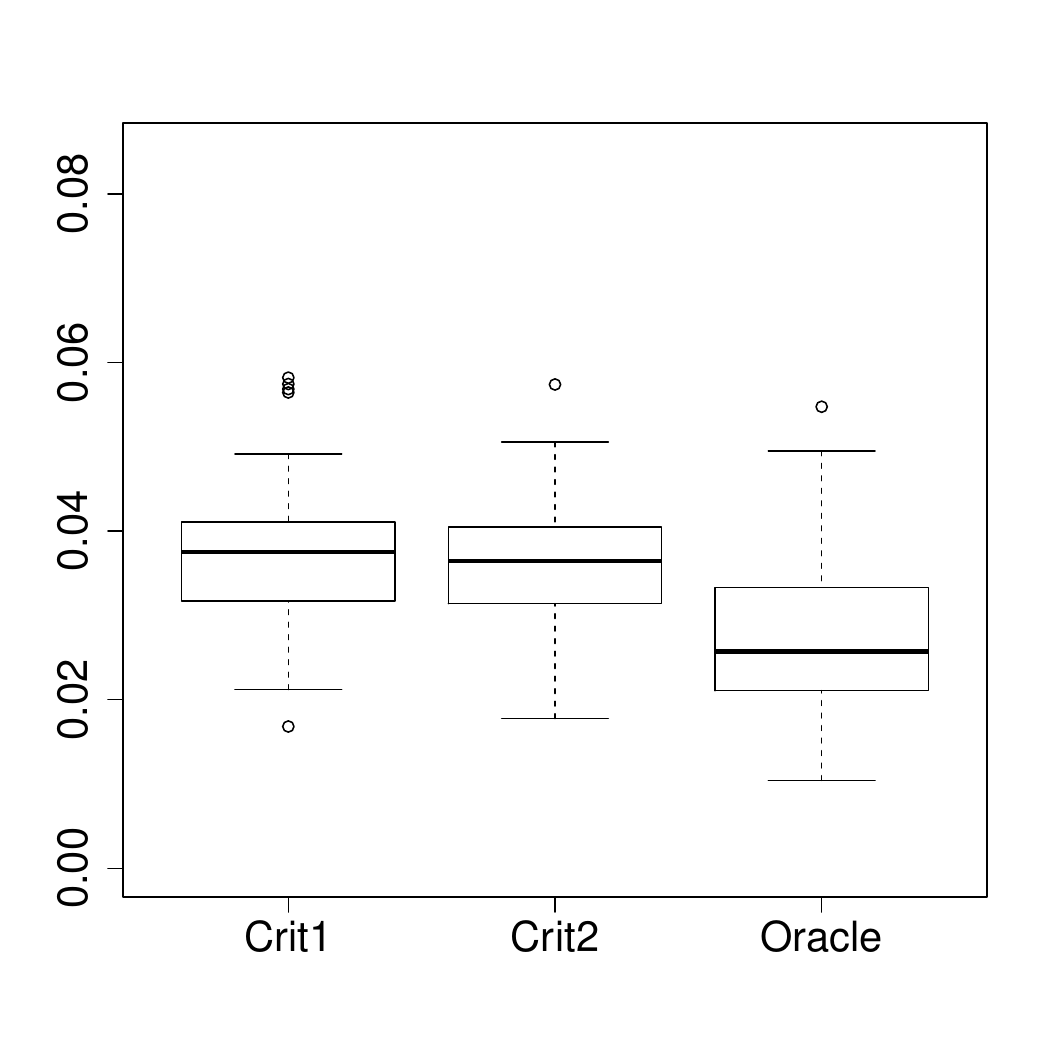} \\
\end{tabular}
\caption{\footnotesize\it  Bandwidths and errors for the estimation of $h_2$ (Truncated normal). \label{fig:box2}}
\end{figure}

\begin{figure}[!ht]
\begin{center}
\begin{tabular}{cc}
\includegraphics[scale=0.4]{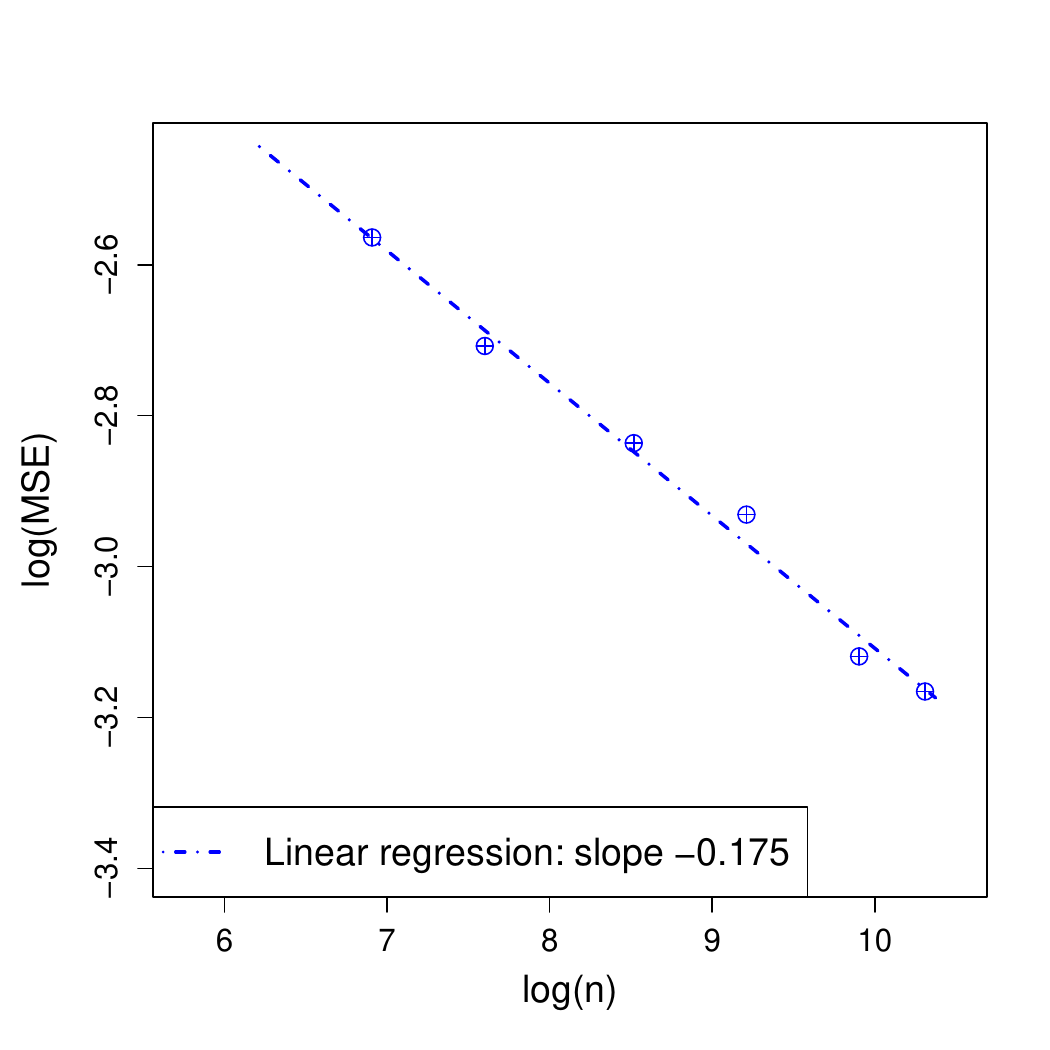} &
\includegraphics[scale=0.4]{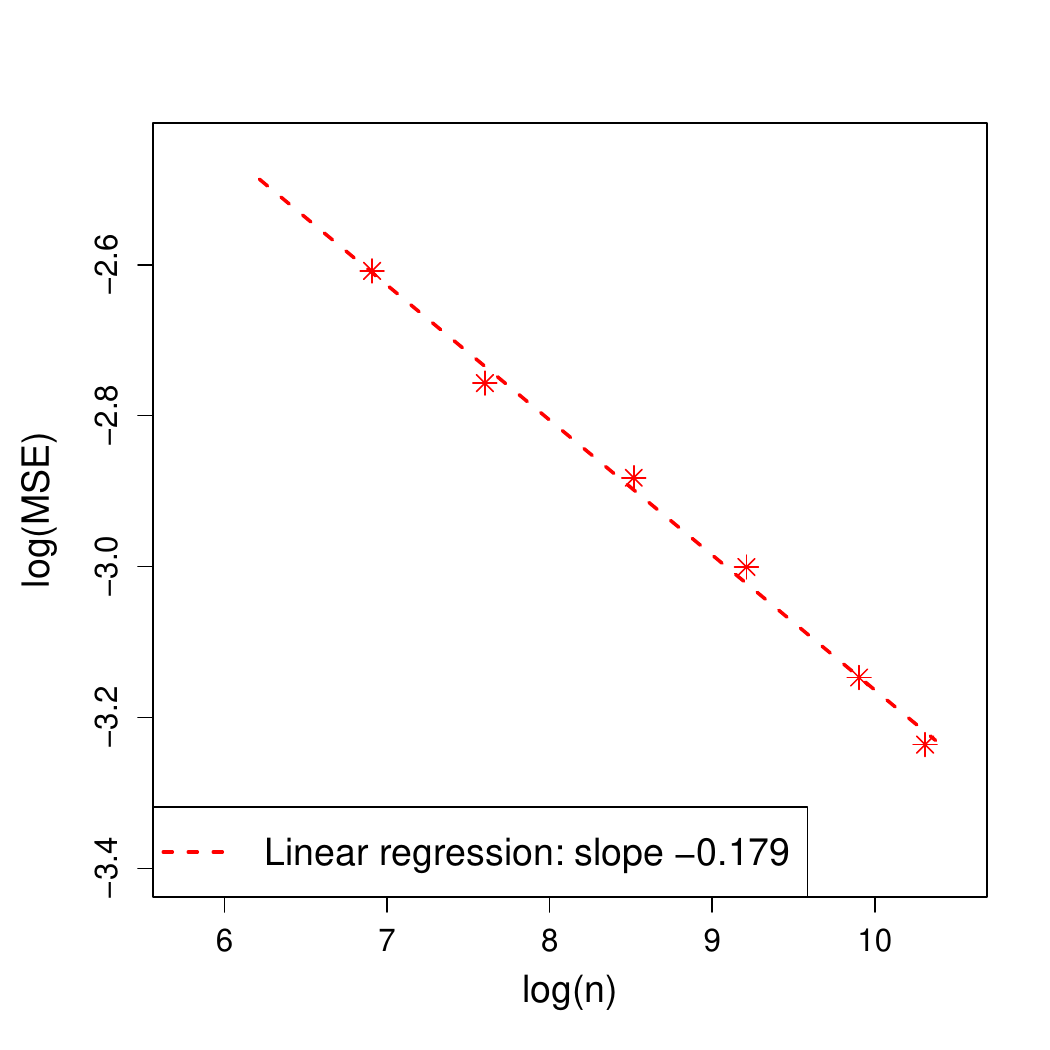} \\
(a) & (b)
\end{tabular}
\caption{\footnotesize The regression lines for log-mean square error for the reconstruction of $g_1$ (left, Eq: $y = -1.351 - 0.175x$) and the reconstruction of $g_2$ (right, Eq: $y = -1.375 - 0.179x$) versus $\log (n)$, with the sample size $n$ varying from $1,000$ to $30,000$.}
\label{fig:regline-error}
\end{center}\end{figure}

{\section{Conclusion}
Many statistical papers interested in aging phenomena for population of dividing cells concentrate on the estimation of the division rate (the constant $R$ in the present work) \cite{Bourgeron14,DoumicHoffmann12,Doumic12,Doumic17,Hoffmann14}. In these papers, the division rate is assumed to be a function that depends on certain quantities growing with time and that can be seen as ages of the cells (bad chemical contents, size...). The decrease of the division rate with respect to these quantities can be understood as senescence at the individual level. At the population level, lineages along which the distribution of these quantities tend to shift to higher values can be seen as aging or senescent lineages. In the present work, we investigate another aspect that is the kernel $h(.)$ ruling the division of the mother cell. In most of the previous works, \eg \cite{DoumicHoffmann12,Doumic12}, the mother cell divides into two identical cells. From the microscopic individual-based point of view, the senescence in lineages then arises from the sole randomness in the times of division. The asymmetry between daughter cells is however an important feature that has to be taken into account. Some recent developments have been made in this direction: by the authors \cite{Hoang2015,HoangPhD} and also by \cite{Bourgeron14,Doumic17} but from a deterministic point of view. \\
The inverse problem arising from the estimation of $h$, Equation \eqref{eq:FT-g}, is not a standard deconvolution problem. The function $M^*$ is closely linked to the function of interest $g$ (that is the function $h$ in a logarithmic scale) through Equation \eqref{eq:eigenvalue-change-variable}. Compared with classical deconvolution problems (\eg \eqref{deconv-classique} below or \cite{meister}), $M^*$ cannot be handled as an independent known noise. Firstly, the regularity and positivity of $M^*$ have to be studied and give way to involved and technical proofs. Secondly, one has to estimate $M^*$ in order to devise an estimator of $g^*$. This actually complicates the study of the proposed estimation procedure since we have to control the fluctuations of the empirical process $\widehat M_n^*$. {However, the theoretical study led in Section 3 allows us to circumvent these issues and to show consistency of our estimates. These nice  performances are also illustrated from a numerical point of view by the use of artificial data whose size is consistent with real ones.}\\
\tvc{In line with the previous comment, a natural extension would consist in deriving the bandwidth by using an alternative theoretical approach to the cross-validation type approach described in Section~\ref{sec:num-CV}. It would be natural to use, for instance, the Goldenshluger-Lepski methodology \cite{GL13} in the same spirit as \cite{ComteLacour} or the PCO methodology \cite{PCO_2016}, with the aim of deriving oracle inequalities. But such technical approaches require sharp controls of the variance of estimates and powerful concentration inequalities. Obtaining such results is beyond the scope of this paper but constitutes interesting challenges for future research.}\\
Of course, our probabilistic model could be enhanced by taking into account some observational noise and instead of the $X_i$'s we would observe,
\begin{equation}\label{deconv-classique}
Y_i= X_i + \varepsilon_i,
\end{equation}
with $X_i$ and the noise $\varepsilon_i$ being independent, which corresponds to a classical deconvolution problem if we are interested in recovering the density of $X_i$'s. The estimation of $h$ would then require to combine classical deconvolution technics and the approach of this paper.\\
\tvc{Furthermore, our study is in line with the references mentioned at the beginning of the conclusion: large populations close to their stationary states. This is justified by the exponentially fast convergence rate given in \eqref{eq:tpslong}. A possible direction for further research would be to  focus directly on the evolution problem \eqref{eq:growth-fragmentation}. Following ideas of Comte and Genon-Catalot \cite{ComteGenonCatalot}, it could be possible to study a projection estimator computed from the finite particle system on the compact time interval $[0, T]$.} These challenging inverse problems provide nice motivations for further work.
}

\clearpage


\appendix
\noindent{\LARGE \bf Appendix}\\[6pt]

\noindent This section is devoted to the proofs of the paper's results. $C$ is a constant whose value may change from line to line.

\section{Large population renormalization}\label{app:renorm}

Before proving the results of Section \ref{section:microscopique}, let us build the SDE satisfied by the process $(Z^K_t)_{t\geq 0}$.
Consider
$$\widetilde{Z}^K_t = \frac{1}{K}\sum_{i\in V^K_t} \delta_{(i,x_i(t))}$$the random point measure on $\mathcal{I}\times \R_+ =\cup_{\ell\geq 1}\N \times \{0,1\}^{\ell-1} \times \R_+$ with marginal measure $Z^K_t$ on $\R_+$, and that keeps track of the sizes and labels of the individuals in the population.\\

Let us consider as in Section \ref{section:microscopique} a sequence $(\widetilde{Z}^K_0)_{K\in \N^*}$ of random point measures on $\mathcal{I}\times \R_+$ such that the sequence of marginal measures $(Z^K_0)_{K\in \N^*}$ of the form \eqref{zk} converges to $\xi_0\in \mathcal{M}_F(\R_+)$ in probability and for the weak convergence topology on $\mathcal{M}_F(\R_+)$ and satisfies \eqref{eq:non-explosion-condition}.
Let also $Q(ds,di,d\gamma)$ be a Poisson point measure on $\rhoa_+\times \mathcal{E}:= \rhoa_+\times \mathcal{I} \times [0,1]$ with intensity $q(ds,di,d\gamma)=R ds\,n(di)\,h(\gamma)d\gamma$ where $n(di)$ is the counting measure on $\mathcal{I}$ and $ds$ and $d\gamma$ are Lebesgue measures on $\rhoa_+$. \\
We denote $\{\mathcal{F}_t \}_{t\ge 0}$ the canonical filtration associated with the Poisson point measure and the sequence $(\widetilde{Z}^K_0)_{K\in \N^*}$.

For a given $K\in \N^*$, it is possible to describe the measure $\widetilde{Z}^K_t$ at time $t$ by the following equation:
\begin{align}
\widetilde{Z}^K_t = & \sum_{i\in V_0^K} \delta_{(i,x_i(0)+\alpha t)} \notag \\
& +  \int_0^t \int_{\mathcal{E}}   \ind_{\{i\in V_{s_-}^K\}} \big( \delta_{(i0,\gamma x_{i}(s_-)+\alpha(t-s))}+\delta_{(i1,(1-\gamma) x_{i}(s_-)+\alpha(t-s))}-\delta_{(i,x_i(s_-) + \alpha(t-s)}\big)Q(ds,di,d\gamma),\label{eds:micro-dirac}
\end{align}
where the notation $x_i(s)$ stands for the size of the individual with label $i$ in the population $Z^K_s$ (we omit the dependence in $K$). This representation allows to take deterministic motions into account and the idea comes from \cite{Tran08,metztran}: we build the population at time $t$ by considering the contribution of the initial condition for this time $t$, and then the modifications due to all the divisions between times $0$ and $t$. The first term in the r.h.s. of \eqref{eds:micro-dirac} corresponds to the individuals alive at time $0$ with their sizes at time $t$ if they don't die. In the integral with respect to the Poisson point process, an atom at $(s,i,\gamma)$ of $Q$ corresponds to a `virtual' division event at time $s$ of the individual $i$ associated with the fraction $\gamma$. This event effectively takes place only if the individual with label $i$ is alive at time $s_-$. In this case, the Dirac masses corresponding to the mother at $t$ (at size $x_i(s_-) + \alpha(t-s)$) is replaced with the Dirac masses of the two daughters, at the size that they will have if they are still alive at time $t$ ($\gamma x_{i}(s_-)+\alpha(t-s)$ and $(1-\gamma) x_{i}(s_-)+\alpha(t-s)$).\\

The moment assumption \eqref{eq:non-explosion-condition} propagates to positive time and it is possible to show that for any $T>0$, (see \cite[Prop.3.2.5]{HoangPhD})
$$\sup_{K\in \N^*} \E\big(\sup_{t\in [0,T]} \langle Z^K_t,1\rangle^2 \big)<+\infty.$$
For every $K\in \N^*$ and every test function $f_s(x)=f(x,s) \in \mathcal{C}^{1,1}_b(\rhoa_+\times \rhoa_+,\rhoa)$, the stochastic process $(Z^K_t)_{t\in \R_+}$ satisfies:
\begin{align}
 \lefteqn{\langle Z_t^K,f_t\rangle = \langle Z_0^K, f_0\rangle + \int_0^t\int_{\R_+} \big(\pt_s f_s(x) + \alpha \pt_x f_s(x)\big)
 Z_s^K(dx)ds }\notag \\
&+ \frac 1K\int_0^t\int_{\mathcal{E}} \mathds{1}_{\{i\in V_{s-}^K\}}\Big(f_s\big(\gamma x_i(s-)\big)
+ f_s\big((1-\gamma)x_i(s_-)\big) - f_s\big(x_i(s_-)\big)\Big)Q(ds,di,d\gamma), \nonumber\\
= & \langle Z_0^K, f_0\rangle + M^{K,f}_t\label{app:SDE-renormalized}\\
& + \int_0^t\int_{\R_+} \Big(\pt_s f_s(x) + \alpha \pt_x f_s(x) + R \int_0^1 \big(f_s(\gamma x)+f_s((1-\gamma)x)-f_s(x)\big)h(\gamma)d\gamma \Big)
 Z_s^K(dx)ds ,\nonumber
\end{align}
where $(M^{K,f}_t)_{t\geq 0}$ is a square integrable martingale started at 0 with bracket:
\begin{equation}
\langle M^{K,f}\rangle_t= \frac{1}{K} \int_0^t \int_{\R_+}\int_0^1 R\big(f_s(\gamma x)+f_s((1-\gamma)x)-f_s(x)\big)^2 h(\gamma)d\gamma
 Z_s^K(dx)ds.\label{app:bracket}
\end{equation}
\textbf{The proof of Proposition \ref{prop:renorm1}} then follows the ideas in \cite{Tran08,Tran06} and are detailed in \cite{HoangPhD}. Equation \eqref{app:SDE-renormalized} corresponds to Equation \eqref{eq:martingale} in the main body.\\

\noindent \textbf{The proof of Theorem \ref{th:large-population-limit}} uses the martingale problem established in Prop. \ref{prop:renorm1} and standard arguments (see \eg \cite{EthierKurtz09,Joffe86,BansayeMeleard} and \cite[Th.1.1.8 and proof of Th.1.1.11]{Tran14}). Let us denote by $A^{K,f}$ the finite variation part of $\langle Z_t^K,f_t\rangle$:
\begin{equation}
A^{K,f}_t=\int_0^t\int_{\R_+} \left(\pt_s f_s(x) + \alpha \pt_x f_s(x) + R \int_0^1 \big(f_s(\gamma x)+f_s((1-\gamma)x)-f_s(x)\big)h(\gamma)d\gamma \right)
 Z_s^K(dx)ds.
\end{equation}
First, using the moment assumptions together with  \eqref{app:SDE-renormalized}-\eqref{app:bracket}, we can show that the sequences of real valued processes $(A^{K,f})_{K\in \N^*}$ and $(\langle M^{K,f}\rangle)_{K\in \N^*}$ are tight in $\D(\R_+,\R)$, which by the Aldous-Rebolledo condition imply the tightness of the sequence $(\langle Z^K_.,f\rangle)_{K\in \N^*}$ for all test function $f\in \Co^1_b(\R_+,\R)$. As a consequence, the sequence $(Z^K)_{K\in \N^*  }$ is tight in $\D(\R_+,(\mathcal{M}_F(\R_+),v))$, where $(\mathcal{M}_F(\R_+),v)$ means that the space of finite positive measures $\mathcal{M}_F(\R_+)$ is embedded with the topology of vague convergence.\\
Secondly, the limiting values $\bar{Z}$ to which subsequences of $(Z^K)_{K\in \N^*}$ converge vaguely, are continuous measure-valued processes of $\Co(\R_+,(\mathcal{M}_F(\R_+),w))$, where $\mathcal{M}_F(\R_+)$ is embedded with the weak convergence topology.\\
Thirdly, proceeding as in \cite[proof of Th.1.1.11]{Tran14} (see also \cite{Jourdain12,Meleard12}), we can prove that
$$\lim_{k\rightarrow +\infty}\lim_{K\rightarrow +\infty}\E\big(\sup_{t\leq T}\langle Z^K_t,\varphi_k\rangle\big)=0,$$
where the functions $\varphi_k$ are $\Co^2$ approximations of $\ind_{\{x\geq k\}}$ for $k\in \N$ and are defined by $\varphi_0(x)=1$ and for all $k\in \N^*$, $\varphi_{k}(x)=\psi(0 \vee (x-k+1)\wedge 1)$ with $\psi(x)=6x^5-15x^4+10x^3$. This ensures that for every subsequence of $(Z^K)_{K\in \N^*}$ that converges vaguely to a limiting process $\bar{Z}$, their masses converge in distribution to $\langle \bar{Z},1\rangle$, which provides the tightness in $(\mathcal{M}_F(\R_+),w)$ by a criterion due to M\'el\'eard and Roelly \cite{Meleard93}.\\
We can now establish that the limiting values to which subsequences of $(Z^K)_{K\in \N^*}$ converge in $\D(\R_+,(\mathcal{M}_F(\R_+),w))$ are solutions of \eqref{eq:limit-PDE} (see \cite{HoangPhD}). This integro-differential equation admits a unique solution. Indeed, let $\xi^1$ and $\xi^2$ be two solutions of \eqref{eq:limit-PDE} starting with the same initial condition $\xi_0$. For a test function $\varphi\in \Co^1_b(\R_+,\R)$ and $t>0$, setting
\begin{equation}f(x,s)=f_s(x)=\varphi(x+\alpha(t-s)),\label{app:test-function-f(x,s)}
\end{equation}we obtain that for $i\in \{1,2\}$,
\begin{align*}
\langle \xi^i_t,\varphi\rangle= \langle \xi_0,\varphi(.+\alpha t)\rangle +\int_0^t \int_{\R_+} \int_0^1 R \big(f_s(\gamma x)+f_s((1-\gamma)x)-f_s(x)\big)h(\gamma)d\gamma \, \xi_s^i(dx)\, ds.
\end{align*}
Substracting these two equations for $i=1$ and $i=2$, we obtain
$$\|\xi_t^1-\xi_t^2\|_{TV}\leq 3 R \|\varphi\|_\infty \int_0^t \|\xi_s^1-\xi_s^2\|_{TV} ds$$
where $\|.\|_{TV}$ stands for the total variation norm. Gronwall's inequality concludes the proof of uniqueness of the solution of \eqref{eq:limit-PDE}. Since the limiting value of $(Z^K)_{K\in \N^*}$ is unique, the sequence hence converges in $\D(\R_+,(\mathcal{M}_F(\R_+),w))$ to this unique solution. This concludes the proof of Theorem \ref{th:large-population-limit}.\\

\noindent \textbf{The proof of Proposition \ref{prop:PDE}} is detailed in \cite{HoangPhD} (see also \cite{Tran08}).
First, notice that if $\xi_0(dx)$ admits a density $n_0(x)$ with respect to the Lebesgue measure, then for any $t>0$, $\xi_t$ also admits a density. Indeed, for a function $\varphi\in \Co^1(\R_+,\R_+)$ with non-negative values, let us define the test function $f(x,s)$ as in \eqref{app:test-function-f(x,s)}. Then, neglecting the negative terms in the second line of \eqref{eq:limit-PDE} and using the symmetry of $h$ with respect to $1/2$:
\begin{align*}
\langle \xi_t,\varphi\rangle \leq & \int_{\R_+} \varphi(x+\alpha t)n_0(x)dx+ 2 R \int_0^t \int_{\R_+} \int_0^1  \varphi(\gamma x+\alpha(t-s)) h(\gamma)d\gamma\, \xi_s(dx)\, ds\nonumber\\
= & \int_{\alpha t}^{+\infty} \varphi(y)n_0(y-\alpha t) dy + 2 R \int_0^t  \varphi(\alpha(t-s))  \xi_s(\{0\}) \, ds\nonumber\\
& \hspace{1.1cm} + 2R \int_0^t \int_{\R_+\setminus \{0\}} \int_{\R} \ind_{(\alpha(t-s),x+\alpha(t-s))}(y) \varphi(y) h\left(\frac{y-\alpha(t-s)}{x}\right)\frac{dy}{x}\, \xi_s(dx) \, ds\nonumber\\
 = &  \int_{\alpha t}^{+\infty} \varphi(y)n_0(y-\alpha t) dy + 2 R \int_0^{\alpha t}  \varphi(y)  \xi_{t-\frac{y}{\alpha}}(\{0\}) \, \frac{dy}{\alpha}\nonumber\\
& \hspace{1.1cm} + 2R \int_0^{+\infty} \Bigg\{\int_0^t \int_{\R_+\setminus \{0\}}  \ind_{(\alpha(t-s),x+\alpha(t-s))}(y) \frac{1}{x}h\left(\frac{y-\alpha(t-s)}{x}\right) \xi_s(dx) \, ds\Bigg\} \varphi(y)dy.
\end{align*}Since $\xi_t$ is dominated by a nonnegative measure absolutely continuous with respect to the Lebesgue measure on $\R_+$ it follows that $\xi_t$ admits itself a density. \tvc{Let us denote by $n(t,x)$ the density of $\xi_t$ with respect to the Lebesgue measure $dx$ on $\R_+$. Then, for a non-negative test function $f\in \Co_b^1(\R_+,\R_+)$ depending only on $x$ and using the symmetry of $h$, \eqref{eq:limit-PDE} becomes:
\begin{align}
\int_{\R_+} f(x)n(t,x)dx= & \int_{\R_+} f(x)n_0(x)dx+\int_0^t \int_{\R_+} \alpha f'(x) n(s,x)dx\ ds \nonumber\\
+ &
 \int_0^t \int_{\R_+} 2R \int_0^1 f(\gamma x)h(\gamma)d\gamma\  n(s,x) dx\ ds - R  \int_0^t \int_{\R_+} f(x) n(s,x)dx\ ds\nonumber\\
 = & \int_{\R_+} f(x)n_0(x)dx +\int_0^t \int_{\R_+} \alpha f'(x) n(s,x)dx\ ds  \nonumber\\
 + & \int_0^t \int_{\R_+} f(x) \times 2R \int_x^{+\infty} h\Big(\frac{x}{y}\Big) n(s,y) \frac{dy}{y}\ dx\ ds - R  \int_0^t \int_{\R_+} f(x) n(s,x)dx\ ds,\label{etape13}
 \end{align}where we used Fubini's Theorem for the third term in the right hand side. For the second term in the r.h.s., integrating by part gives:
\begin{align}
\int_0^t \int_{\R_+} \alpha f'(x) n(s,x)dx\ ds  = &  \int_0^t \Big\{ \big[\alpha f(x)n(s,x)\big]_0^{+\infty}  -  \int_{\R_+} \alpha f(x) \partial_x n(s,x)dx\Big\} ds \nonumber\\
= &  -  \alpha f(0)\int_0^t n(s,0) ds - \int_0^t \int_{\R_+} \alpha f(x) \partial_x n(s,x)dx\ ds.\label{etape14}
\end{align}Gathering \eqref{etape13} and \eqref{etape14} that are true for any test function $f$ and time $t$, we can identify the equations satisfied by $(n(t,x), x\in \R_+, t>0)$. We find that $n(t,0)=0$ for every $t\geq 0$ and that $(n(t,x), x\in \R_+, t>0)$ solves in distribution sense \eqref{eq:growth-fragmentation} for which uniqueness of the solution holds (\eg \cite[Theorem 4.3 p.90]{Perthame07}). }\\


\noindent \textbf{The proof of Proposition \ref{prop:renorm2}} is a particular case of \cite[Th.4.6 p. 94]{Perthame07} based on Krein-Rutman theorem (\eg \cite[Th.6.5 p.175]{Perthame07}) (see also \cite{Doumic10}). In the case that we consider, the proof can be simplified compared with \cite{Perthame07}. \\
Let us consider the eigenelements $(\lambda, N,\phi)$ associated with \eqref{eq:growth-fragmentation}, \ie the solution of:
\begin{equation}\label{eq-app:eigenvalue}
\begin{cases}
 &\alpha \pt_x N(x) + (\lambda + R)N(x) = 2R\int_0^1 N\left(\frac{x}{\gamma}\right)h(\gamma) \dfrac{d\gamma}{\gamma} ,\quad x\ge 0, \\
 &N(0) = 0,\quad \int N(x)dx =1,\quad N(x) \ge 0, \quad \lambda >0,\\
 & \alpha \pt_x \phi(x)-(\lambda+R) \phi(x)=-2R \int_0^1 \phi(\gamma x)h(\gamma) d\gamma,\quad x\ge 0, \\
 & \phi(x)\geq 0,\quad \int_0^{+\infty}\phi(x)N(x)dx=1.
\end{cases}
\end{equation}
It is clear that $\lambda=R$ and $\phi\equiv 1$ solve the third equation of \eqref{eq-app:eigenvalue}. Because the first line is linear in $N$, we can forget for the proof the condition $\int N(x)dx=1$: if there exists a nonnegative integrable solution, we can renormalize it.\\


\noindent \textit{Step 1}: Let us consider the following auxiliary PDE, for a constant $\mu>0$ and two functions $f\in \Co(\R_+,\R_+)$, and $M\in \L^1(\R_+,\R)\cap \Co(\R_+,\R_+)$:
\begin{equation}\label{eq-app:eigenvalue4}
\alpha \pt_x N(x) + (\mu+R) \, N(x) - 2R\int_0^1 M\left(\frac{x}{\gamma}\right)h(\gamma) \dfrac{d\gamma}{\gamma} = f(x) ,\quad x\ge 0 \quad ; \quad N(0)=0.
\end{equation}
Equation \eqref{eq-app:eigenvalue4} is a first order ODE that can be solved with the variation of constant method. It admits a unique solution, that we denote $T(M)$:
$$T(M)(x)=\frac{1}{\alpha}\int_0^x e^{-\frac{\mu+R}{\alpha}(x-y)}  \Big(2R\int_0^1 M\left(\frac{y}{\gamma}\right)h(\gamma) \dfrac{d\gamma}{\gamma} + f(y)\Big)  dy.$$
Consider $M_1$ an $M_2\in \L^1(\R_+,\R)\cap \Co(\R_+,\R_+)$. Then, for $x\geq 0$:
\begin{align}
|T(M_1)(x)-T(M_2)(x)|
\leq & \frac{2R}{\alpha} \int_0^1 \int_0^{x/\gamma} e^{-\frac{\mu+R}{\alpha}(x-\gamma z)}  | M_1(z)-M_2(z) | h(\gamma) dz\, d\gamma \nonumber \\
\leq & \frac{2R}{\mu+R}  \big(1-e^{-\frac{(\mu+R)x}{\alpha}}\big)  \int_0^1 \frac{h(\gamma)}{\gamma}   d\gamma \| M_1-M_2\|_\infty.\label{etape}
\end{align}Provided the integral in the term above is finite, then for
$\mu>2R \int_0^1 h(\gamma)/\gamma \, d\gamma -R$, the map $M\in \L^1(\R_+,\R)\cap \Co(\R_+,\R_+) \mapsto T(M)\in \L^1(\R_+,\R)\cap \Co(\R_+,\R_+)$ is a contraction. Thus it admits a unique fixed point that is the unique solution of
\begin{equation}\label{eq-app:eigenvalue3}
 \alpha \pt_x N(x) + (\mu+R) \, N(x) - 2R\int_0^1 N\left(\frac{x}{\gamma}\right)h(\gamma) \dfrac{d\gamma}{\gamma} = f(x) ,\quad x\ge 0\quad ;  \quad N(0) = 0.
\end{equation}
\noindent \textit{Step 2}: The map $A$ that associates to $f\in \Co(\R_+,\R_+)\cap \L^1(\R_+,\R_+)$ the unique corresponding solution of \eqref{eq-app:eigenvalue3} is thus well defined. Following the path of \cite[Section 6.6.2]{Perthame07}, we can show that this map is linear, continuous (with computation similar to \eqref{etape}) and strongly positive. Finally, the boundedness of $N$ implies the boundedness of $\pt_x N$, with norms controlled by $\|f\|_\infty$. This allows to use Arzela-Ascoli theorem to obtain the compactness of the map $A$. We can then use Krein-Rutman theorem (using similar truncations as in \cite{Perthame07}) to obtain that the spectral radius of $A$, $\rho(A)$, is a positive simple eigenvalue associated with a positive eigenvector satisfying:
\begin{equation}\label{eq-app:eigenvalue5}
 \alpha \pt_x N(x) + \left(\mu+R-\frac{1}{\rho(A)}\right) \, N(x) - 2R\int_0^1 N\left(\frac{x}{\gamma}\right)h(\gamma) \dfrac{d\gamma}{\gamma} = 0 ,\quad x\ge 0 \quad ;  \quad N(0) = 0.
\end{equation}
The fact that $\lambda:=\mu+R-\frac{1}{\rho(A)}$ is equal to $2R$ is a consequence of integrating the direct equation against the the adjoint eigenvector (here $\phi \equiv 1$) and using that $\int N(x)dx=1$.\\

\noindent \textit{Step 3}: The computation to establish the speed of convergence of $n(t,x)e^{-Rt}$ to $\rho N(x)$ stated in \eqref{eq:tpslong}, are obtained by generalizing the proof of \cite[Th.4.2 p.88]{Perthame07} (see also \cite{perthameryzhik}). Define $g(t,x)=n(x,t)e^{-Rt}- \rho N(x)$, $G(t,x)=\int_0^x g(t,y)dy$ and $K(t,x)=\partial_t G(t,x)$.
One can write the PDEs satisfied by $g$ and $G$. The PDE for $G$ implies that
$\partial_t \int_0^{+\infty} \big|G(t,x)e^{Rt}\big| dx \leq 0$. As a consequence,
\begin{equation}
\int_0^{+\infty} |G(t,x)|dx\leq e^{-Rt} \|G_0\|_1. \label{etape3}
\end{equation}
From the PDE of $g$, $K(0,x)=\partial_t G(t,x)|_{t=0}=2R \int_0^1 G_0(x/u)h(u)du-2RG_0(x)-\alpha g_0(x)$.
Proceeding similarly as for $G$, we show that
\begin{align}
\int_0^{+\infty} |K(t,x)|dx\leq & e^{-Rt} \int_0^{+\infty} |K(0,x)|dx \leq  e^{-Rt} \big(3R \|G_0\|_1+\alpha \| g_0\|_1\big).\label{etape4}
\end{align}Plugging \eqref{etape3} and \eqref{etape4} in the PDE of $g$ (where we notice that $g(t,x)=\partial_x G(t,x)$), we obtain
the result announced in the proposition.

\section{Proof of Proposition \ref{lem:regM}}\label{app:lemmas}

\begin{proof}[Proof of Proposition \ref{lem:regM} (i)]
Let $\epsilon >0$ to be chosen small enough. Since $N$ is a probability density, we have for $\nu\leq (\nu_0+2)\wedge ([\beta]+1)$:
\[
\int_0^{+\vc} x^{-\nu} N(x)dx \le \int_0^\epsilon x^{-\nu} N(x)dx + \frac{1}{\epsilon^{\nu}}.
\]
Hence, it remains to prove
\[
\int_0^\epsilon x^{-\nu} N(x)dx <+\vc.
\]
We follow and adapt the steps of the proof of Theorem 1 in Doumic and Gabriel \cite{Doumic10}.
Integrating both side of equation \eqref{eq:eigenvalue} between $0$ and $x_0 \leq x$, we get:
\begin{equation}\label{eq:tmp00}
\alpha N(x_0) + 2R \int_0^{x_0} N(y)dy = 2R \int_0^{x_0}\int_0^{+\infty} N(y)h\left(\dfrac{z}{y}\right)\dfrac{dy}{y}dz.
\end{equation}
Thus,
\[
\alpha N(x_0) \le 2R\int_0^{x_0}\int_0^{+\infty} N(y)h\left(\dfrac{z}{y}\right)\dfrac{dy}{y}dz \le
2R \int_0^{x}\int_0^{+\infty} N(y)h\left(\dfrac{z}{y}\right)\dfrac{dy}{y}dz.
\]
Let us define:
\[
f: x \mapsto \underset{x_0\in (0,x]}{\sup} N(x_0),
\]
then we have for all $x$
\begin{equation}\label{eq:tmp01}
f(x) \le \frac{2R}{\alpha}\int_0^x \int_0^{+\infty} N(y)h\left(\dfrac{z}{y}\right)\dfrac{dy}{y}dz.
\end{equation}
Recall Assumption \ref{assump:moment-h}. Using a Taylor expansion, it implies that for any $t \in (0,1)$,
\begin{equation}
\int_0^t h(x) dx \le C \int_0^t x^{(\nu_0+1)\wedge [\beta]} dx\leq Ct^{(\nu_0+2)\wedge ([\beta]+1)}\leq C t^\nu
\end{equation}by choice of $\nu\leq (\nu_0+2)\wedge ([\beta]+1)$.
Then, we have for all $x < \epsilon$:
\begin{align*}
f(x) &\le \frac{2R}{\alpha}\int_0^{+\infty}N(y)dy \int_0^x h\left(\dfrac{z}{y}\right)\dfrac{dz}{y} \\
&\le \frac{2R}{\alpha}\int_0^{+\infty} N(y)\min\left(1,C\frac{x^\nu}{y^\nu}\right)dy \\
&\le \frac{2R}{\alpha}\left(\int_0^x N(y)dy + C\int_x^{\epsilon} N(y)\frac{x^\nu}{y^\nu}dy
+ C\int_{\epsilon}^{+\infty} N(y)\frac{x^\nu}{y^\nu}dy \right) \\
&\le \frac{2R}{\alpha}\left(\int_0^x \underset{z\in(0,x]}{\sup}N(z)dy + Cx^\nu \int_x^{\epsilon} \underset{z\in (0,y]}{\sup}N(z) \frac{dy}{y^\nu} \right) +  \left(\frac{2CR}{\alpha}\int_{\epsilon}^{+\vc}\frac{N(y)}{y^\nu}dy\right)x^\nu  \\
&\le \frac{2R\epsilon}{\alpha}f(x) + \frac{2CRx^\nu}{\alpha}\int_x^{\epsilon} \frac{f(y)}{y^\nu}dy + Kx^\nu,
\end{align*}
with $K = \dfrac{2CR}{\alpha \epsilon^\nu}$.
We choose $\epsilon$ such that
\[
0 < \epsilon < \frac{\alpha}{2R}.
\]
and by setting $F(x) = x^{-\nu}f(x)$, we get
\begin{align}
F(x) \le \frac{K}{1 -  \frac{2R\epsilon}{\alpha}}  + \frac{2CR}{\alpha - 2R\epsilon}\int_x^{\epsilon} F(y)dy. \label{eq:tmp02}
\end{align}
Then, applying Gronwall's inequality to \eqref{eq:tmp02}, we obtain
\[
F(x) \le \frac{K}{1 - \frac{2R\epsilon}{\alpha}}\exp\left(\frac{2CR\epsilon }{\alpha - 2R\epsilon} \right)= :\tilde{C},\quad \forall x\in [0,\epsilon]
\]
and
\[
x^{-\nu}N(x) \le \tilde{C}, \quad\forall x\in [0,\epsilon].
\]
We finally obtain
\[
\int_0^\epsilon x^{-\nu}N(x)dx \le \tilde{C}\epsilon < + \vc.
\]
This ends the proof of Proposition \ref{lem:regM}(i).

\end{proof}

\begin{proof}[Proof of Proposition \ref{lem:regM}(ii)]Let us first notice that by the fixed point theorem in the proof of Proposition \ref{prop:renorm2}, $N$ is continuous as uniform limit of a sequence of continuous functions. Let us show that under Assumption \ref{assump:moment-h} the map
$$\Phi \, :\, x \in (0,+\infty)\mapsto \int_x^{+\infty} N(y)h\Big(\frac{x}{y}\Big)\frac{dy}{y}$$
is of class $\Co^{[\beta]}$ on $(0,+\infty)$. We proceed by induction, and start by computing the first derivative of $\Phi$ for $x>0$.
\begin{align}
\frac{\Phi(x+\varepsilon)-\Phi(x)}{\varepsilon}= &  -\frac{1}{\varepsilon} \int_x^{x+\varepsilon} N(y) h\Big(\frac{x}{y}\Big)\frac{dy}{y}
+ \int_{x+\varepsilon}^{+\infty} N(y) \frac{y}{\varepsilon}\Big[h\Big(\frac{x+\varepsilon}{y}\Big)-h\Big(\frac{x}{y}\Big)\Big]\frac{dy}{y^2}\nonumber\\
\rightarrow_{\varepsilon\rightarrow 0} & -\frac{N(x)h(1)}{x}+\int_x^{+\infty}N(y)h'\Big(\frac{x}{y}\Big)\frac{dy}{y^2}=\int_x^{+\infty}N(y)h'\Big(\frac{x}{y}\Big)\frac{dy}{y^2},\label{etape2}
\end{align}since $h(1)=h(0)=0$ by Assumption \ref{assump:moment-h}. This shows that $\Phi$ is of class $\Co^1$. Plugging this information into \eqref{eq:eigenvalue}, it follows that $\partial_x N$ is continuous, and hence $N$ is of class $\Co^1$ and thus $\partial_x N$ also. This entails from the computation of $\Phi'$ that $\Phi$ is of class $\Co^2$.\\
Suppose that we have computed the successive derivatives of $\Phi$ up to $k-1$ and that we have shown that $N$ is of class $\Co^{k-1}$ for $k\leq [\beta]\wedge \nu_0$. Then, since the successive derivatives of $h$ at 0 vanish by Assumption \ref{assump:moment-h},
$$\Phi^{(k)}(x)=\int_x^{+\infty}N(y)h^{(k)}\Big(\frac{x}{y}\Big)\frac{dy}{y^{k+1}}.$$
Since $N$, $h$ and their derivatives are bounded functions, the latter integrals are always finite for $x>0$. This implies that $\Phi$ is of class $\Co^k$ and that using this information in \eqref{eq:eigenvalue}, $\partial_x N$ is of class $\Co^{k-1}$ entailing that $N$ is of class $\Co^k$. As the computation of the first derivative of $\Phi$ shows, we are limited by the regularity of $h$. \\
So we finally have that $x\mapsto N(x)$ is of class $\Co^{[\beta]}$, and thus, $u\mapsto M(u)$ is also of class $\Co^{[\beta]}$. \\

Take $k\leq [\beta]$. That $M$ is of class $\Co^{[\beta]}$ implies that $(\ic \xi)^{k} M^*(\xi)$ is the Fourier transform of $M^{(k)}$ and bounded on $\R$ provided we additionally prove that the derivatives of $M$ up to the order $k$ are integrable. Since $M(u)=e^{u}N(e^u)$, $M^{(k)}$ is a linear combination of terms of the form $e^{(\ell +1)u}N^{(\ell)}(e^u)$ with $\ell\leq k$. We thus have to check the finiteness, for all $\ell\leq k$, of:
\begin{equation}\label{etape10}
\int_\R e^{(\ell+1) u}\big|N^{(\ell)}(e^u)\big|du =\int_0^{+\infty} v^{\ell} \big|N^{(\ell)}(v)\big|dv.
\end{equation}
It is known (as a direct adaptation of \cite[Th.4.6 p.95]{Perthame07} for example) that as soon as $\mu<R/\alpha$,
\begin{equation}\label{momentexpN}N(x)e^{\mu x}\in \L^1(\R_+,\R_+) \cap \L^\infty(\R_+,\R_+).
\end{equation}
Assume that for some $\ell< k$, we have proved that $\int_0^{+\infty}e^{\mu x}\big|N^{(\ell)}(x)\big|dx<+\infty$ for $\mu<R/\alpha$. Let us prove that this also holds for $\ell+1$, which would entail \eqref{etape10}.
Deriving \eqref{eq:eigenvalue} $\ell$ times, multiplying by $e^{\mu x}$ and integrating again in $x\in(0,+\infty)$, we obtain:
\begin{align}\alpha \int_0^{+\infty} e^{\mu x}\big| N^{(\ell+1)}(x) \big|dx \leq & 2R \int_0^{+\infty} e^{\mu x} \big| N^{(\ell)}(x) \big| dx +   2R \int_0^{+\infty} e^{\mu x} \Big|\int_x^{+\infty} h^{(\ell)}\Big(\frac{x}{y}\Big) \frac{N(y)}{y^{\ell+1}}dy\Big| dx\nonumber\\
\leq &  2R \int_0^{+\infty} e^{\mu x} \big| N^{(\ell)}(x) \big| dx + 2R \int_0^{+\infty} \frac{N(y)e^{\mu y}}{y^{\ell}} \|h^{(\ell)}\|_1 dy.\label{etape12}
\end{align}
By the induction assumption, the first term in the right hand side is finite. Because $h^{(\ell)}$ is a continuous function on $[0,1]$, $\|h^{(\ell)}\|_1$ is finite. Using that $N(x)e^{\mu x}\in \L^1(\R_+,\R_+)$ implies that the second term is integrable at $+\infty$. Point (i) of Proposition \ref{lem:regM} ensures the integrability at 0. Thus, the right hand side of \eqref{etape12} is finite. The use of point (i) of Proposition \ref{lem:regM} explains why $[\beta]\wedge (\nu_0+3)$ appears in the announced result.\\
The finiteness of $\int_0^{+\infty}e^{\mu x} \big|N^{(\ell)} (x)\big|dx$, for $\ell\leq [\beta]\wedge (\nu_0+3)$, is thus proved by recursion, implying the finiteness of the terms in \eqref{etape10} and concluding the proof.

\end{proof}

\begin{proof}[Proof of Proposition \ref{lem:regM}(iii)]The proof is divided into several steps.\\

\noindent \textit{Step 1:} First, notice that
$$M^*\ :\ \xi=\xi_1+\ic \xi_2 \mapsto \int_{-\infty}^{+\infty} e^{\ic x\xi}M(x)dx=\int_0^{+\infty} e^{\ic \xi \log(y)} N(y)dy =\int_0^{+\infty} e^{\ic \xi_1 \log(y)} y^{-\xi_2} N(y)dy.$$
We first prove that $M^*$ has isolated zeros in $\{\xi\in \mathbb{C} : \Im(\xi)<1\}$.\\
Because $N$ is such that $e^{\mu x}N(x)\in \L^\infty(\R_+,\R_+) \cap \L^1(\R_+,\R_+)$ for $\mu<R/\alpha$ (see \cite[p.95]{Perthame07}), $M^*$ is well defined on the half plane $\{\xi\in \mathbb{C} : \Im(\xi)\leq 0\}$. By Proposition \ref{lem:regM}(i), $N(y)/y$ is integrable on the neighborhood of zero. Hence, for $\xi_2<1$, so is $y^{-\xi_2}N(y)$. As a consequence, $M^*$ is analytic on the lower half-plane $\{\xi\in \mathbb{C} : \Im(\xi)<1\}$ which contains the real axis. The derivative of the integrand with respect to $\xi$ has modulus $|\log(y)|\, y^{-\xi_2} N(y)$. The latter function is upper bounded, when $\xi_2<1$ and when $y<1$, by $\frac{N(y)}{y}$ which is integrable on the neighborhood of zero by Proposition \ref{lem:regM}(i). It follows from the results on integrals with parameters that the extension of $M^*$ to the complex plane is holomorphic on $\{\xi\in \mathbb{C} : \Im(\xi)<1\}$. Because $M^*$ is not the null function, its zeros in $\{\xi\in \mathbb{C} : \Im(\xi)<1\}$ have to be isolated.\\

\noindent \textit{Step 2:} Let us now show that $M^*$ admits no zero on the real line. First we link the Fourier transform to the Mellin transform of $N$. Recall that $M^*$ is related to the solution $N$ of \eqref{eq:eigenvalue} where the right hand term is a multiplicative convolution of $x\mapsto N(x)/x$ and $h$. A natural way of treating multiplicative deconvolution is by using Mellin transform (e.g. \cite{Epstein,Titchmarsh}).  \\

It is natural to set
\begin{equation}\label{def:psi}\psi(x)=\frac{N(x)}{x},\quad x\geq 0.
\end{equation}
The Mellin transforms of $\psi$ and $h$ are defined for $s\in \mathbb{C}$ as:
\begin{equation}\label{eq:mellin}
\Psi(s)=\int_0^{+\infty} x^{s-1} \psi(x)dx,\qquad \mbox{ and }\qquad H(s)=\int_0^1 x^{s-1}h(x)dx.
\end{equation}
Notice that for $\xi\in\R$,
\begin{align}\label{eq:relMstarPsi}
M^*(\xi)&:=\int_{-\infty}^{\infty}M(x)e^{ix\xi}dx=\int_{-\infty}^{\infty}e^xN(e^x)e^{ix\xi}dx
=\int_0^{+\infty}N(u)u^{i\xi}du=\Psi(i\xi+2).
\end{align}
So the Point (iii) of Proposition \ref{lem:regM} will be proved if $\Psi$ does not vanish on the line $2+i\R$. We first establish an equation satisfied by $\Psi$ (Step 3) and then conclude with the proof by adapting the results of Doumic et al. \cite{Doumic17}.\\

Recall \eqref{momentexpN}. Let us prove below that for any $s\in\mathbb C$ such that
\begin{equation}\label{Nlim}
\lim_{x\to 0}x^{s-1}N(x)=0, 
\end{equation}the function $\Psi$ satisfies:
\begin{equation}\label{eq:Psi}(1-s)\Psi(s)=\frac{2R}{\alpha}\big(H(s)-1\big)\Psi(s+1).
\end{equation}
Notice firstly that \eqref{eq:Psi} is in particular true for all $s$ such that $\Re(s)>1$ since \eqref{Nlim} is then satisfied. Secondly, let us add that \eqref{eq:Psi} is reminiscent of Equation (2.4) in Doumic et al. \cite{Doumic17}, with the difference that our equation is establish on the eigenvalue problem \eqref{eq:eigenvalue} with constant branching rate $R$, while Doumic and co-authors work with an evolution equation with no evolution of the cell sizes but power-law branching rate.\\

To prove \eqref{eq:Psi}, observe that $$N(x)=x\psi(x)\quad\mbox{ and }\quad   \pt_x N(x) =x\psi'(x)+\psi(x).$$
Therefore,  \eqref{eq:eigenvalue} can be expressed by using $\psi$ instead of $N$:
\begin{align*}
&\alpha \pt_x N(x) + 2R \, N(x) = 2R\int_0^\vc N(y)h\left(\dfrac xy\right) \dfrac{dy}{y} ,\quad x\ge 0 \\
\iff&x\psi'(x)+\Big(1+\frac{2R}{\alpha}x\Big)\psi(x)=\frac{2R}{\alpha}\int_x^{+\infty}y\psi(y)h\left(\dfrac xy\right) \dfrac{dy}{y} ,\quad x\ge 0.
\end{align*}
Multiplying each side of this equation by $x^{s-1}$ and integrating on $x\in [0,+\infty)$, we obtain:
\begin{align*}
& \int_0^{+\infty} x^s \psi'(x)dx+ \int_0^{+\infty} \Big(1+\frac{2R}{\alpha}x\Big)x^{s-1}\psi(x)dx=\frac{2R}{\alpha}\int_0^{+\infty} x^{s-1}\int_x^{+\infty}y\psi(y)h\left(\dfrac xy\right) \dfrac{dy}{y}\ dx\\
\iff & (1-s) \Psi(s)=\frac{2R}{\alpha}\big(H(s)-1\big)\Psi(s+1),
\end{align*}
by using an integration by parts for the left term and the Fubini theorem for the right term. This shows \eqref{eq:Psi}.\\


\noindent \textit{Step 3:} We now prove that for any $\xi\in \R$, and any $n\in \N$,
\begin{equation}\label{eq:relationMstar-Psi}
M^*(\xi)\not=0 \ \iff \ \Psi(i\xi+n+2)\not=0.
\end{equation}This is true for $n=0$, by \eqref{eq:relMstarPsi}. For $\xi\in\R$, applying \eqref{eq:Psi} with $s=2+i\xi$, for which \eqref{Nlim} is satisfied:
$$M^*(\xi)=\Psi(i\xi+2)=-\frac{2R}{\alpha(1+i\xi)}\big(H(i\xi+2)-1\big)\Psi(i\xi+3).$$
By induction, we have for any $n\in{\mathbb N}$ ($n$ will be chosen subsequently),
$$M^*(\xi)=\left(-\frac{2R}{\alpha}\right)^n\prod_{k=1}^n\frac{1}{k+i\xi}\Big(H(i\xi+k+1)-1\Big)\times \Psi(i\xi+n+2).$$
Since for any $\xi=\xi_1+i\xi_2 \in \mathbb{C}$ such that $\xi_1\geq 1$,
\begin{equation}\label{eq:H}
|H(i\xi_2+\xi_1+1)|=\Big|\int_0^1x^{i\xi_2+\xi_1}h(x)dx\Big|\leq \int_0^1xh(x)dx=\frac{1}{2},
\end{equation}by symmetry of $h$, then the product in the right hand side is non zero and this leads to \eqref{eq:relationMstar-Psi}.\\

As a consequence, to prove that $M^*$ admits no zero on the real line, it is sufficient to prove that $\Psi$ admits no zero on some line $n_0+i\R$ where $n_0$ is an integer larger than 2.\\

\noindent \textit{Step 4:} Following the computation in Doumic et al. \cite{Doumic17}, we can prove that:
\begin{equation}\label{eq:Doumic}
\forall s\in \mathbb{C} \mbox{ with }\Re(s)\geq 2,\   \ |\Psi(s)|\not=0.
\end{equation}
First, notice that \eqref{eq:H} implies that $H(s)-1\not= 0$ as soon as $\Re(s)\geq 2$. For such $s\in \mathbb{C}$, dividing by $H(s)-1$, we can reformulate \eqref{eq:Psi} as:
\begin{equation}\label{eq:Psi2}
\Psi(s+1)= \Phi(s)\Psi(s),\quad \mbox{ where }\quad \Phi(s)=\frac{\alpha}{2R}\times \frac{1-s}{H(s)-1}.
\end{equation}
Notice that our equation has much more regularities than the one studied in Doumic et al. since the application $\Phi$ is analytic on $\{s\in \mathbb{C},\ \Re(s)\geq 2\}$ and does not vanish on this half-plane. \\

Let us fix $s_0\in \R$ such that $s_0\geq 2$. For $s\in \mathbb{C}$ such that $s_0<\Re(s)<s_0+1$, we look for particular solutions of \eqref{eq:Psi2} of the form
\[\Psi(s)=\exp\big(P(\zeta(s))\big)\quad \mbox{ where }\quad \zeta(s)=e^{i 2\pi (s-s_0)}.\]
Then, following the steps of Doumic et al. \cite[Proposition 2]{Doumic17}, the function $P$ solves a Carleman equation, from which a solution of \eqref{eq:Psi2} can be obtain. There exist several solutions to \eqref{eq:Psi2}, and computing the inverse Mellin transforms of the latters, it appears that the solution corresponding to the Mellin transform of the solution of \eqref{eq:eigenvalue} is:
\begin{equation}\label{sol:Psi}
\Psi(s)=\exp\Big(-\int_{\Re(\sigma)=s_0} \log\big(\Phi(\sigma) \big)\Big[\frac{1}{1-e^{2i\pi(s-\sigma)}}-\frac{1}{1+e^{2i\pi(s_0-\sigma)}}\Big] d\sigma \Big),\end{equation}
for $s\in \{\Re(s)\in(s_0,s_0+1)\}$, where the chosen determination of the logarithm is $\log(z)=\log|z|+i \arg(z)$ with $\arg(z)\in [0,2\pi)$. That the right hand side of \eqref{sol:Psi} is well defined and does not vanish on $\{\Re(s)\in(s_0,s_0+1)\}$ follows closely the proofs in \cite[Lemma 3 and Section 6.3]{Doumic17} by careful study of the behavior of $\Phi(s)$ when the imaginary part of $s$ tend to $\pm\infty$.\\

\noindent \textit{Conclusion:} choose for example $s_0=\frac{5}{2}$ so that $s_0\geq 2$ and $n_0=3\in (s_0,s_0+1)$. The function given in \eqref{sol:Psi} is analytic, non-vanishing and coincide with the Mellin transform of $\psi$ on $\{s\in \mathbb{C}\ : \ \Re(s)\in (s_0,s_0+1)\}$. This implies that $\Psi$ admits no zero on $n_0+i\R$ and finishes the proof.
\end{proof}

\section{Proof of Propositions  \ref{prop:htilde} and \ref{prop:h-hat}}\label{app:risk-h}
\begin{proof}[Proof of Proposition \ref{prop:htilde}]
We have
\begin{align}
\|\hat{g}_{n,\bw} - g\|^2_2 &= \int_{\R_-} \big(\hat{g}_{n,\bw}(u) - g(u) \big)^2du = \int_{\R_-} \Big(e^u\hat h_{n,\bw}(e^u) - e^uh(e^u) \Big)^2du \nonumber\\
&= \int_0^1 \big(\hat h_{n,\bw}(x) - h(x)\big)^2xdx.\label{risk:hatg}
\end{align}
 Since $g(u)=e^u h(e^u)=e^u h(1-e^u)$ by the symmetry of $h$, we can show that
\begin{align*}
\ch{\tilde{g}_{n,\bw} - g}^2 
&= \int_0^1 \big(\hat h_{n,\bw}(x) - h(x)\big)^2(1-x)dx.
\end{align*}
Thus,
\begin{align}
\kv\left[\ch{\check{g}_{n,\bw} - g}^2 \right] &= \kv\left[\ch{\tau\hat{g}_{n,\bw} + (1-\tau)\tilde{g}_{n,\bw} - g}^2 \right]\nonumber\\
&= \frac 12\kv\big[\ch{\hat{g}_{n,\bw} - g}^2\big] + \frac 12\kv\big[\ch{\tilde{g}_{n,\bw} - g}^2\big] = \frac 12\kv\big[\|\hat{h}_{n,\bw} - h\|^2_2\big],\label{etape5}
\end{align}
since $\ch{\hat{g}_{n,\bw} - g}^2 + \ch{\tilde{g}_{n,\bw} - g}^2 = \|\hat{h}_{n,\bw} - h\|^2_2.$
Let us now compute $\kv\big[\ch{\tilde{g}_{n,\bw} - g}^2\big]$. Recall that $h=0$ on $\R\setminus (0,1)$, so  $g=0$ on $\R_+$. For $u<0$, we define the new variable $v\in \R_-^*$ such that $e^v=1-e^u$. We have
\begin{align*}
\tilde{g}_{n,\bw}(u)= & e^u \hat{h}_{n,\bw}(1-e^u)=e^u \hat{h}_{n,\bw}(e^v)= e^{u-v}\hat{g}_{n,\bw}(v)=\frac{e^u}{1-e^u}\hat{g}_{n,\bw}\big(\log(1-e^u)\big).
\end{align*}Similarly, we have that $g(u)=\frac{e^u}{1-e^u} g(\log(1-e^u))$ and thus
\begin{align*}
\kv\big[\ch{\tilde{g}_{n,\bw} - g}^2\big]= & \E\Big[ \int_{\R_-}\big(\tilde{g}_{n,\bw}(u)-g(u)\big)^2 du\Big] \\
= & \E\Big[\int_{\R_-} \Big(\frac{e^u}{1-e^u}\Big)^2 \Big(\hat{g}_{n,\bw}\big(\log(1-e^u)\big)-g\big(\log(1-e^u)\big)\Big)^2 du\Big]\\
= & \E\Big[\int_{\R_-} \Big(\frac{1-e^v}{e^v}\Big)\big(\hat{g}_{n,\bw}(v)-g(v)\big)^2 dv\Big].
\end{align*}As a consequence, the middle term in \eqref{etape5} is
\begin{align*}
\frac 12\kv\big[\ch{\hat{g}_{n,\bw} - g}^2\big] + \frac 12\kv\big[\ch{\tilde{g}_{n,\bw} - g}^2\big]
= & \E\Big[ \int_{\R_-} \frac{1}{2}\Big(1+ \frac{1-e^v}{e^v}\Big) \big(\hat{g}_{n,\bw}(v)-g(v)\big)^2 dv \Big]\\
=  & \E\Big[\int_{\R_-} \frac{e^{-v}}{2} \big(\hat{g}_{n,\bw}(v)-g(v)\big)^2 dv\Big].
\end{align*}
This concludes the proof.

\end{proof}

\begin{proof}[Proof of Proposition \ref{prop:h-hat}]
Remember \eqref{risk:hatg}. Then, since $h(x)=h(1-x)$,
\begin{multline*}
\int_0^1 \big(\hat h_{n,\bw}^{sym}(x) - h(x)\big)^2m(x)dx\\
\begin{aligned}
= & \frac{1}{4}\int_0^1\big(\hat h_{n,\bw}(x) - h(x)+\hat h_{n,\bw}(1-x) - h(1-x)\big)^2m(x)dx\\
\leq&\frac{1}{2}\int_0^1\big(\hat h_{n,\bw}(x) - h(x)\big)^2m(x)dx+\frac{1}{2}\int_0^1\big(\hat h_{n,\bw}(1-x) - h(1-x)\big)^2m(1-x)dx\\
= &\int_0^1\big(\hat h_{n,\bw}(x) - h(x)\big)^2m(x)dx\\
\leq&\int_0^1\big(\hat h_{n,\bw}(x) - h(x)\big)^2xdx = \ch{\hat{g}_{n,\bw} - g}^2.
\end{aligned}
\end{multline*}This concludes the proof.

\end{proof}
\section{Proof of Theorem \ref{prop:rate-of-convergence}}\label{app:main-theorems}

\begin{proofof}{Theorem  \ref{prop:rate-of-convergence}}
Let $g_\ell=K_\ell\star g$. We have
\begin{eqnarray*}
\| \hat g_{n,\ell}- g  \|_2 \leq  \|  g_\ell- g  \|_2 + \| \hat g_{n,\ell}- g_\ell  \|_2.
\end{eqnarray*}
The first term of the above r.h.s  inequality is a bias term whereas the second is a variance term.
 To control the variance term, we have by the Parseval's identity and by \eqref{eq:estigell}:
\begin{small}
\begin{align*}
 &\| \hat g_{n,\ell}- g_\ell  \|^2_2 =  \frac{1}{2\pi} \| \hat{ g}_{n,\ell}^*- g_\ell^*  \|^2_2 \\
 &= \frac{1}{2\pi} \intn \Bigg | K_\ell^*(\xi) \Big [\Big (  \frac {\alpha \widehat{D_n^*(\xi)}}{2R }\frac{\id_{\Omega_n(\xi)}}{\widehat{M_n^*(\xi)}} +  1  \Big ) -g^*(\xi)     \Big ]  \Bigg|^2 d\xi \\
 &=  \frac{1}{2\pi} \intn \Bigg | K_\ell^*(\xi) \Big [\Big (  \frac {\alpha \widehat{D_n^*(\xi)}}{2R }\frac{\id_{\Omega_n(\xi)}}{\widehat{M_n^*(\xi)}} - \frac {\alpha \widehat{D_n^*(\xi)}}{2R {M^*(\xi)}} + \frac {\alpha \widehat{D_n^*(\xi)}}{2R {M^*(\xi)}}   +  1  \Big ) -g^*(\xi)     \Big ]  \Bigg|^2 d\xi \\
 &=  \frac{1}{2\pi} \intn  \Bigg | \frac{\alpha}{2R}K_\ell^*(\xi)  \widehat{D_n^*(\xi)} \Big (\frac{\id_{\Omega_n(\xi)}}{\widehat{M_n^*(\xi)}}-\frac{1}{{M^*(\xi)}} \Big )   +  K_\ell^*(\xi)\Big  ( \frac {\alpha \widehat{D_n^*(\xi)}}{2R {M^*(\xi)}}   + 1  -g^*(\xi) \Big )  \Bigg |^2 d\xi\\
 &\leq C \intn  \Bigg | K_\ell^*(\xi)  \widehat{D_n^*(\xi)} \Big (\frac{\id_{\Omega_n(\xi)}}{\widehat{M_n^*(\xi)}}-\frac{1}{{M^*(\xi)}} \Big )   \Bigg |^2 d\xi + C \intn |K_\ell^*(\xi)^2| \Bigg   | \frac {\alpha \widehat{D_n^*(\xi)}}{2R {M^*(\xi)}}   +  1  -g^*(\xi)   \Bigg |^2 d\xi \\
 &:= \text{I} + \text{II }.
\end{align*}
\end{small}
We deal with variance of complex variables. Note that for a complex variable, say $Z$, by distinguishing real and imaginary parts one gets that $$\ps(Z) := \kv [|Z-\kv(Z)|^2] = \kv [ |Z|^2 ] -|\kv [Z]|^2 \leq  \kv [ |Z|^2]. $$

\noindent For  the term II, because
\begin{eqnarray*}
\kv \left (  K_\ell^*(\xi)\left  ( \frac {\alpha \widehat{D_n^*(\xi)}}{2R {M^*(\xi)}}   +  1 \right ) \right ) &=&   K_\ell^*(\xi)\left  ( \frac {\alpha {D^*(\xi)}}{2R {M^*(\xi)}}   +  1  \right )  \\
&=&  K_\ell^*(\xi) g^*(\xi),
\end{eqnarray*}
we have
\begin{eqnarray*}
 \kv[\text{II}]&=&  C\intn \ps \left( K_\ell^*(\xi)    \left (\frac {\alpha \widehat{D_n^*(\xi)}}{2R {M^*(\xi)}}   +  1   \right)  \right)d\xi \\
 &\leq & C   \intn \ps \left( K_\ell^*(\xi)    \frac {\widehat{D_n^*(\xi)}}{{M^*(\xi)}}       \right)d\xi \\
 &\leq & C\intn \left|\frac{K_\bw^*(\xi)}{M^*(\xi)} \right|^2\ps\left(\frac{(-\ic\xi)}{n}\sum_{j=1}^n e^{(\ic\xi-1)U_j} \right)d\xi \\
 &\leq & \frac C n  \intn \left|\frac{K_\bw^*(\xi)\xi}{M^*(\xi)} \right|^2\ps\left(e^{(\ic\xi-1)U_1} \right)d\xi \\
 &\leq &  \frac C n  \intn \left|\frac{K_\bw^*(\xi)\xi}{M^*(\xi)} \right|^2\kv\left[\left|e^{(\ic\xi-1)U_1} \right|^2 \right]d\xi \leq \frac C n  \intn \left|\frac{K_\bw^*(\xi)\xi}{M^*(\xi)} \right|^2\kv\left[e^{-2U_1}\right]d\xi  \\
 &\leq &  \frac C n  \left  \|  \frac{ K_\ell^*(\xi) \xi  }{M^*(\xi) } \right \|^2_2 ,
\end{eqnarray*}
since $\kv[e^{-2U_1}]=\int_0^{+\infty}x^{-2}N(x)dx<+\vc$ thanks to Proposition \ref{lem:regM}.

\noindent We now set
\begin{equation}
\triangle(\xi) := \frac{\id_{\Omega_n(\xi)}}{\widehat{M_n^*(\xi)}} - \frac{1}{M^*(\xi)}.\label{eq:trianglexi}
\end{equation}

\noindent Then we get
\begin{align*}
\kv[\,\text{I}\,] &\leq  C\intn \kv\left[\left|K_\bw^*(\xi)\widehat{D_n^*(\xi)}\triangle(\xi)\right|^2 \right]d\xi \leq  C\intn \big|K_\bw^*(\xi)\big|^2\kv\left[\big|\widehat{D_n^*(\xi)} \big|^2\big|\triangle(\xi)\big|^2 \right]d\xi \\
&\le C\intn \big|K_\bw^*(\xi)\big|^2\kv\left[\big|\widehat{D_n^*(\xi)} - \kv\big[\widehat{D_n^*(\xi)} \big] \big|^2 \big|\triangle(\xi)\big|^2\right]d\xi  \\
&\hspace{3cm} + C\intn \big|K_\bw^*(\xi)\big|^2\Big|\kv\big[\widehat{D_n^*(\xi)} \big]\Big|^2\kv\big[|\triangle(\xi)|^2\big]d\xi \\
&:= \text{III} + \text{IV}.
\end{align*}
To control the term IV, we need the following lemma whose proof is postponed in Appendix \ref{app:technicallemmas}.

\begin{lemma}\label{lem:control-delta}
There exists a positive constant $C_p$ such that
\begin{equation}\label{eq:control-delta}
\kv\left[|\triangle(\xi)|^{2p} \right] \le C_p\min\left\{\frac{1}{|M^*(\xi)|^{2p}},\frac{n^{-p}}{|M^*(\xi)|^{4p}} \right\}\quad\text{ for } p = 1,2.
\end{equation}
\end{lemma}
Since $\widehat{ D_n^*}$ is an unbiased estimator of $D^*$  using Lemma \ref{lem:control-delta} with $p=1$ we get
\[
\text{IV} \leq  C\intn \big|K_\bw^*(\xi)\big|^2\big| D^*(\xi)\big|^2\frac{n^{-1}}{|M^*(\xi)|^4}d\xi = \frac{C}{n}\intn \bigg|\frac{K_\bw^*(\xi)}{M^*(\xi)}\bigg|^2\bigg|\frac{D^*(\xi)}{M^*(\xi)}\bigg|^2d\xi.
\]
Moreover, from equation (\ref{eq:FT-g}) and using that $g^*$ is the Fourier transform of a density function, we get
\[
\left | \frac{D^*({\xi})}{M^*(\xi)} \right |\leq \frac{2R}{\alpha}\left (|g^*(\xi)| + 1\right) \leq \frac{4R}{\alpha}.
\]
Thus we obtain
\begin{align*}
\text{IV} & \leq \frac{C}{n} \intn \bigg|\frac{K_\bw^*(\xi)}{M^*(\xi)}\bigg|^2 d\xi = \frac{C}{n} \Big\|\frac{K_\bw^*(\xi)}{M^*(\xi)}\Big\|_2^2.
\end{align*}
For the term III, we have by applying Cauchy-Schwarz's inequality and by Lemma \ref{lem:control-delta}:
\begin{align*}
\text{III} &\le C\intn \big|K_\bw^*(\xi)\big|^2\left(\kv\Big[\big|\widehat{D_n^*(\xi)} - \kv\big[\widehat{D_n^*(\xi)} \big] \big|^4\Big] \right)^{1/2}\left(\kv\Big[\big|\triangle(\xi)\big|^4 \Big] \right)^{1/2}d\xi \\
&\le C\intn \big|K_\bw^*(\xi)\xi\big|^2\left(\kv\Big[\big|\frac{1}{n}\sum_{j=1}^ne^{(\ic\xi-1)U_j} - \kv\big[e^{(\ic\xi-1)U_1} \big]\big|^4 \Big] \right)^{1/2}\\
&\hspace{5.3cm}\times \min\left\{\frac{1}{|M^*(\xi)|^{4}},\frac{n^{-2}}{|M^*(\xi)|^{8}} \right\}^{1/2}d\xi \\
&\le C\intn \frac{\big|K_\bw^*(\xi)\xi\big|^2}{|M^*(\xi)|^2}\left(\kv\left[\Big|\frac 1n\sum_{j=1}^n Z_j(\xi) \Big|^4 \right] \right)^{1/2}d\xi,
\end{align*}
where $Z_j(\xi) = e^{(\ic\xi-1)U_j} - \kv\big[e^{(\ic\xi-1)U_1} \big]$. Since $Z_1(\xi), \ldots, Z_n(\xi)$ are independent centered variables with

\begin{align}
\kv[|Z_1(\xi)|^4] &= \kv\left[\Big|e^{(\ic\xi-1)U_1} - \kv\big[e^{(\ic\xi-1)U_1} \big] \Big|^4 \right]
\le \kv\left[\left( \big|e^{(\ic\xi-1)U_1}\big| + \big| \kv\big[e^{(\ic\xi-1)U_1} \big] \big|\right)^4 \right] \notag \\
&\le 2^3 \left( \kv\big[\big|e^{(\ic\xi-1)U_1}\big|^4\big] + \big|\kv\big[ e^{(\ic\xi-1)U_1} \big]\big|^4 \right) \notag \\
&\le 16 \kv\big[e^{-4U_1}\big] = 16 \int_0^{+\infty}x^{-4}N(x)dx < +\vc \label{eq:bound-Zj}
\end{align}
by Proposition \ref{lem:regM}, applying Rosenthal inequality to real and imaginary parts of complex variables $Z_j$'s, we get
\[
\kv\left[\Big|\frac{1}{n}\sum_{j=1}^n Z_j(\xi) \Big|^4 \right] \le C n^{-4}\left(n\kv[|Z_1(\xi)|^4] + \big(n\kv[|Z_1(\xi)|^2]\big)^2 \right) \le Cn^{-2}.
\]
Hence
\begin{align*}
\text{III} &\le \frac{C}{n}\intn \frac{\big|K_\bw^*(\xi)\xi\big|^2}{|M^*(\xi)|^2}d\xi = \frac{C}{n}\Big\|\frac{K^*_\bw(\xi)\xi}{M^*(\xi)}\Big\|^2.
\end{align*}
Finally, we obtain
\[
\kv\big[\ch{\hat{g}_{n,\bw} - g}^2 \big]\le  \ch{K_\bw\star g - g}^2 +  \frac{C}{n}\left(\Big\| \frac{K^*_\bw(\xi)\xi}{M^*(\xi)}\Big\|_2^2 + \Big\| \frac{K^*_\bw(\xi)}{M^*(\xi)}\Big\|_2^2 \right).
\]
This ends the proof of Theorem \ref{prop:rate-of-convergence}.

\end{proofof}


\begin{proofof}{Proposition \ref{vitesse:g}}
First let us find an upper bound for the bias term $ \ch{K_\bw\star g - g}^2$. We assume that $g^* $ is integrable and $g$ belongs to the Sobolev class $S(\beta,L)$. 
Proposition 1 of Comte and Lacour \cite{ComteLacour} yields
$$
 \ch{K_\bw\star g - g}^2 \leq L^2 \ell^{2\beta}.
$$
Now we shall consider the variance term. We have
$$
 \Big\| \frac{K^*_\bw(\xi)\xi}{M^*(\xi)}\Big\|_2^2=  \int_{-\bw^{-1}}^{\bw^{-1}}\frac{\xi^2}{|M^*(\xi)|^2}d\xi=O(\bw^{-(3+2([\beta]\wedge (\nu_0+3)))}).$$
Performing the usual trade-off between the bias and the variance terms, we get the following choice of the bandwith $\ell$, for a fixed $\alpha >0$ :
\[
\ell = \alpha n^{-\frac{1}{2\beta + 2([\beta]\wedge (\nu_0+3)) + 3}},
\]
which  concludes the proof.
\end{proofof}
\section{Proof of technical lemma}\label{app:technicallemmas}
\begin{proofof}{Lemma \ref{lem:control-delta}}
This proof is inspired by the proof of Neumann \cite{Neumann97}. We will prove the result with $p=1$. For $p=2$, the proof is similar.

We split the proof in two cases: $|M^*(\xi)|<2n^{-1/2}$ and $|M^*(\xi)|\ge 2n^{-1/2}$. Recall that $\Omega_n(\xi) = \left\{|\widehat{M_n^*(\xi)}|\ge n^{-1/2} \right\}$ and $\kv\big[\widehat{M_n^*(\xi)}\big] = \kv\big[e^{\ic\xi U_1} \big] = M^*(\xi)$, we have:
\begin{align}
\kv\left[|\triangle(\xi)|^2 \right] &= \kv\left[\Big|\frac{\id_{\Omega_n(\xi)}}{\widehat{M_n^*(\xi)}} - \frac{1}{M^*(\xi)]} \Big|^2\right] = \kv\left[\Bigg|\frac{\id_{\Omega_n(\xi)}}{\widehat{M_n^*(\xi)}} - \Bigg(\frac{\id_{\Omega_n(\xi)}}{M^*(\xi)} + \frac{\id_{\Omega_n^c(\xi)}}{M^*(\xi)}\Bigg)\Bigg|^2\right] \notag\\
&= \frac{\xs(\Omega_n^c(\xi))}{|M^*(\xi)|^2} + \kv\left[\id_{\Omega_n(\xi)}\frac{|\widehat{M_n^*(\xi)} - M^*(\xi)|^2}{|\widehat{M_n^*(\xi)}|^2|M^*(\xi)|^2} \right]. \label{eq:proof-lem1-eq1}
\end{align}

\noindent \textit{$i)$  If $|M^*(\xi)|<2n^{-1/2}$:}
\begin{align*}
\kv\left[|\triangle(\xi)|^2 \right] \le \frac{1}{|M^*(\xi)|^2} + \frac{\kv\left[|\widehat{M_n^*(\xi)} - M^*(\xi)|^2 \right]n}{|M^*(\xi)|^2}.
\end{align*}
But
\begin{align*}
\kv\left[\left|\widehat{M_n^*(\xi)} - M^*(\xi)\right|^2 \right] &= \ps\left[\widehat{M_n^*(\xi)} \right] = \ps\left[\frac{1}{n}\sum_{j=1}^n e^{\ic\xi U_j} \right] \\
&= \frac 1n \ps\left(e^{\ic\xi U_1} \right) \le \frac{1}{n}\kv\left[|e^{\ic\xi U_1}|^2 \right] = \frac 1n.
\end{align*}
Hence we obtain
\begin{equation}\label{eq:upper-bound-E_delta}
\kv\left[|\triangle(\xi)|^2 \right] \le \frac{C}{|M^*(\xi)|^2} \leq C\min\left\{\frac{1}{|M^*(\xi)|^{2}},\frac{n^{-1}}{|M^*(\xi)|^{4}} \right \},
\end{equation}
 since $|M^*(\xi)|<2n^{-1/2}$.\\

\noindent \textit{$ii)$ If $|M^*(\xi)|\ge 2n^{-1/2}$:}\\[6pt]
We first control the probability $\xs(\Omega_n^c(\xi))$,
\begin{align}
\xs\left(\Omega_n^c(\xi) \right) &= \xs\left(|\widehat{M_n^*(\xi)}| < n^{-1/2}\right) =
\xs\left(|\widehat{M_n^*(\xi)}|< |M^*(\xi)| - |M^*(\xi)| + n^{-1/2}\right) \notag \\
&\le \xs\left(| \widehat{M_n^*(\xi)} - M^*(\xi)| > |M^*(\xi) | - n^{-1/2} \right)\notag \\
&\le  \xs\left(| \widehat{M_n^*(\xi)} - M^*(\xi)| > |M^*(\xi)|/2 \right). \label{eq:proof-lem1-eq2}
\end{align}
Let $T_j(\xi) = e^{\ic\xi U_j} - \kv\big[e^{\ic\xi U_1} \big]$, then
\[
\mhxi - \mxi = \frac{1}{n}\sum_{j=1}^n e^{\ic\xi U_j} - \kv\big[e^{\ic\xi U_1} \big] = \frac{1}{n}\sum_{j=1}^n T_j(\xi).
\]
We have
\[
|T_1(\xi)| = \big|e^{\ic\xi U_j} - \kv\big[e^{\ic\xi U_1} \big] \big| \le \big|e^{\ic\xi U_j} \big| + \big|\kv\big[e^{\ic\xi U_1} \big] \big| \le 2,
\]
and
\[
\ps\big(T_1(\xi)\big) \le \kv\big[|e^{\ic\xi U_1}|^2 \big] = 1.
\]
Since $|M^*(\xi)|\le 1$ for all $\xi\in\rhoa$ because of $M$ is a density function, we get by Bernstein inequality (cf. for instance Comte and Lacour \cite[Lemma 2, p.20]{ComteLacour})
\begin{align}
\xs\left(| \widehat{M_n^*(\xi)} - M^*(\xi)| > |M^*(\xi)|/2 \right) &\le 2\max\left\{\exp\Big(-\frac{n|\mxi|^2}{16} \Big), \exp\Big(-\frac{n|\mxi|}{16} \Big)  \right\} \notag \\
&\le 2\exp\Big(-\frac{n|\mxi|^2}{16} \Big) \notag \\
&\le C\frac{n^{-1}}{|\mxi|^2}.
\end{align}

\noindent We also have that
\begin{align}
\frac{1}{|\widehat{M_n^*(\xi)}|^2} &= \frac{|\mxi|^2}{|\mhxi|^2|\mxi|^2} = \frac{|\mhxi - \big(\mhxi - \mxi\big)|^2}{|\mhxi|^2|\mxi|^2} \notag\\
&\le 2\left\{\frac{1}{|\mxi|^2} + \frac{|\mhxi - \mxi|^2}{|\mhxi|^2|\mxi|^2}\right\}. \label{eq:proof-lem1-eq3}
\end{align}
Thus, from \eqref{eq:proof-lem1-eq1}, \eqref{eq:proof-lem1-eq2} and \eqref{eq:proof-lem1-eq3} we have:
\begin{align}
\kv\left[|\triangle(\xi)|^2 \right] &\le C\left\{\frac{n^{-1}}{|\mxi|^4} + \kv\left[\id_{\Omega_n(\xi)}\frac{|\mhxi - \mxi|^2}{|\mhxi|^2||\mxi|^2} \right]\right\} \notag \\
&\le C\left\{\frac{n^{-1}}{|\mxi|^4} + \frac{\kv\big[|\mhxi - \mxi|^2 \big]}{|\mxi|^4}  + \frac{\kv\big[|\mhxi - \mxi|^4 \big]n}{|\mxi|^4}\right\}.\label{eq:proof-lem1-eq5}
\end{align}
To find an upper bound for $\kv\big[|\mhxi - \mxi|^4 \big]$, recall that $T_j(\xi) = e^{\ic\xi U_j} - \kv\big[e^{\ic\xi U_1} \big]$. By similar calculations as obtained \eqref{eq:bound-Zj}, we have $\kv[|T_1(\xi)|^4] < +\vc$ as $|T_1(\xi)|\le 2$. Thus we get by Rosenthal's inequality applied to real and imaginary parts of the sequence of independent centered variables $T_1(\xi), \ldots, T_n(\xi)$:
\begin{align*}
\kv\big[|\mhxi - \mxi|^4 \big] &= \kv\left[\Big|\frac{1}{n}\sum_{j=1}^n T_j(\xi) \Big|^4 \right] \\
&\le C n^{-4}\left(n\kv[|T_1(\xi)|^4] + \big(n\kv[|T_1(\xi)|^2]\big)^2 \right) \le Cn^{-2}.
\end{align*}
Thus, from \eqref{eq:proof-lem1-eq2} and \eqref{eq:proof-lem1-eq5} we get
\[
\kv\left[|\triangle(\xi)|^2 \right] \le C\frac{n^{-1}}{|\mxi|^4}.
\]
Furthermore
$$  \frac{1}{|M^*(\xi)|^2} \geq \frac{n^{-1}}{|M^*(\xi)|^4 }, $$ since $|M^*(\xi)| > 2n^{-1/2}$. Hence
$$ \kv\left[|\triangle(\xi)|^2 \right]  \leq C\min\left\{\frac{1}{|M^*(\xi)|^{2}},\frac{n^{-1}}{|M^*(\xi)|^{4}} \right \}.  $$
Combining the two cases, we obtain
\[
\kv\left[|\triangle(\xi)|^2 \right] \le C\min\left\{\frac{1}{|M^*(\xi)|^{2}},\frac{n^{-1}}{|M^*(\xi)|^{4}} \right\}.
\]
This ends the proof of Lemma \ref{lem:control-delta}.

\end{proofof}

%

\noindent\textbf{Acknowledgements:} The authors thank Sylvain Arlot, Thibault Bourgeron and Matthieu Lerasle for helpful discussions. Van H\`a Hoang and Viet Chi Tran have been supported by  the Chair ``Mod\'elisation Math\'ematique et Biodiversit\'e" of Veolia Environnement-Ecole Polytechnique-Museum National d'Histoire Naturelle-Fondation X, and also acknowledge support from Labex CEMPI (ANR-11-LABX-0007-01) and ANR project Cadence (ANR-16-CE32-0007). 

{\footnotesize
\bibliographystyle{plain}	

}
\end{document}